\documentclass[10pt]{article}
\usepackage[utf8]{inputenc} 
\usepackage[T1]{fontenc}    
\usepackage[hidelinks]{hyperref}       
\usepackage{url}            
\usepackage{booktabs}       
\usepackage{amsfonts}       
\usepackage{nicefrac}       
\usepackage{microtype}      
\usepackage[table]{xcolor}
\usepackage{natbib}
\usepackage{fullpage}
\usepackage{multirow}
\usepackage{soul}           

\newcommand{\Proj}{\textup{Proj}}

\newcommand{\calA}{\mathcal{A}}

\newcommand{\calS}{\mathcal{S}}
\newcommand{\calT}{\mathcal{T}}
\newcommand{\calF}{\mathcal{F}}

\newcommand{\calG}{\mathcal{G}}

\newcommand{\E}{\mathbb{E}}
\renewcommand{\P}{\mathbb{P}}

\newcommand{\ls}{\left}
\newcommand{\rs}{\right}

\newcommand{\inner}[2]{\ls\langle #1, \; #2 \rs\rangle}

\newcommand\numberthis{\addtocounter{equation}{1}\tag{\theequation}}
\makeatletter

\newcommand{\Rmnum}[1]{\textup{\expandafter\@slowromancap\romannumeral #1@}}
\makeatother

\usepackage[table]{xcolor}
\usepackage{tabstackengine}
\usepackage{appendix}
\usepackage{titletoc}
\usepackage{titlesec}
\usepackage{enumerate}
\usepackage{amsmath}
\usepackage{amssymb}
\usepackage{mathtools}
\usepackage{amsthm}
\usepackage{algorithm}
\usepackage{algorithmic}
\usepackage{graphicx}
\usepackage{subfigure}
\usepackage{makecell}
\usepackage{array}
\usepackage{threeparttable}
\usepackage{setspace}
\usepackage{dsfont}
\usepackage{multicol}
\usepackage{authblk}
\usepackage{lipsum}
\allowdisplaybreaks

\theoremstyle{plain}
\newtheorem{theorem}{Theorem}[section]
\newtheorem{proposition}[theorem]{Proposition}
\newtheorem{lemma}[theorem]{Lemma}
\newtheorem{corollary}[theorem]{Corollary}
\newtheorem{remark}{Remark}[section]
\theoremstyle{definition}

\newtheorem{assumption}[theorem]{Assumption}

\title{Policy Mirror Descent with Temporal Difference Learning: Sample Complexity under Online Markov Data}
\author{Wenye Li, Hongxu Chen, Jiacai Liu, Ke Wei}
\affil{School of Data Science, Fudan University, Shanghai, China.}

\begin{document}
\maketitle
\begin{abstract}
    This paper studies the policy mirror descent (PMD) method, which is a general policy optimization framework in reinforcement learning and can cover a wide range of policy gradient methods by specifying difference mirror maps. Existing sample complexity analysis for policy mirror descent either focuses on the generative sampling  model, or the Markovian sampling model but with the action values being explicitly approximated to certain pre-specified accuracy. In contrast, we consider the sample complexity of policy mirror descent with temporal difference (TD) learning under the Markovian sampling model. Two algorithms called Expected TD-PMD and Approximate TD-PMD have been presented, which are off-policy and mixed policy algorithms respectively. Under a small enough constant policy update step size, the $\tilde{O}(\varepsilon^{-2})$ (a logarithm factor about $\varepsilon$ is hidden in $\tilde{O}(\cdot)$) sample complexity can be established for them to achieve average-time $\varepsilon$-optimality. The sample complexity is further improved to $O(\varepsilon^{-2})$ (without the hidden logarithm factor) to achieve the last-iterate $\varepsilon$-optimality based on adaptive policy update step sizes.
\end{abstract}
\section{Introduction}
The goal of {Reinforcement learning} (RL) is to solve  a sequential decision problem via interacting with the environment~\citep{suttonRL}, and it plays an important role in various application fields including robotics~\citep{robot1,robot2,robot3,LZ-robot}, structure explorations~\citep{ref-AlphaFold,ref-AlphaTensor}, and the post-training in Large Language Models (LLMs)~\citep{Zhang2025ASO}. {Markov decision process (MDP)} is a fundamental model for RL. Let $\Delta(\mathcal{X})$ denotes the probabilistic simplex over set $\mathcal{X}$. An MDP can be described as a tuple $(\calS, \calA, P, r, \gamma)$, where $\calS$ is the state space, $\calA$ is the action space, $P: \calS\times \calA \to \Delta(\calS)$ is the state transition model, $r: \calS \times \calA \to [0,1]$ is the reward function, and $\gamma \in [0,1)$ is the discounted factor. To interact with the environment, the agent is equipped with a policy $\pi: \calS \to \Delta(\calA)$. At each time, the agent chooses an action $a\in\calA$ at state $s\in\calS$ with probability $\pi(a|s)$, then it transfers to a new state $s^\prime$ with the probability $P(s^\prime|s,a)$ with a reward $r(s,a)$. By repeating the procedure above, the agent obtains a trajectory $\{ (s_t, a_t, r_t) \}_{t\geq 0}$, where $s_t\sim P(\cdot|s_{t-1},a_{t-1})$, $a_t\sim\pi(\cdot|s_t)$, and $r_t = r(s_t, a_t)$. The state value  induced by policy $\pi$ at  state $s\in\calS$ is defined as the expectation of the discounted cumulative rewards over trajectories,
\begin{align*}
    V^\pi(s) := {\E} \ls[ \sum_{t=0}^{\infty} \gamma^t r(s_t, a_t) \, \bigg | \, s_0=s,\pi \rs].
\end{align*}
 Similarly,  the action value function induced by $\pi$ is defined as
\begin{align*}
    Q^\pi(s,a) := {\E} \ls[ \sum_{t=0}^{\infty} \gamma^t r(s_t, a_t) \, \bigg | \, s_0=s, \; a_0 = a,\pi \rs].
\end{align*}
The goal of RL is to seek the optimal policy which maximizes the state values or action values, 
\begin{align}
    \max_{\pi \in \Pi} \; V^{\pi}(\mu) := \E_{s\sim\mu}[V^\pi(s)] \quad \mbox{or} \quad \max_{\pi\in\Pi} \; Q^\pi(\rho) := \E_{(s,a)\sim\rho} \ls[ Q^\pi(s,a) \rs],
    \label{eq:RL-goal}
\end{align}
where $\mu \in \Delta(\calS)$, $\rho \in \Delta(\calS\times\calA)$ are some distribution and $\Pi$ denotes the set of all admissible policies. 

{Policy mirror descent} (PMD) is an instance of policy optimization methods for solving~\eqref{eq:RL-goal}. Let $h$ be a function that is proper, closed, strictly convex, essentially smooth, and satisfying $\mbox{rint dom}(h) \cap \Delta(\calA) \neq \emptyset$ (see for example \citet{beck2017first} for related definitions). The policy update of PMD is given by~\citep{Xiao_2022,Lan_2021}
\begin{align*}
    \forall\, s\in\calS: \quad \pi^+(\cdot|s) = \underset{p\in\Delta(\calA)}{\arg\max} \ls\{ \eta  \ls\langle p, \, Q^{\pi}(s,\cdot) \rs\rangle - D_h(p \, \| \, \pi(\cdot|s)) \rs\},
\end{align*}
where $\eta > 0$ is the policy update step size and $D_h(p \, \|\, q) = h(p) - h(q) - \ls\langle \nabla h(q), \, p-q \rs\rangle$ is the Bregman divergence induced by $h$. The PMD update fits into the general mirror descent (MD) framework in optimization \citet{Xiao_2022} and covers a wide range of policy gradient methods by specifying different $h$. Thus, it has received  intensive investigations recently. 

\subsection{Related works}
\paragraph{PMD} Finite-time global convergence  of exact PMD (i.e., when action value $Q^{\pi_k}$ can be computed exactly) have been studied in for example~\citet{Xiao_2022,Lan_2021,Johnson_Pike-Burke_Rebeschini_2023}. Specifically, it is shown in~\citet{Lan_2021,Xiao_2022} that exact PMD enjoys a $O({1}/{K})$ dimension-free global sublinear convergence with any constant step size. Moreover, the global linear convergence can be obtained with  geometrically increasing step sizes~\citep{Xiao_2022}. In~\citet{Johnson_Pike-Burke_Rebeschini_2023}, the $\gamma$-rate global linear convergence is established under adaptive step sizes. Another line of research considers the regularized setting, i.e., when the object value function (equation~\eqref{eq:RL-goal}) is equipped with a regularized term. Under such setting, PMD can achieve global linear convergence with a large constant step size~\citep{Lan_2021}. The global linear convergence is then extended to the scenario for any constant step sizes~\citep{Zhan_Cen_Huang_Chen_Lee_Chi_2021}. Furthermore, PMD enjoys the local super-linear convergence and the policy convergence under geometrically increasing the step sizes if the coefficient of regularization term decreases properly~\citep{Li_Zhao_Lan_2022}. There are also works investigating the convergence of PMD with function approximation~\citep{yuan2023general}, variance reduction PMD~\citep{Yang2019PolicyOW,huang2022bregman}, and the acceleration of PMD~\citep{chelu2024functional}.
There are also other variants  of PMD  where $Q^{\pi_k}$ is replaced by different Q-values. For instance, $h$-PMD~\citep{protopapas2024policy} applies an $h$-step lookahead of optimal Bellman operation over $Q^{\pi_k}$ and achieves a $\gamma^h$-rate global linear convergence under adaptive step sizes. 
A TD-PMD algorithm (also known as MD-MPI Type I in~\citet{Geist_Scherrer_Pietquin_2019}) is considered in~\citet{liu2025tdpmd}, which replaces exact action value $Q^{\pi_k}$ with a temporal difference learning estimation. Despite the inaccuracy of the action value, it is shown in~\citet{liu2025tdpmd} that TD-PMD enjoys similar convergence as vanilla PMD, including the sublinear convergence for constant step sizes and linear convergence for adaptive step sizes. 

Most of existing sample complexity analysis of PMD is based on the generative model, where one can generate arbitrary trajectories starting from any state-action pair $(s,a) \in \calS \times \calA$. The sample complexity for unregularized PMD  to find a high-probability last-iterate $\varepsilon$-optimal policy  under the generative model is $\tilde O((1-\gamma)^{-8}\varepsilon^{-2})$~\citep{Xiao_2022,Johnson_Pike-Burke_Rebeschini_2023}, and to find an average-time $\varepsilon$-optimal policy is $\tilde O((1-\gamma)^{-6}\varepsilon^{-2})$ in~\citet{Lan_2021}. TD-PMD \citep{liu2025tdpmd} and $h$-PMD \citep{protopapas2024policy} improve the last-iterate sample complexity of PMD to $\tilde O((1-\gamma)^{-7}\varepsilon^{-2})$. For regularized PMD, the $\tilde O(\varepsilon^{-1})$ sample complexity is obtained in \citet{Lan_2021}. By carefully decreasing the coefficient of the regularization term, the $\tilde O(\varepsilon^{-2})$ sample complexity for obtaining the last-iterate  $\varepsilon$-optimal solution of the unregularized MPD is established in~\citet{Lan_2021, Li_Zhao_Lan_2022}\footnote{More concretely, the sample complexity in~\citet{Li_Zhao_Lan_2022} is $\tilde O(\varepsilon_0^{-2})$ with probability at least $O(1-\varepsilon_0^{1/3})$ for any $\varepsilon < \varepsilon_0$, where $\varepsilon_0$ is a sufficiently small constant.}. Besides the generative model, there are several works investigating the sample complexity of PMD under the Markovian sampling model~\citep{Lan_2021,li2024stochastic,li2025policy}, i.e., sampling a single and complete trajectory $\{ (s_t, a_t) \}_{t\geq 0}$ via the interaction with environment. In~\citet{Lan_2021,li2024stochastic}\footnote{\citet{li2024stochastic} investigates the average-reward MDP case.}, conditional temporal difference learning (CTD) is adopted to  evaluate the action value of current policy. More precisely, CTD  takes samples every $m$ steps based on the current policy, and then  solves the Bellman equation in a stochastic manner to obtain an accurate approximation of $Q^{\pi_k}$. In \citet{li2025policy}, it adopts an alternative approach based on Monte Carlo estimation (MC) to approximate $Q^{\pi_k}$ from the Markov data. The sample complexity for unregularized PMD is $\tilde O(\varepsilon^{-2})$ and  is $\tilde O(\varepsilon^{-1})$ for regularized PMD for  both CTD~\citep{Lan_2021,li2024stochastic} and MC learning~\citep{li2025policy}. 


\paragraph{Natural policy gradient}  When $h$ is specified to the negative entropy, i.e., $h(p) = -\sum_{a\in\calA} p(a) \log p(a)$, PMD reduces to NPG under the softmax policy parameterization (softmax NPG). The $O({1}/{\sqrt{K}})$ global sublinear convergence of softmax NPG is established in~\citet{Shani_Efroni_Mannor_2019}, and the rate is improved to $O({1}/{K})$ in~\citet{Agarwal_Kakade_Lee_Mahajan_2019} with arbitrary constant step size. It is worth remarking that the $O({1}/{K})$ sublinear convergence of exact PMD in~\citet{Xiao_2022} is indeed an extension of that for softmax NPG in~\citet{Agarwal_Kakade_Lee_Mahajan_2019}. Though softmax NPG can be covered by PMD, specific results can be obtained for softmax NPG by leveraging its explicit policy update formula. For instance, by characterizing the policy over non-optimal actions, the local $\exp(-\eta \Delta(1-1/\lambda) )$ linear convergence rate for softmax NPG is established in~\citet{Khodadadian_Jhunjhunwala_Varma_Maguluri_2021}, where $\lambda > 1$ is the parameter correlated to the local phase and $\Delta > 0$ is some MDP-dependent constant. Furthermore, the global linear convergence is achieved by applying  adaptive step sizes~\citep{Khodadadian_Jhunjhunwala_Varma_Maguluri_2021}. Another linear convergence analysis can be found in~\citet{Bhandari_Russo_2021}, where the relation of softmax NPG and policy iteration algorithm (PI, see for instance~\citet{suttonRL}) is discussed and the $\frac{1+\gamma}{2}$-rate linear convergence is established with  adaptive step sizes. The global linear convergence is then extended to the arbitrary constant step size $\eta > 0$ case in~\citet{pg-liu} by computing the state improvement $\sum_{a\in\calA} \pi_{k+1}(a|s) [Q^{\pi_k}(s,a) - V^{\pi_k}(s)]$, though the convergence rate convergence rate is dynamic and cannot be expressed in terms of the problem parameters. Recently, it is shown in~\citet{li2025phi} that softmax NPG indeed converges in the policy domain and enjoys an exact asymptotic convergence rate of $\exp(-\eta\Delta)$. The analysis of regularized softmax NPG can be found in~\citet{Shani_Efroni_Mannor_2019, Cen_Cheng_Chen_Wei_Chi_2022,pg-liu,cayci2024entropyNPG}, and we note that a tight local linear convergence rate of $1/(\eta\tau+1)^2$ is established in~\citet{pg-liu} for entropy regularized softmax NPG, where $\tau$ is the coefficient of the entropy regularization term.

Under softmax policy parameterization with a generative model, the $\tilde O(\varepsilon^{-2})$ sample complexity of NPG can be directly obtained by the aforementioned PMD analyzes~\citep{Xiao_2022,Lan_2021,Johnson_Pike-Burke_Rebeschini_2023}. For log-linear policy, NPG and its variant Q-NPG are studied in~\citet{Agarwal_Kakade_Lee_Mahajan_2019} with the $O(1/\sqrt{K})$ global sublinear convergence rate. Further combined with an unbiased sampler for state/action values, the $\tilde O(\varepsilon^{-2})$ sample complexity is obtained in~\citet{yuan2022linear} where  the regression sub-problem is solved by stochastic gradient descent. The $\tilde O(\varepsilon^{-2})$ sample complexity is also established in \citet{feng2024global} for general smooth and non-degenerate policies under Monte Carlo sampling through variance reduction, which is inspired by \citet{pmlr-v202-fatkhullin23a}. 
Natural actor-critic methods (NAC~\citep{konda2000ac,peters2008natural}) under i.i.d. sampling\footnote{Here the i.i.d. sampling assumes that there exists a sampler which can provide i.i.d. samples from the stationary distribution.}~\citep{wang2019neural,fu2021single} and Markov sampling~\citep{khodadadian2021finite,khodadadian2022finite,gaur2024closing,wang2024non,xu2020improving,xu2020non,cayci2022finite} are also extensively investigated. In the softmax tabular case, the sample complexity for on-policy NAC is $\tilde O(\varepsilon^{-6})$~\citep{khodadadian2022finite} and the one for off-policy NAC is $\tilde O(\varepsilon^{-3})$~\citep{khodadadian2021finite}.
 The best known sample complexity for NAC  is $\tilde O(\varepsilon^{-3})$ for general smooth policies as well as s well as the neural policy parameterization~\citep{xu2020improving,wang2024non,gaur2024closing}. 


\paragraph{Other related policy gradient methods} For exact projected policy gradient (PPG) in the tabular case, its $O(1/\sqrt{K})$ global sublinear convergence with small constant step size is established in~\citet{Agarwal_Kakade_Lee_Mahajan_2019}, which is improved to $O(1/K)$ in~\citet{Xiao_2022}. To get rid of the small step size constraint, \citet{ppgliu} develops a new analysis by computing the policy improvement directly and  establishes the $O(1/K)$ global sublinear convergence for any constant step size  and the $\gamma$-rate global linear convergence with  adaptive step sizes. For exact softmax PG, the asymptotic convergence is first established in~\citet{Agarwal_Kakade_Lee_Mahajan_2019}. The  $O(1/K)$ global sublinear rate is firstly established in~\citet{Mei_Xiao_Szepesvari_Schuurmans_2020} for a small constant step size, and later is extended to any constant step size in ~\citep{pg-liu}. These convergence bound, however, relies on some parameter  which is problem dependent and can be very small, resulting an exponential-time convergence in some hard MDP instances~\citep{li2023exponential}. To handle such issue, \citet{Mei_Xiao_Dai_Li_Szepesvari_Schuurmans_2020} proposes the escort PG, and one of its instances namely Hadamard PG~\citep{hadamard_pg} is shown to enjoy a global linear convergence. For the entropy regularization case,  softmax PG enjoys a global linear convergence~\citep{pg-liu,Mei_Xiao_Szepesvari_Schuurmans_2020}. There are also works that study PG under a general class of policies and investigate the convergence to the second-order stationary point and the global optimality convergence ~\citet{zhang2020global} and~\citet{Bhandari_Russo_2021}.

The sample complexity of stochastic PG is also widely studied. For stochastic tabular PG, we refer the readers to~\citet{yuan2022general} for a summary of the sample complexity results. For the general smooth policies, the sample complexity of stochastic PG is shown to be $\tilde O(\varepsilon^{-2})$ under an unbiased policy gradient sampler in~\citet{yuan2022general}. The sample complexity of stochastic PG  within the actor-critic framework has also received a lot of attention~\citep{olshevsky2023small,wang2019neural,chen2021closing,xu2020improving,xu2020non,wang2024non,chen2023finite,wu2020finite,kumar2023ac,xu2021doubly,qiu2021on,suttle2023beyond,barakat2022analysis,tian2023convergence}, where the actor is some general smooth policy or neural policy and the critic uses a linear approximation, compatible parameterization, or is approximated by a neural network. Overall, the sample complexity is $\tilde O(\varepsilon^{-2})$~\citep{xu2020improving,chen2023finite,wang2024non} to achieve a small gradient.

\subsection{Motivation and main contributions}
Existing sample complexity results for PMD-type algorithms mostly focus on the setting of generative model~\citep{Lan_2021,Xiao_2022,Johnson_Pike-Burke_Rebeschini_2023,protopapas2024policy} or the Markovian sampling but with nested-loop conditional temporal difference learning (CTD) estimation~\citep{Lan_2021,li2024stochastic}, {where action values are explicitly approximated to certain prerequisite accuracy. In contrast, TD-PMD~\citep{Geist_Scherrer_Pietquin_2019,liu2025tdpmd} updates the critic by applying only one-step Bellman operator over the last value estimate (without controlling the estimation accuracy). However, their analysis is limited to the exact setting and the generative model.} 
In this paper, we consider the more practical setting which combines both Markovian sampling {(with a mini-batch mechanism)} and TD learning. In particular, two algorithms including off-policy {Expected TD-PMD} (Algorithm~\ref{alg:Expected-TD-AC-PMD}) and mixed-policy {Approximate TD-PMD} (Algorithm~\ref{alg:Approximated-TD-AC-PMD}) are presented, and their sample complexities have been investigated. We focus on the tabular setting in this paper. Despite this, it is worth emphasizing that even though various policy gradient methods with policy/value function approximations  have been studied, it often requires strong assumptions in the parameter domain (such as the smoothness and non-degeneracy  of the policy) which are not applicable for the tabular setting. 
The main contributions of this paper are summarized as follows.
\begin{itemize}
    \item Under a small enough constant policy update step size, both Expected TD-PMD and Approximate TD-PMD enjoy a $\tilde O(\varepsilon^{-2})$ sample complexity to obtain  an average $\varepsilon$-optimal solution, where some constants and poly-logarithmic factors are hidden in $\tilde{O}(\cdot)$. Note that existing average-time sample complexities for PMD are also $\tilde O(\varepsilon^{-2})$ under the generative model or the Markovian sampling with {an inner-loop CTD learning}~\citep{Lan_2021,li2024stochastic}. In contrast, in Expected TD-PMD and Approximate TD-PMD, the new estimate of action value is obtained by applying only one step weighted Bellman operator (in sample form) to the last estimate, without {an inner-loop critic learning to approximate the exact action values accurately.} {The crux in our analysis is the development of a route to bound the stochastic bias term\footnote{See Lemma~\ref{lem:Expected-TD-AC-PMD:Lan-Xiao-bound} for the detailed description of the stochastic bias term.} in an inductive way.}
    \item Based on adaptive policy update step sizes, we further improve the sample complexity of Expected TD-PMD and Approximate TD-PMD to the $O(\varepsilon^{-2})$  sample complexity for the last-iterate-$\varepsilon$ optimality, where $O(\cdot)$ does not contain any poly-logarithmic factors about $\varepsilon$. To the best of our knowledge, it is the first  $O(\varepsilon^{-2})$ sample complexity for PMD-type methods under the unregularized MDP setting. In contrast, only the $\tilde{O}(\varepsilon^{-2})$
    sample complexity for the last-iterate optimality has  been obtained in \citet{Lan_2021,Xiao_2022, Johnson_Pike-Burke_Rebeschini_2023, protopapas2024policy,liu2025tdpmd} under the generative model. {The key of removing the poly-logarithmic factors is the adoption of  adaptive batch sizes.} As a byproduct, by specifying the mirror map $h$ to the squared $\ell_2$-norm, we have obtained the $O(\varepsilon^{-2})$ sample complexity for batch Q-learning  with absolute constant learning step size which is not known a prior as far as we know. 
\end{itemize}

The comparisons with typical sample complexity results for PMD-type methods (applicable for general mirror map $h$) are listed in Table~\ref{tab:AC-PMD}. Overall, we have studied the more practical setting which combines the Markovian sampling and the TD learning, and still be able to establish the competitive $\tilde{O}(\varepsilon^{-2})$ sample complexity and the improved ${O}(\varepsilon^{-2})$ sample complexity.
\begin{table}[ht!]
    \renewcommand{\arraystretch}{1.5}
    \small
    \centering
    \caption{Comparison with typical sample complexity results for PMD-type algorithms.}
    \vspace{0.2cm}
    \label{tab:AC-PMD}
    \begin{tabular}{ccccc} 
    \toprule
     & \makecell[cc]{\textbf{Sampling} \\ \textbf{scheme}} & \makecell[cc]{\textbf{Critic} \\ \textbf{update}} & \textbf{Metric} & \makecell[cc]{\textbf{Sample} \\ \textbf{complexity}} \\
    \midrule
    \multirow{2}{*}{\citet{Lan_2021}} & Generative model & MC & \multirow{2}{*}{\makecell[cc]{Average-time\\  \& Last-iterate}} & \multirow{2}{*}{$\tilde O (\varepsilon^{-2})$}\\
     & Markovian & {CTD} & &\\
    \hline
    \makecell*[ccc]{\citet{Xiao_2022} \\ \citet{Johnson_Pike-Burke_Rebeschini_2023} \\ \citet{protopapas2024policy}} & {Generative model} & MC & Last-iterate & $\tilde O(\varepsilon^{-2})$ \\ \hline
    \makecell*[cc]{\citet{li2024stochastic} \\ (Average-reward)} & Markovian & CTD  & Average-time & $\tilde O(\varepsilon^{-2})$ \\ \hline
    \citet{li2025policy} & Markovian & MC & Average-time & $\tilde O(\varepsilon^{-2})$ \\ \hline
    \citet{liu2025tdpmd} & Generative model & TD  & Last-iterate & $\tilde O(\varepsilon^{-2})$ \\ \hline
    \makecell*[cc]{\textbf{This paper} \\ (Theorems~\ref{thm:Expected-TD-AC-PMD:constant-step-size} and~\ref{thm:Approximated-TD-AC-PMD:constant-step-size})} & Markovian & TD & \makecell[cc]{Average-time \\ (Constant step size)} & $\tilde O (\varepsilon^{-2})$ \\ \hline 
    \makecell*[cc]{\textbf{This paper} \\ (Theorems~\ref{thm:Expected-TD-AC-PMD:adaptive-step-size} and~\ref{thm:Approximated-TD-AC-PMD:adaptive-step-size})} & Markovian & TD & \makecell[cc]{Last-iterate \\ (Adaptive step sizes)} & $O (\varepsilon^{-2})$ \\
    \bottomrule
    \end{tabular}
\end{table}

The rest of the paper is organized as below. In Section~\ref{sec:preliminary} we provide more preliminary knowledge about MDP and PMD. Section~\ref{sec:Expected-TD-AC-PMD} and Section~\ref{sec:Approximated-TD-AC-PMD} introduce Expected TD-PMD and Approximated TD-PMD, and offers the sample complexity analyzes respectively. {The proofs of critical lemmas are presented in Section~\ref{sec:lemma-proof}.} We conclude in Section~\ref{sec:conclusion} with a brief summary.

\section{Preliminary}
\label{sec:preliminary}
\subsection{More on MDP}
Throughout the paper, we use vectors $V \in \mathbb{R}^{|\calS|}$ and $Q \in \mathbb{R}^{|\calS||\calA|}$ to represent the value functions associated with the state and action, respectively. In addition, for any  $\mu\in\Delta(\calS)$, and $\rho\in\Delta(\calS\times\calA)$, let $V(\mu) = \E_{s\sim\mu}[V(s)]$ and $Q(\rho) = \E_{(s,a)\sim\rho}[Q(s,a)]$.
Note that  we will arrange any $\pi \in \Pi$ and vector $Q \in \mathbb{R}^{|\calS||\calA|}$ first by states and then  by actions, that is,
\begin{align*}
    Q &= [Q(s_1, a_1), \dots, Q(s_1, a_{|\calA|}), Q(s_2, a_1), \dots, Q(s_2, a_{|\calA|}), \dots, Q(s_{|\calS|}, a_{|\calA|})]^\top, \\
    \pi &= [\pi(a_1 | s_1), \dots, \pi(a_{|\calA|} | s_1), \pi(a_1 | s_2), \dots, \pi(a_{|\calA|} | s_2), \dots, \pi(a_{|\calA|} | s_{|\calS|})]^\top. 
\end{align*}

Define the state/action Bellman operator 
\begin{align*}
    [\calT^\pi V](s) &:= \E_{a\sim\pi(\cdot|s),\, s^\prime\sim P(\cdot|s,a)}[r(s,a) + \gamma \, V(s^\prime)], \\
    [\calF^\pi Q](s,a) &:= \E_{s^\prime\sim P(\cdot|s,a), \, a^\prime\sim\pi(\cdot|s^\prime)} [r(s,a) + \gamma \, Q(s^\prime, a^\prime)].
\end{align*}
It can be easily verified that $V^\pi$ and $Q^\pi$ satisfy $\calT^\pi V^\pi = V^\pi$ and $\calF^\pi Q^\pi = Q^\pi$, and both $\calT^\pi$ and $\calF^\pi$ are $\gamma$-contraction operators over $\mathbb{R}^{|\calS|}$ and $\mathbb{R}^{|\calS||\calA|}$ with respect to the $l_\infty$-norm. Furthermore, define the state/action transition matrix as
\begin{align*}
    P_{\scriptscriptstyle\calS}^\pi \in \mathbb{R}^{|\calS| \times |\calS|} &: \quad P^\pi_{\scriptscriptstyle\calS}(s, s^\prime) = \sum_{a\in\calA} \pi(a|s) P(s^\prime | s,a), \\
    P^{\pi}_{\scriptscriptstyle \calS\times\calA} \in \mathbb{R}^{|\calS||\calA| \times |\calS||\calA|}&: \quad P^{\pi}_{\scriptscriptstyle \calS\times\calA}((s,a), (s^\prime, a^\prime)) = P(s^\prime | s,a) \, \pi(a^\prime | s^\prime).
\end{align*}
That is, $P^{\pi}_{\scriptscriptstyle\calS}$ is the transition matrix in the state space and $P^{\pi}_{\scriptscriptstyle\calS\times\calA}$ is the transition matrix in the state-action space. Then the Bellman operators can be written in the following vector form
\begin{align*}
    \calT^\pi V = r^\pi + \gamma P^\pi_{\scriptscriptstyle\calS} V, \quad \calF^\pi Q = r + \gamma P^\pi_{\scriptscriptstyle\calS\times\calA} Q,
\end{align*}
where $r^\pi \in \mathbb{R}^{|\calS|}$ with $r^\pi(s) = \sum_a \pi(a|s) r(s,a)$ and $r \in \mathbb{R}^{|\calS||\calA|}$ is the reward function. Moreover, define the optimal Bellman operators:
\begin{align*}
    [\calT V](s) &:= \max_{\pi\in\Pi} \, [\calT^\pi V](s) = \max_{a\in\calA} \, \E_{s^\prime}[r(s,a) + \gamma V(s^\prime)], \\
    [\calF Q](s,a) &:= \max_{\pi\in\Pi} \, [\calF^\pi Q](s,a) = r(s,a) + \gamma\E_{s^\prime} \ls[\max_{a^\prime\in\calA} \,Q(s^\prime, a^\prime) \rs].
\end{align*}
Let $V^*$ and $Q^*$ be the fixed point of $\calT$, $\calF$, respectively. It is well-known~\citep{suttonRL} that $V^*$ and $Q^*$ is the optimal solution to the problem~\eqref{eq:RL-goal}, i.e., 
\begin{align*}
    \forall\, s, a: \quad V^*(s) = \max_{\pi\in\Pi} \, V^\pi(s), \quad Q^*(s,a) = \max_{\pi\in\Pi} \, Q^\pi(s,a),
\end{align*}
and there exists at least one optimal (deterministic) policy $\pi^*$ such that $V^{\pi^*}=V^*$ and $Q^{\pi^*}=Q^*$. Let $\Pi^* := \ls\{ \pi^* \in \Pi: \;\; V^{\pi^*}=V^* \rs\}$ be the optimal policy set. 

{For conciseness, it is assumed that  $r(s,a) \in [0,1]$.} Then,  for any policy $\pi \in \Pi$, there holds
\begin{align}
    \forall\, s\in\calS, \; a\in\calA: \quad 0 \leq V^\pi(s) \leq V^*(s) \leq \frac{1}{1-\gamma}, \quad 0\leq Q^{\pi}(s,a) \leq Q^*(s,a) \leq \frac{1}{1-\gamma}.
    \label{eq:bounded-policy-values}
\end{align}
 Define the state/action visitation measure as
\begin{align*}
    d^\pi_\mu(s) &:= (1-\gamma) \,\E\ls[ \sum_{t=0}^{\infty} \gamma^t \cdot \mathbf{1}(s_t = s) \; \bigg | \; s_0 \sim\mu \rs], \\
    v^\pi_{\rho}(s,a) &:= (1-\gamma)\, \E\ls[ \sum_{t=0}^\infty \gamma^t \cdot \mathbf{1}(s_t=s, a_t=a) \; \bigg | \; (s_0, a_0) \sim \rho \rs].
\end{align*}
The visitation measures corresponding to an optimal policy will be denoted as $d^*_\mu$ and $v^*_\rho$. The  performance difference lemma is a fundamental result in policy optimization.
\begin{lemma}[Performance difference lemma~\citep{kakade2002approximately,liu2025tdpmd}]
    For any $\pi\in\Pi$, $V\in\mathbb{R}^{|\calS|}$, $Q\in\mathbb{R}^{|\calS||\calA|}$, $\mu\in\Delta(\calS)$, and $\rho\in\Delta(\calS\times\calA)$, there holds
    \begin{align*}
        [V^\pi - V](\mu) &= \frac{1}{1-\gamma} \ls[ \calT^\pi V - V \rs](d^\pi_\mu), \\
        [Q^\pi - Q](\rho) &= \frac{1}{1-\gamma} \ls[ \calF^\pi Q - Q \rs](v^\pi_\rho).
    \end{align*}
    \label{lem:pdl}
\end{lemma}
\noindent As a corollary, we can obtain the following useful result.
\begin{lemma}
    For arbitrary $\pi, \pi^\prime \in \Pi$,
    \begin{align*}
        \| Q^{\pi} - Q^{\pi^\prime} \|_\infty &\leq \frac{\gamma |\calA|}{(1-\gamma)^2} \cdot \| \pi - \pi^\prime \|_\infty.
    \end{align*}
    \label{lem:Lip:Q-value}
\end{lemma}
\begin{proof}
    By the performance difference lemma, one has
    \begin{align*}
    \forall\, (s,a) \in \calS \times \calA: \quad | Q^{\pi}(s,a) - Q^{\pi^\prime}(s,a) | &= \frac{1}{1-\gamma} | [\calF^{\pi}Q^{\pi^\prime} - Q^{\pi^\prime}](d^{\pi}_{s,a}) | \\
    &= \frac{1}{1-\gamma} | [\calF^{\pi}Q^{\pi^\prime} - \calF^{\pi^\prime}Q^{\pi^\prime}](d^{\pi}_{s,a}) | \\
    &\leq \frac{1}{1-\gamma} \| \calF^{\pi} Q^{\pi} - \calF^{\pi^\prime} Q^{\pi^\prime} \|_\infty.
    \end{align*}
    Since
    \begin{align*}
        \forall\, (s,a) : \quad | \calF^{\pi} Q^{\pi^\prime}(s,a) - \calF^{\pi^\prime} Q^{\pi^\prime}(s,a) | &= \gamma \cdot | \E_{s^\prime \sim P(\cdot|s,a)} [ \langle Q^{\pi^\prime} (s^\prime, \cdot), \, \pi(\cdot|s^\prime) - \pi^\prime(\cdot|s^\prime) \rangle ] | \\
        &\leq \gamma \cdot | \E_{s^\prime \sim P(\cdot|s,a)} [ \frac{|\calA|}{1-\gamma} \cdot \| \pi(\cdot|s^\prime) - \pi^\prime(\cdot|s^\prime) \|_\infty ] | \\
        &\leq \frac{\gamma |\calA|}{1-\gamma} \cdot \| \pi - \pi^\prime \|_\infty,
    \end{align*}
   it follows that
    \begin{align*}
        \forall\, \pi, \pi^\prime: \quad \| Q^\pi - Q^{\pi^\prime} \|_\infty \leq \frac{\gamma |\calA|}{(1-\gamma)^2} \cdot \ls\| \pi - \pi^\prime \rs\|_\infty,
    \end{align*}
    which completes the proof.
\end{proof}

\subsection{More on TD-PMD}
As mentioned previously, TD-PMD updates the critic by tracking a sequence of TD learning of the action values which does not only involve current policy but also relies on the last value estimator. More precisely, TD-PMD admits the following form in the exact setting: 
\begin{equation}\label{eq:AC-PMD-policy-update}
\begin{aligned}
    \forall\, s\in\calS: \quad \pi_{k+1}(\cdot|s) &= \underset{p\in\Delta(\calA)}{\arg\max} \, \ls\{ \eta_k \ls\langle p, \, Q^k(s,\cdot) \rs\rangle - D^p_\pi(s) \rs\},\\
    Q^{k+1} &=\mathcal{F}^{\pi_{k+1}}Q^k,
\end{aligned}
\end{equation}
where $\eta_k > 0$ is the step size, $h$ is a convex function over $\Delta(\calA)$, and $D_h(p \, ||\, q)$ is the Bregman divergence induced by $h$. Note that here  and in the rest of the paper, we will omit the subscript $h$ in $D_h$ and use the short notation $D^\pi_{\pi^\prime}(s)$ for $D_h(\pi(\cdot|s) \, ||\, \pi^\prime(\cdot|s))$ and $D^p_{\pi}(s)$ for $D_h(p \, || \, \pi(\cdot|s))$. In contrast to the classical PMD which utilizes the true value function $Q^{\pi_k}$ as the critic, TD-PMD obtains the new critic by only performing one-step of Bellman iteration based on the current policy and the last critic estimation. When $h$ is specified to  the negative entropy $h(p) = -\sum_{a\in\calA} p(a)\log p(a)$ and quadratic function $h(p) = \frac{1}{2}\| p \|_2^2$, TD-PMD reduces to TD-NPG and TD-PQA, respectively, with the following policy update rules: 
\begin{align*}
    \mbox{(TD-NPG)} \quad&\pi_{k+1}(a|s) \propto \pi_k(a|s) \cdot \exp \ls( \eta \, Q^k(s,a) \rs), \\
    \mbox{(TD-PQA)} \quad& \pi_{k+1}(\cdot | s) = \Proj_{\Delta(\calA)} \ls[ \pi_k(\cdot|s) + \eta \, Q^k(s,\cdot) \rs],
\end{align*}
where in TD-PQA, $\Proj_{\Delta(\calA)}$ denotes the projection on the probabilistic simplex $\Delta(\calA)$.

 The following  three-point-descent lemma (see for example \citet{Xiao_2022,chen1993convergence}) is essential in the analysis of the PMD type methods.
\begin{lemma}
    For any $\pi \in \Pi$, $\eta > 0$, and $Q \in \mathbb{R}^{|\calS||\calA|}$, consider the following update:
    \begin{align*}
        \forall\, s\in\calS: \quad \pi^+(\cdot|s) = \underset{p\in\Delta(\calA)}{\arg\max} \ls\{ \eta  \ls\langle p, \, Q(s,\cdot) \rs\rangle - D^p_\pi(s) \rs\}.
    \end{align*}
    There holds $\pi^+(\cdot | s) \in \mbox{\textup{rint dom }}h \cap \Delta(\calA)$ and
    \begin{align*}
        \forall\, s\in\calS, \; p \in \Delta(\calA): \quad \eta \ls\langle \pi^+(\cdot|s) - p, \; Q(s,\cdot) \rs\rangle \geq D^{\pi^+}_{\pi}(s) + D^p_{\pi^+}(s) - D^p_{\pi}(s).
    \end{align*}
    \label{lem:three-point-descent-lemma}
\end{lemma}
\noindent For the mirror map $h$ being the negative entropy, there holds $\mbox{rint dom }h = \mbox{int }\Delta(\calA)$. Hence the policies generated by softmax NPG update always stay in $\mbox{rint }\Delta(\calA)$ if the initial policy lies in $\mbox{rint }\Delta(\calA)$.

\section{Expected TD-PMD with Off-Policy Markov Data}
\label{sec:Expected-TD-AC-PMD}

As shown in \citet{liu2025tdpmd}, TD-PMD admits the same sublinear convergence as PMD in the exact setting. Moreover, under the generative model where $[\mathcal{F}^{\pi_{k+1}}Q^k](s,a)$ in \eqref{eq:AC-PMD-policy-update} is estimated via random samples for every $(s,a)$, the $\tilde{O}(\varepsilon^{-2})$ sample complexity of TD-PMD is also established therein. In contrast, we will consider the sample complexity of TD-PMD under  the more practical and challenging  online Markov sampling model.

This section first considers the case where the  Markov data are obtained through an behavior policy $\pi_b$ which is sufficiently exploratory. 
Let $\tau_k$ be the trajectory data of length $B_k$, i.e.,
\begin{align*}
    \tau_k = \{ (s_t^k, a_t^k, r^k_t, s_{t+1}^k) \}_{t=0}^{B_k-1}, \quad \mbox{where} \; a^k_t \sim \pi_b(\cdot|s_t^k), \;\; r^k_t = r(s^k_t, a^k_t), \;\;  s^k_{t+1} \sim P(\cdot| s_t^k, a_t^k).
\end{align*}
For each tuple $(s_t^k, a_t^k, r_t^k, s_{t+1}^k)$, the expected TD error of critic $Q^k$ with respect to the target policy $\pi_{k+1}$ is constructed by
\begin{align*}
    \forall\, (s,a)\in\calS\times\calA: \quad \delta_t^k(s,a) := \mathds{1}_{(s^k_t,a^k_t)=(s,a) } \cdot\ls[r_t^k + \gamma\E_{a\sim\pi_{k+1}(\cdot|s_{t+1}^k)} \ls[Q^k(s_{t+1}^k, a)\rs] - Q^k(s^k_t, a^k_t)\rs].
\end{align*}
{Notice that when conditioned on $(s_t, a_t)=(s,a)$, the term $r_t^k + \gamma \E_{a\sim\pi_{k+1}(\cdot|s_{t+1}^k)}[Q^k(s_{t+1}^k, a)]$ is an unbiased estimator of $[\calF^{\pi_{k+1}} Q^k](s,a)$.} The expected TD error constructed from the whole trajectory $\tau_k$ is defined as the weighted average of $\delta_t^k$,
\begin{align*}
    \forall\, (s,a)\in\calS\times\calA: \quad \bar\delta_k(s,a) := {\sum_{t=0}^{B_k-1} c_t^k \cdot \delta_t^k(s,a)},
\end{align*}
where $c_t^k \geq 0$ are weights satisfying $\sum_{t=0}^{B_k - 1} c_t^k = 1$.
We then update the critic via the TD learning method, i.e.,
\begin{align*}
    \forall\, (s,a) \in \calS\times\calA: \quad Q^{k+1}(s,a) = Q^k(s,a) + \alpha_k \cdot \bar\delta_k(s,a),
\end{align*}
where $\alpha_k > 0$ is the critic update step size. 

\begin{figure}[ht!]
    \centering
    \includegraphics[width=1\linewidth]{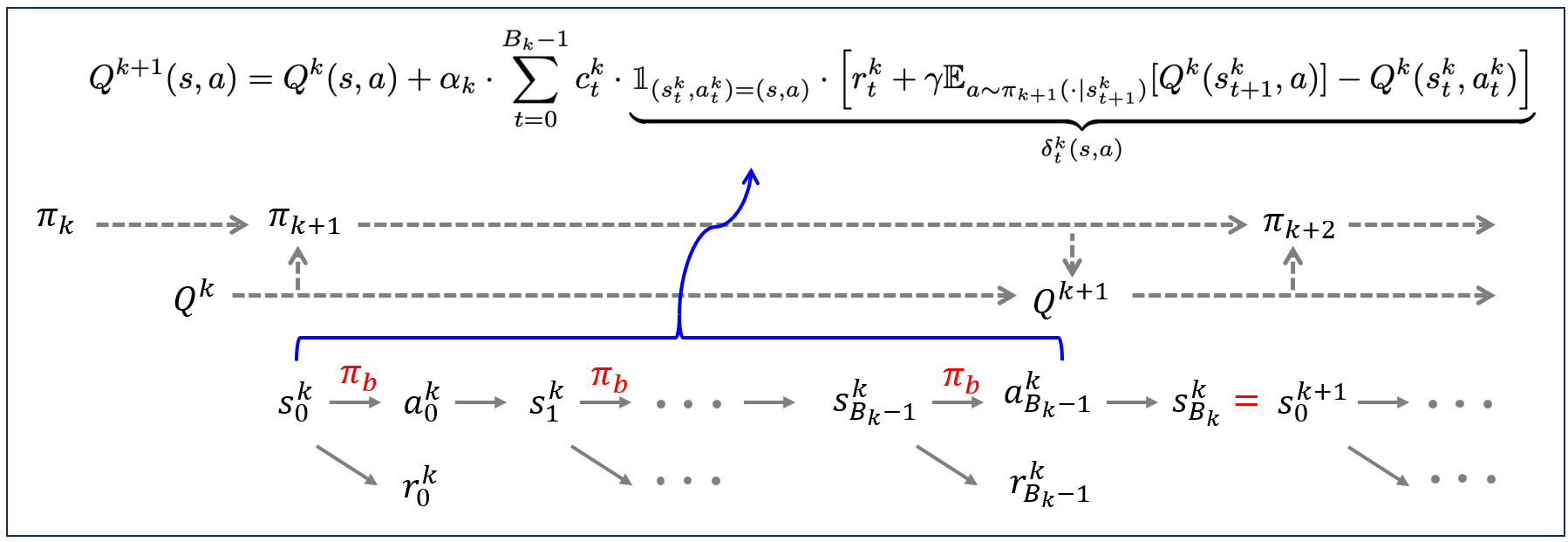}
    \caption{An illustration of Expected TD-PMD with off-policy Markov data (Algorithm~\ref{alg:Expected-TD-AC-PMD}),}
    \label{fig:Expected-TD-AC-PMD}
\end{figure}

The formal description of the Expected TD-PMD  with the Off-Policy Markov data is given in Algorithm~\ref{alg:Expected-TD-AC-PMD}, with an illustration given in Figure~\ref{fig:Expected-TD-AC-PMD}.
First note that  ``expected'' refers to the expectation term $\E_{a\sim\pi_{k+1}(\cdot|s_{t+1}^k)} \ls[Q^k(s_{t+1}^k, a)\rs]$, which is computable in the tabular case since $\pi_{k+1}$ is known. Moreover, Algorithm~\ref{alg:Expected-TD-AC-PMD} is an off-policy algorithm, as the behavior policy $\pi_b$ for sampling the data is not the target policy $\pi_{k+1}$. It is worth remarking that when the policy update step size $\eta_k$ goes to infinity, the updated policy $\pi_{k+1}$ is actually the greedy policy induced by the action value $Q^k$. In such case, Expected TD-PMD reduces to the  Q-learning algorithm~\citep{suttonRL} with batch Markov data. Thus the Expected TD-PMD can be viewed as a ``smoothed'' Q-learning algorithm. Furthermore, for the on-policy version with infinity critic update step size, i.e., $\eta_k = \infty$ and $\pi_b = \pi_{k+1}$, Expected TD-PMD with $B_k=1$ reduces to the expected SARSA algorithm~\citep{suttonRL}. 

\begin{algorithm}[t!]
    \small
    \caption{Expected TD-PMD with Off-Policy Markov Data}
    \label{alg:Expected-TD-AC-PMD}
\begin{algorithmic}
    \STATE {\bfseries Input:} Iterations $K$. Initial action value $Q^0=0$, initial policy $\pi_0$, critic step size $\alpha_k$, policy update step size $\eta_k$, initial state $s_0$, batch size $B_k$, average weights $c^k_t$.
    \STATE Set $s_0^0 = s_0$.
    \FOR{$k=0,1, \dots, K$}
    \STATE \textbf{(Policy update)} Update the target policy by
    \begin{align*}
        \forall\, s\in\calS: \quad \pi_{k+1}(\cdot|s) = \underset{p\in\Delta(\calA)}{\arg\max} \; \ls\{ \eta_k \ls\langle p ,\, Q^k(s,\cdot) \rs\rangle - D^p_{\pi_k}(s) \rs\}.
    \end{align*}
    \STATE \textbf{(Sampling)} Obtain $\tau_k$ by obeying behavior policy $\pi_b$ starting from $s_0^k$,
    \begin{align*}
        \tau_k = \{ (s_t^k, a_t^k, r^k_t, s_{t+1}^k) \}_{t=0}^{B_k-1}, \quad \mbox{where} \; r^k_t = r(s^k_t, a^k_t), \;\; a^k_t \sim \pi_b(\cdot|s_t^k), \;\; s^k_{t+1} \sim P(\cdot| s_t^k, a_t^k),
    \end{align*}
    and set $s_0^{k+1} = s_{B_k}^k$.
    \STATE \textbf{(Critic update)} Construct the expected TD error,
    \begin{align*}
        &\delta_t^k(s,a) := \mathds{1}_{(s^k_t,a^k_t)=(s,a) } \cdot\ls[r_t^k + \gamma\E_{a\sim\pi_{k+1}(\cdot|s_{t+1}^k)} \ls[Q^k(s_{t+1}^k, a)\rs] - Q^k(s^k_t, a^k_t)\rs], \\
        &\bar\delta_k(s,a) := {\sum_{t=0}^{B_k-1} c_t^k \cdot \delta_t^k(s,a)}.
    \end{align*}
    \STATE Update the critic by
    \begin{align*}
        Q^{k+1}(s,a) = Q^k(s,a) + \alpha_k \cdot \bar\delta_k(s,a).
    \end{align*}
    \ENDFOR
    \STATE Sample $\hat K$ from $\{ 0, 1, \dots, K \}$ with uniform probability.
    \STATE Output $\pi_K$ and $\pi_{\hat K}$.
\end{algorithmic}
\end{algorithm}

Before proceeding to the sample complexity analysis of Expected TD-PMD, we give some intuition about the critic update formula. Conditioned on the first state $s_0^k$, the critic $Q^k$ and the target policy $\pi_{k+1}$, it is straightforward to compute the expectation of the TD error,
\begin{align*}
    \E\ls[ \delta_t^k(s,a) \rs] = \P\ls( s_t^k=s \rs) \pi_b(a|s) \cdot \ls[ \calF^{\pi_{k+1}} Q^k - Q^k \rs](s,a).
\end{align*}
Therefore the conditional expectation of $Q^{k+1}$ is given by
\begin{align}
    \E\ls[ Q^{k+1}(s,a) \rs] = \E\Bigg[ Q^k(s,a) + \alpha_k\cdot\sum_{t=0}^{B_k-1} c_t^k\cdot \underbrace{\P(s_t^k=s) \pi_b(a|s)}_{\mbox{exploration}} \cdot \underbrace{\ls[ \calF^{\pi_{k+1}} Q^k - Q^k \rs]}_{\mbox{exploitation}}(s,a) \Bigg].
    \label{eq:Expected-TD-AC-PMD:intuition}
\end{align}
From this equation, it can be seen that there are two parts in the TD error: the exploration part and the exploitation part. The exploitation part is a Bellman operator, offering a shift for the critic to approach $Q^{\pi_{k+1}}$. The exploration part is indeed the probability for the behavior policy $\pi_b$ to obtain action-state $(s,a)$. Noticing that the critic update will fail if the exploration part goes to $0$, we require the following two assumptions to ensure the sufficient  exploration over all $(s,a) \in \calS\times\calA$.

\begin{assumption}[Sufficient exploration of behavior policy]
    We have $\tilde\pi_b := \min_{(s,a) \in \calS\times\calA} \pi_b(a|s) > 0$.
    \label{ass:behavior-exploration}
\end{assumption}

\begin{assumption}[Ergodicity]
    The state transition matrix induced by $\pi_b$, i.e., $P^{\pi_b}_{\scriptscriptstyle \calS}$, is ergodic over $\calS$. 
    \label{ass:ergodicity}
\end{assumption}
Assumption~\ref{ass:behavior-exploration} guarantees the exploration over the action space. Note that if there exists a pair $(s,a)\in\calS\times\calA$ such that $\pi_b(a|s)=0$, then $(s_t^k, a_t^k)\neq(s,a)$ almost surely and $Q^k(s,a)$ will not be updated. 
Assumption~\ref{ass:ergodicity} is very standard for the analysis of the actor-critic algorithms with the Markov data, see for example~\citet{li2024stochastic,khodadadian2021finite,khodadadian2022finite}. This assumption can guarantee the sufficient exploration over the state space.
\begin{proposition}
   Under Assumption~\ref{ass:ergodicity}, there exists a unique stationary distribution $\nu^{\pi_b} \in \Delta(\calS)$ for the transition kernel $P^{\pi_b}_{\scriptscriptstyle\calS}$ which satisfies
    \begin{align*}
        (\nu^{\pi_b})^\top P^{\pi_b}_{\scriptscriptstyle\calS} = (\nu^{\pi_b})^\top\quad\mbox{and}\quad \tilde\nu_b := \min_{s\in\calS} \; \nu^{\pi_b}(s) > 0.
    \end{align*}
    Furthermore, letting $\{ s_t \}_{t\geq0}$ be the Markov chain induced by $P^{\pi_b}_{\scriptscriptstyle\calS}$, then it is geometrically mixing. That is, there exists two constants $m_b>0$ and  $\kappa_b\in(0,1)$ such that
    \begin{align*}
        \forall\, s\in\calS, \; n \geq 0, \; t\geq0: \quad  d_{TV}(\nu^{\pi_b}(\cdot), \, \P(s_{n+t} = \cdot | s_n = s)) \leq m_b\kappa_b^t,
    \end{align*}
    where $d_{TV}$ denotes the total variation distance.
    \label{pro:MDP-ergodicity}
\end{proposition}

The proof of Proposition~\ref{pro:MDP-ergodicity} can be found for example in~\citet{khodadadian2021finite}, {which is the application of~\citet[Theorem~4.9]{levin2017markov}}. Noticing that at any $k$-th iteration of Expected TD-PMD, the generated state trajectory $\{ s_t^k \}_{t\geq 0}$ is the Markov chain with transition kernel of $P^{\pi_b}_{\scriptscriptstyle\calS}$. Thus by Proposition~\ref{pro:MDP-ergodicity},
\begin{align}
    \forall\, 0 \leq k\leq K, \; s\in\calS, \; n \geq 0, \; t\geq0: \quad  d_{TV}(\nu^{\pi_b}(\cdot), \, \P(s_{n+t}^k = \cdot | s_n^k = s)) \leq m_b\kappa_b^t,
    \label{eq:Expected-TD-AC-PMD:MDP-ergodicity}
\end{align}
implying that the exploration part $\P(s_t^k=s)\pi_b(a|s)$ converges to $\nu^{\pi_b}(s)\pi_b(a|s) \geq \tilde\nu_b\tilde\pi_b > 0$ as $B_k$ increases. Therefore,  Assumptions~\ref{ass:behavior-exploration} and~\ref{ass:ergodicity} together can guarantee sufficient exploration over $\calS\times\calA$.

We need to make an additional assumption on the mirror map $h$ for the subsequent analysis when $\eta_k$ is a small constant step size. 
\begin{assumption}
    The mirror map $h$ is further $\lambda$-strongly convex w.r.t. $\ell_p$-norm over $\Delta(\calA)$ where $p \in [1,+\infty]$ is some constant.
    \label{ass:h}
\end{assumption}
\noindent Assumption~\ref{ass:h} is also common in the convergence analysis of mirror descent methods ({see for example~\citet[Chapter~9]{beck2017first}}), and typical mirror maps including quadratic function and negative entropy satisfy this assumption. Indeed, one can verify that  $h(p) = \frac{1}{2}\| p \|_2^2$ is $1$-strongly convex with respect to the $l_2$-norm, while  $h(p) = -\sum_{a} p(a) \log p(a)$ is $1$-strongly convex with respect to the $l_1$-norm. By Assumption~\ref{ass:h}, it is clear that for {PMD-type update (including the policy update in Expected TD-PMD)} there holds
\begin{align}
    \forall\, 0\leq k \leq K, \; s\in\calS: \quad D^{\pi_{k+1}}_{\pi_k}(s) \geq \frac{\lambda}{2} \ls\| \pi_{k+1}(\cdot|s) - \pi_k(\cdot|s) \rs\|_p^2 \geq \frac{\lambda}{2} \ls\| \pi_{k+1}(\cdot|s) - \pi_k(\cdot|s) \rs\|_\infty^2.
    \label{eq:Bregman-lower-bound}
\end{align}
Moreover, the policy  shift $\| \pi_{k+1} - \pi_k \|_\infty$ in {after the PMD-type  update} can be bounded by the policy update step size based on the strong convexity of mirror map $h$ and the three-point-descent lemma.
{
\begin{lemma}
    Consider the PMD-type policy update in \eqref{eq:AC-PMD-policy-update}.
    Under Assumption~\ref{ass:h}, there holds
    \begin{align*}
        \forall\, 0\leq k \leq K: \quad \ls\| \pi_{k+1} - \pi_k \rs\|_\infty \leq \frac{\| Q^k(s,\cdot) \|_1}{\lambda} \eta_k. 
    \end{align*}
    \label{lem:policy-shift}
\end{lemma}
}
\begin{proof}
    By Assumption~\ref{ass:h}, one has
    \begin{align*}
        \forall\, x, y\in\Delta(\calA): \quad D_h(x \, || \, y) = h(x)-h(y) - \ls\langle \nabla h(y), \, x-y \rs\rangle \geq \frac{\lambda}{2} {\| x - y \|_p^2.}
    \end{align*}
    By H\"older's inequality and the three-point descent lemma (Lemma~\ref{lem:three-point-descent-lemma}), 
    \begin{align*}
        \eta_k  \| \pi_{k+1}(\cdot|s) - \pi_k(\cdot|s) \|_p \cdot \| Q^k(s,\cdot) \|_{\frac{p}{p-1}} &\geq \eta_k \ls\langle \pi_{k+1}(\cdot|s) - \pi_k(\cdot|s), \, Q^k(s,\cdot) \rs\rangle \\
        &\geq D_{\pi_k}^{\pi_{k+1}}(s) + D^{\pi_k}_{\pi_{k+1}}(s) \\
        &\geq \lambda \cdot \| \pi_{k+1}(\cdot|s) - \pi_k(\cdot|s) \|_p^2.
    \end{align*}
    Since $\| Q^k(s,\cdot) \|_{\frac{p}{p-1}} \leq \ls\| Q^k(s,\cdot) \rs\|_1$, we obtain
    \begin{align*}
        \| \pi_{k+1}(\cdot|s) - \pi_k(\cdot|s) \|_p \leq \frac{\| Q^k(s,\cdot) \|_1}{\lambda} \eta_k,
    \end{align*}
    which implies that $\| \pi_{k+1} - \pi_k \|_\infty \leq \frac{\| Q^k(s,\cdot) \|_1}{\lambda} \eta_k$.
\end{proof}

The following property about the boundedness of $Q^k$ can be verified easily.
\begin{lemma}
    Consider Expected TD-PMD \textup{(}Algorithm~\ref{alg:Expected-TD-AC-PMD}\textup{)} with critic step size $\alpha_k$ satisfying $\alpha_k \in (0,1]$. There holds
    \begin{align*}
        \forall\, 0 \leq k \leq K, \; (s,a)\in\calS\times\calA: \quad 0 \leq Q^k(s,a) \leq \frac{1}{1-\gamma}.
    \end{align*}
    \label{lem:Expected-TD-AC-PMD:bounded-value}
\end{lemma}

\subsection{Expected TD-PMD with constant step size}
\label{sec:Expected-TD-AC-PMD-with-constant-step-size}
We are now ready to provide the sample complexity analysis for Expected TD-PMD. This subsection focuses on the constant policy update step size setting, i.e., $\eta_k = \eta > 0$. We begin with the following upper bound for Expected TD-PMD.

\begin{lemma}
    Consider Expected TD-PMD \textup{(}Algorithm~\ref{alg:Expected-TD-AC-PMD}\textup{)}. Suppose that Assumption~\ref{ass:h} holds. Let $\eta_k = \eta > 0$ and $\alpha_k = \alpha \in (0,1]$. Fixing any optimal policy $\pi^* \in \Pi^*$ and $\mu \in \Delta(\calS)$, there holds
    \begin{align*}
        \E\ls[ V^*(\mu) - V^{\pi_{\hat K}}(\mu) \rs] &\leq \underbrace{\frac{1}{K+1} \ls[ \frac{\| D^{\pi^*}_{\pi_0} \|_\infty}{\eta(1-\gamma)} + \frac{1}{(1-\gamma)^2}  \rs]}_{\mbox{\small Exact PMD bound}} \\
        &\;\;\;\;+ \underbrace{\frac{1}{K+1} \cdot \frac{\eta |\calA|^2}{2(1-\gamma)^2\lambda} \cdot \E\ls[ \sum_{k=0}^K \| Q^{k} - Q^{\pi_k} \|_\infty^2 \rs]}_{\mbox{\small Stochastic Variance term}} \\
        &\;\;\;\;+ \underbrace{\frac{1}{K+1} \cdot \frac{1}{1-\gamma} \cdot \E \ls[ \sum_{k=0}^K \E_{s\sim d^*_\mu} \ls[ \ls\langle \pi^*(\cdot|s) - \pi_k(\cdot|s), \, Q^{\pi_k}(s,\cdot) - Q^k(s,\cdot) \rs\rangle \rs] \rs]}_{\mbox{\small Stochastic Bias term}},
    \end{align*}
    where $\| D^{\pi^*}_{\pi_0} \|_\infty := \max_{s\in\calS} \, D^{\pi^*}_{\pi_0}(s)$ and $\hat K$ is a random time index that is uniformly sampled from $0, \dots, K$. 
    \label{lem:Expected-TD-AC-PMD:Lan-Xiao-bound}
\end{lemma}

    The proof of Lemma~\ref{lem:Expected-TD-AC-PMD:Lan-Xiao-bound}  follows the standard route of the  stochastic PMD/NAC analysis, see for example~\citet{Lan_2021,khodadadian2022finite,xu2020improving,xu2020non,fu2021single,Xiao_2022,li2024stochastic,liu2025tdpmd,li2025policy}, and we present the proof in Section~\ref{sec:proof:Expected-TD-AC-PMD:Lan-Xiao-bound} for completeness. In the exact vanilla PMD setting (i.e., $Q^k = Q^{\pi_k}$), the stochastic variance and bias terms vanish and the upper bound in Lemma~\ref{lem:Expected-TD-AC-PMD:Lan-Xiao-bound} reduces to the same upper bound for the  exact PMD in~\citet[Theorem~8]{Xiao_2022}.

Note that the exact PMD bound vanishes in Lemma~\ref{lem:Expected-TD-AC-PMD:Lan-Xiao-bound} as $K$ grows, and it is also not hard to see that the stochastic variance term can be well controlled if the policy update step size $\eta$ is small enough. Thus the challenge is  to bound the stochastic bias term, which closely relies on the critic update formula of Expected TD-PMD and occupies most of the proof. To this end, we reformulate the critic update formula of Expected TD-PMD into a more concise form. By equation~\eqref{eq:Expected-TD-AC-PMD:intuition} we know that
\begin{align*}
    \E\ls[ Q^{k+1}(s,a) \rs] &= \E\Bigg[ Q^k(s,a) + \alpha_k\cdot \sum_{t=0}^{B_k-1} c_t^k \cdot \underbrace{\P(s_t^k=s)}_{\to \nu^{\pi_b}(s) \mbox{\small \;by equation~\eqref{eq:Expected-TD-AC-PMD:MDP-ergodicity}}} \pi_b(a|s) \cdot {\ls[ \calF^{\pi_{k+1}} Q^k - Q^k \rs]}(s,a) \Bigg] \numberthis \label{eq:Expected-TD-AC-PMD:intuition-2} \\
    &\approx \E\ls[ Q^k(s,a) + \alpha_k\cdot {\sigma^{\pi_b}(s,a)} \cdot {\ls[ \calF^{\pi_{k+1}} Q^k - Q^k \rs]}(s,a) \rs],
\end{align*}
where $\sigma^{\pi_b}(a|s) := \nu^{\pi_b}(s) \pi_b(a|s)$. One can verify that $\sigma^{\pi_b}$ is the stationary distribution of $\{ (s^k_t, a^k_t) \}_{t\geq0}$, that is, $(\sigma^{\pi_b})^\top P^{\pi_b}_{\scriptscriptstyle \calS\times\calA} = (\sigma^{\pi_b})^\top$. For simplicity, assume $\alpha_k=\alpha\in(0,1]$. It is natural to define the following $\sigma^{\pi_b}$-weighted Bellman operator,
\begin{align*}
    \forall\, \pi \in \Pi, \; Q\in\mathbb{R}^{|\calS||\calA|}: \quad \calF^{\pi}_{\pi_b,\alpha} Q(s,a) := Q(s,a) + \alpha\cdot\sigma^{\pi_b}(s,a) \ls[ \calF^{\pi} Q - Q \rs](s,a),
\end{align*}
or equivalently,
\begin{align*}
    \calF^{\pi}_{\pi_b,\alpha} Q := Q + \alpha\cdot\Sigma^{\pi_b} \ls[ \calF^\pi Q - Q \rs],
\end{align*}
where $\Sigma^{\pi_b} = \mbox{diag}(\sigma^{\pi_b}) \in \mathbb{R}^{|\calS||\calA| \times |\calS||\calA|}$. When $\alpha=1$, we simply denote $\calF^{\pi}_{\pi_b,\alpha}$ by $\calF^{\pi}_{\pi_b}$, i.e., 
\begin{align*}
    \calF^{\pi}_{\pi_b} Q := Q + \Sigma^{\pi_b} \ls[ \calF^\pi Q - Q \rs].
\end{align*}

By Assumption~\ref{ass:behavior-exploration}, Assumption~\ref{ass:ergodicity} and Proposition~\ref{pro:MDP-ergodicity}, it is easily shown that 
\[\tilde\sigma_{b} := \min_{s,a} \; \sigma^{\pi_b}(s,a) \geq \tilde\nu_b \tilde\pi_b > 0.\]
In this case, one can show that $\calF^{\pi}_{\pi_b,\alpha}$ shares similar properties with the vanilla Bellman operator $\calF^\pi$.

\begin{proposition}
    Suppose Assumption~\ref{ass:behavior-exploration} and Assumption~\ref{ass:ergodicity} hold. Then for any $\pi\in\Pi$,
    \begin{enumerate}
        \item[\textup{(a)}] $\calF^{\pi}_{\pi_b,\alpha} Q^\pi = Q^\pi$\textup{;}
        \item[\textup{(b)}] $\forall\, Q, \, Q^\prime \in \mathbb{R}^{|\calS||\calA|}: \;\; \ls\| \calF^{\pi}_{\pi_b,\alpha} Q - \calF^{\pi}_{\pi_b,\alpha} Q^\prime \rs\|_\infty \leq (1-\alpha(1-\gamma) \tilde\sigma_b) \ls\| Q - Q^\prime \rs\|_\infty$.
    \end{enumerate}
    \label{pro:Expected-TD-AC-PMD:sigma-bellman-property}
\end{proposition}

\begin{proof}
    Claim $(a)$ follows directly from the fact that $\calF^\pi Q^\pi = Q^\pi$. Notice that for any $(s,a) \in \calS\times\calA$,
    \begin{align*}
        \ls| \calF^{\pi}_{\pi_b, \alpha} Q(s,a) - \calF^{\pi}_{\pi_b, \alpha} Q^\prime(s,a) \rs| &= \ls| (1-\alpha\sigma^{\pi_b}(s,a)) \ls(Q(s,a) - Q^\prime(s,a)\rs) + \alpha\sigma^{\pi_b}(s,a) \ls[ \calF^\pi Q(s,a) - \calF^\pi Q^\prime(s,a)  \rs] \rs| \\
        &\leq (1-\alpha\sigma^{\pi_b}(s,a)) \ls\| Q - Q^\prime \rs\|_\infty + \alpha\sigma^{\pi_b}(s,a) \underbrace{\ls\| \calF^\pi Q - \calF^\pi Q^\prime \rs\|_\infty}_{\leq \gamma\ls\| Q - Q^\prime \rs\|_\infty} \\
        &\leq (1-\alpha(1-\gamma) \sigma^{\pi_b}(s,a)) \cdot \ls\| Q - Q^\prime \rs\|_\infty \\
        &\leq (1-\alpha(1-\gamma) \tilde\sigma_b) \cdot \ls\| Q - Q^\prime \rs\|_\infty,
    \end{align*}
    which yields (b) directly.
\end{proof}

Let
$
A^\pi = I-\alpha\Sigma^{\pi_b}(I-\gamma P^{\pi}_{\scriptscriptstyle \mathcal S\times \mathcal A}).
$
One has 
\begin{align*}
    \mathcal{F}^\pi_{\pi_b,\alpha}Q-\mathcal{F}^\pi_{\pi_b,\alpha}Q'&=A^\pi (Q-Q').
\end{align*}
It follows immediately from Proposition~\ref{pro:Expected-TD-AC-PMD:sigma-bellman-property} that 
\begin{align}
    \forall\, \pi \in \Pi: \quad \|A^\pi\|_\infty\leq (1-\alpha(1-\gamma) \tilde\sigma_b).\label{eq:Api-bound}
\end{align}

Define the stochastic noise $\bar\omega_k \in \mathbb{R}^{|\calS||\calA|}$ as follows:
\begin{equation}
\begin{aligned}
    \omega_t^k(s,a) &= \delta_t^k(s,a) - \ls[ \calF^{\pi_{k+1}}_{\pi_b} Q^k - Q^k \rs](s,a), \\
    \bar\omega_k(s,a) &=  {\sum_{t=0}^{B_k-1} c_t^k \cdot \omega_t^k(s,a)} = \bar\delta_k(s,a) - \ls[ \calF^{\pi_{k+1}}_{\pi_b} Q^k - Q^k \rs](s,a).\\
\end{aligned}
    \label{eq:Expected-TD-AC-PMD:stochastic-noise}
\end{equation}
We can rewrite the critic update of Expected TD-PMD into the following two equivalent ways: 
\begin{equation}
    Q^{k+1} = \calF^{\pi_{k+1}}_{\pi_b, \alpha} Q^k + \alpha\cdot \bar\omega_k\,\,\mbox{ or }\,\,
    Q^{k+1} = Q^k + \alpha\cdot \Sigma^{\pi_b} \ls[ \calF^{\pi_{k+1}} Q^k - Q^k \rs] + \alpha\cdot \bar\omega_k.
    \label{eq:Expected-TD-AC-PMD:critic-update}
\end{equation}
This expression means that the critic update in Expected TD-PMD is essentially the application of the $\sigma^{\pi_b}$-weighted Bellman operator $\calF^{\pi_{k+1}}_{\pi_b, \alpha}$ with an additive stochastic Markov noise $\bar\omega_k$.  Moreover, we will consider the following average weight:
\begin{align}
    c_t^k = \dfrac{\vartheta^{B_k - t - 1}}{\sum_{t=0}^{B_k-1} \vartheta^{t}},
    \label{eq:Expected-TD-AC-PMD:average-weight}
\end{align}
where $\vartheta \in [0,1]$ is some preset parameter. Note that when $\vartheta =0$, we define $0^0=1$ here, therefore only the last sample in the batch will be used. 

To bound the stochastic bias part,  it is convenient to define the following index matrix:
\begin{equation}\label{eq:EsJs}
\begin{aligned}
    E_s &= \left[ {\mathbf 0} \;\; {\cdots}\;\; {\mathbf 0} \;\;\begin{array}{cccc}
        1 & & & \\
         & \ddots && \\
         && 1& 
    \end{array}\;\; \mathbf{0}\;\;  \cdots\;\; \mathbf{0}  \right] \in \mathbb{R}^{|\calA| \times |\calS||\calA|}, \\[1em]
    J_s &= E_s^\top E_s \in \mathbb{R}^{|\calS||\calA| \times |\calS||\calA|},
    \end{aligned}
\end{equation}
where $\mathbf 0\in \mathbb{R}^{|\mathcal{A}|\times|\mathcal{A}|}$, and each $1$ corresponds to an action associated with state $s$.
Then it is evident that $Q(s,\cdot) = E_s Q$ and $\pi(\cdot |s) = E_s \pi$, and thus we have
\begin{align*}
    \inner{\pi(\cdot|s)}{Q(s,\cdot)} = \inner{E_s\pi}{E_sQ} = \inner{E_s^\top E_s \pi}{\;Q} = \inner{J_s \pi}{\;Q}.
\end{align*}

We first have the following decomposition for the stochastic bias, whose proof is presented in Section~\ref{sec:proof-bias-deomposition}.
\begin{lemma}\label{lem:decomposition of bias}
Consider Expected TD-PMD \textup{(}Algorithm~\ref{alg:Expected-TD-AC-PMD}\textup{)}. Let $\eta_k = \eta > 0$ and $\alpha_k = \alpha \in (0,1]$. Fixing any optimal policy $\pi^* \in \Pi^*$, for any $0 \leq k \leq K$ and $s\in\calS$, there holds
 \begin{align*}
    \ls|\E\ls[ \inner{\pi^*(\cdot|s) - \pi_k(\cdot|s)}{Q^{\pi_k}(s,\cdot) - Q^k(s,\cdot)} \rs]\rs| 
        &\leq {\ls| B_0^{(k)} \rs|} + \sum_{j=1}^k \E{\ls[\ls|C_j^{(k)}\rs| + \ls|D_j^{(k)}\rs| + \ls|E_j^{(k)}\rs| \rs]} + \sum_{j=1}^k {\ls| \E\ls[F_j^{(k)}\rs] \rs|},
 \end{align*}
 where 
 \begin{align*}
     B_0^{(k)} &= \inner{\ls[ (A_0)^k \rs]^\top J_s [\pi^* - \pi_0]}{Q^{\pi_0} - Q^0},\\
     C_j^{(k)} &= \inner{\ls[ (A_j)^{k-j+1} \rs]^\top J_s[\pi^* - \pi_j]}{[Q^{\pi_j} - Q^{\pi_{j-1}}]},\\
     D_j^{(k)} & = \inner{\ls[ (A_j)^{k-j+1} \rs]^\top J_s[\pi_{j-1} - \pi_j]}{[Q^{\pi_{j-1}} - Q^{{j-1}}]},\\
     E_j^{(k)} & =\inner{\ls[ (A_j)^{k-j+1} - (A_{j-1})^{k-j+1} \rs]^\top J_s[\pi^* - \pi_{j-1}]}{[Q^{\pi_{j-1}} - Q^{{j-1}}]},\\
     F_j^{(k)} & =\inner{\ls[(A_j)^{k-j}\rs]^\top J_s [\pi^* - \pi_j]}{\alpha \bar\omega_{j-1}},
 \end{align*}
 with $A_j = A^{\pi_j} = I-\alpha\Sigma^{\pi_b}(I-\gamma P^{\pi_j}_{\mathcal S\times \mathcal A})$.
\end{lemma}

{
\begin{remark}
    In~\textup{\citet{Lan_2021,li2024stochastic}}, the stochastic bias term is bounded by explicitly controlling $\ls| \E[Q^{\pi_k}(s,a) - Q^k(s,a) \; | \; \pi_k] \rs| \leq \varsigma$ with sufficiently small $\varsigma > 0$ for all iteration $k$ via incorporating the inner-loop CTD learning, which constructs an empirical Bellman operator $\hat\calF^{\pi_k}$ by interval samples and applies it multiple times. In contrast, we track the critic update process without introducing any inner-loop critic learning algorithm. In order to bound the stochastic bias term, we  decompose it into the accumulation of policy shifts and stochastic noises terms and we bound them separately. 
\end{remark}
}

{The first term $B_0^{(k)}$ is a remainder term that vanishes exponentially by the contraction of $(A_0)^k$.} Suppose that Assumptions~\ref{ass:behavior-exploration} and~\ref{ass:ergodicity} hold. Noting $\|A_j\|_\infty\leq (1-\alpha(1-\gamma) \tilde\sigma_b)$ (see equation~\eqref{eq:Api-bound}), the term with $B_0^{(k)}$ can be bounded easily as follows:
 \begin{align*}
        \ls| B_0^{(k)} \rs| &= \ls| \inner{\ls[ (A_0)^k \rs]^\top J_s [\pi^* - \pi_0]}{Q^{\pi_0} - Q^0} \rs| \\
        &\leq \ls\| \ls[ (A_0)^k \rs]^\top J_s [\pi^* - \pi_0] \rs\|_1 \cdot \ls\| Q^{\pi_0} - Q^0 \rs\|_\infty \\
        &\leq \underbrace{\ls\| (A_0)^\top \rs\|_1^k}_{=\ls\| A_0 \rs\|_\infty ^k} \times \underbrace{\ls\| J_s[\pi^* - \pi_0] \rs\|_1}_{\leq 2} \times \underbrace{\ls\| Q^{\pi_0} - Q^0 \rs\|_\infty}_{\leq (1-\gamma)^{-1} \mbox{\small  by equation~\eqref{eq:bounded-policy-values} and Lemma~\ref{lem:Expected-TD-AC-PMD:bounded-value}}} \\
        &\leq \frac{2}{1-\gamma} \cdot \ls[ 1-(1-\gamma) \alpha\tilde \sigma_b \rs]^k.\numberthis\label{eq:expected-B0k-bound}
    \end{align*}
{The terms $C_j^{(k)}$, $D_j^{(k)}$, and $E_j^{(k)}$ can be controlled by the policy shift $\| \pi_j - \pi_{j-1} \|_\infty$ (which is further controlled by $\eta$ in Lemma~\ref{lem:policy-shift}) using the Lipschitz-continuity (w.r.t. $\pi$) of $Q^\pi$ and $P^{\pi}_{\scriptscriptstyle\calS\times\calA}$. The term $F_j^{(k)}$ is the stochastic noise term in terms of $\bar\omega_{j-1}$, which can be bounded by the batch size.} The bounds for the terms with $C_j^{(k)}, D_j^{(k)}, E_j^{(k)}$ and $F_j^{(k)}$ are presented in the following three lemmas whose proofs are deferred to Sections~\ref{sec:expected-CjkDjk-bounds}, \ref{sec:expected-Ejk-bound}, and \ref{sec:expected-Fjk-bound}, respectively.

\begin{lemma}[Bounds for terms with $C_j^{(k)}$ and $D_j^{(k)}$]\label{lem:expected-CjkDjk-bounds} Suppose that Assumptions~\ref{ass:behavior-exploration},~\ref{ass:ergodicity}, and~\ref{ass:h} hold. One has 
\begin{align*}
     \sum_{j=1}^k \ls| C_j^{(k)} \rs| \leq \frac{2\gamma |\calA|^2 \eta}{\alpha\lambda\tilde \sigma_b(1-\gamma)^4}\quad\mbox{and}\quad
     \sum_{j=1}^k \ls| D_j^{(k)} \rs| \leq \frac{|\calA|^2\eta}{\alpha\lambda\tilde \sigma_b(1-\gamma)^3}.
        \end{align*}
\end{lemma}

\begin{lemma}[Bound for term with $E_j^{(k)}$]\label{lem:expected-Ejk-bounds} Suppose that Assumptions~\ref{ass:behavior-exploration},~\ref{ass:ergodicity}, and~\ref{ass:h} hold. One has 
     \begin{align*}
        \sum_{j=1}^k \ls| E_j^{(k)} \rs| \leq \frac{2\gamma |\calA|^2 \eta}{\alpha\lambda\tilde\sigma_b^2 (1-\gamma)^4}.
    \end{align*}
\end{lemma}

\begin{lemma}[Bound for term with $F_j^{(k)}$]\label{lem:expected-Fjk-bounds} Suppose that Assumptions~\ref{ass:behavior-exploration} and~\ref{ass:ergodicity} hold. Recall the average weight $c_t^k$ in equation~\eqref{eq:Expected-TD-AC-PMD:average-weight}. Assume $B_k=B$, $k=0,\dots,K$. One has 
\begin{align*}
\sum_{j=1}^k\ls|\mathbb{E}\ls[F_j^{(k)}\rs]\rs|\leq \frac{2}{(1-\gamma)\tilde{\sigma}_b}\Psi({B},\vartheta,m_b,\kappa_b), 
\end{align*}
where 
\begin{align*}
 \Psi(B, \vartheta, m_b, \kappa_b) := \begin{cases}
            \dfrac{m_b}{(1-\gamma)(\kappa_b-\vartheta)} \cdot {\kappa_b^{B}}, & \mbox{if } 0 \leq \vartheta < \kappa_b, \\[1em]
            \dfrac{m_b}{(1-\gamma)} B \cdot \vartheta^{B - 1}, & \mbox{if } \vartheta = \kappa_b, \\[1em]
            \dfrac{m_b}{(1-\gamma)(\vartheta-\kappa_b)} \cdot {\vartheta^{B}}, & \mbox{if } 1 > \vartheta > \kappa_b, \\[1em]
            \dfrac{m_b}{B(1-\gamma)(1-\kappa_b)}, &\mbox{if } \vartheta = 1.
        \end{cases}
    \end{align*}
\label{lem:Expected-TD-AC-PMD:error-control}
\end{lemma}

Based on the above results, we are now ready to present the first sample complexity for Expected TD-PMD.
\begin{theorem}\label{thm:Expected-TD-AC-PMD:constant-step-size}
    Consider Expected TD-PMD \textup{(}Algorithm~\ref{alg:Expected-TD-AC-PMD}\textup{)}. Suppose that Assumptions~\ref{ass:behavior-exploration},~\ref{ass:ergodicity}, and~\ref{ass:h} hold. Let 
    \[
        \eta_k = \eta = \sqrt{\frac{\alpha\lambda\tilde\sigma_b^2 (1-\gamma)^4 \| D^{\pi^*}_{\pi_0} \|_\infty}{6|\calA|^2 (K+1)}}, \quad \alpha_k = \alpha \in (0,1], \quad B_k = B > 0,
    \]
    and assume $c_t^k$ takes the form in equation~\eqref{eq:Expected-TD-AC-PMD:average-weight}. Fixing any optimal policy $\pi^* \in \Pi^*$ and $\mu \in \Delta(\calS)$, there holds
    \begin{align*}
        \E\ls[ V^*(\mu) - V^{\pi_{\hat K}}(\mu) \rs] &\leq \frac{1}{K+1} \cdot \frac{3}{\alpha\tilde\sigma_b (1-\gamma)^3} + \frac{2}{\sqrt{K+1}} \cdot \ls( \frac{6|\calA|^2 \| D^{\pi^*}_{\pi_0} \|_\infty}{\alpha\lambda \tilde\sigma_b^2(1-\gamma)^6} \rs)^{\nicefrac{1}{2}} + \frac{2}{\tilde\sigma_b(1-\gamma)^2} \Psi(B, \vartheta, m_b, \kappa_b).
    \end{align*}
 
\end{theorem}

\begin{proof}
First recall the error decomposition in Lemma~\ref{lem:Expected-TD-AC-PMD:Lan-Xiao-bound}. 
For the stochastic variance term, it is easy to see that 
\begin{align*}
    \frac{1}{K+1} \cdot \frac{\eta |\calA|^2}{2(1-\gamma)^2\lambda} \cdot \E\ls[ \sum_{k=0}^K \| Q^{k} - Q^{\pi_k} \|_\infty^2 \rs] \leq \frac{\eta|\calA|^2}{2(1-\gamma)^4 \lambda}
\end{align*}
since $\| Q^k - Q^{\pi_k} \|_\infty \leq (1-\gamma)^{-1}$. For the stochastic bias term, based on equation \eqref{eq:expected-B0k-bound} and Lemmas~\ref{lem:decomposition of bias}--\ref{lem:expected-Fjk-bounds}, one has 
\begin{align*}
    &\mathbb{E}\left[\sum_{k=0}^K\mathbb{E}_{s\sim d_\mu^*}[\langle \pi^*(\cdot|s)-\pi_k(\cdot|s),Q^{\pi_k}(s,\cdot)-Q^{k}(s,\cdot)\rangle]\right]\\
    &=\mathbb{E}_{s\sim d_\mu^*}\left[\sum_{k=0}^K\mathbb{E}[\langle \pi^*(\cdot|s)-\pi_k(\cdot|s),Q^{\pi_k}(s,\cdot)-Q^{k}(s,\cdot)\rangle]\right]\\
    &\leq \mathbb{E}_{s\sim d_\mu^*}\left[\sum_{k=0}^K\left|\mathbb{E}[\langle \pi^*(\cdot|s)-\pi_k(\cdot|s),Q^{\pi_k}(s,\cdot)-Q^{k}(s,\cdot)\rangle]\right|\right]\\
    &\leq \mathbb{E}_{s\sim d_\mu^*}\left[\sum_{k=0}^K\left(\ls|B_0^{(k)}\rs|+\sum_{j=1}^k\mathbb{E}[\ls|C_j^{(k)}\rs|+\ls|D_j^{(k)}\rs|+\ls|E_j^{(k)}\rs|]+\sum_{j=1}^k\ls|\mathbb{E}\ls[F_j^{(k)}\rs]\rs|\right)\right]\\
    &\leq \mathbb{E}_{s\sim d_\mu^*}\left[\sum_{k=0}^K \frac{2}{1-\gamma}(1-\alpha(1-\gamma)\tilde{\sigma}_b)^k\right]+\mathbb{E}_{s\sim d_\mu^*}\left[\sum_{k=0}^K\frac{2\gamma |\mathcal{A}|^2\eta}{\alpha\tilde{\sigma}_b\lambda(1-\gamma)^4}\right]\\
    &\;\;\;\;+\mathbb{E}_{s\sim d_\mu^*}\left[\sum_{k=0}^K\frac{|\mathcal{A}|^2\eta}{\alpha\tilde{\sigma}_b\lambda(1-\gamma)^3}\right]+\mathbb{E}_{s\sim d_\mu^*}\left[\frac{2\gamma|\mathcal{A}|^2\eta}{\alpha\lambda\tilde{\sigma}_b^2(1-\gamma)^4}\right]\\
    &\;\;\;\;+\mathbb{E}_{s\sim d_\mu^*}\left[\frac{2}{(1-\gamma)\tilde{\sigma}_b}\Psi({B},\vartheta,m_b,\kappa_b)\right].
\end{align*}
It follows that 
\begin{align*}
    &\frac{1}{K+1}\cdot\frac{1}{1-\gamma}\mathbb{E}\left[\sum_{k=0}^K\mathbb{E}_{s\sim d_\mu^*}[\langle \pi^*(\cdot|s)-\pi_k(\cdot|s),Q^{\pi_k}(s,\cdot)-Q^{k}(s,\cdot)\rangle]\right]\\
    &\leq \frac{1}{K+1}\cdot\frac{2}{\alpha\tilde{\sigma}_b(1-\gamma)^3}+\frac{5|\mathcal{A}|^2\eta}{\alpha\tilde{\sigma}_b^2\lambda(1-\gamma)^5}+\frac{2}{(1-\gamma)^2\tilde{\sigma}_b}\Psi({B},\vartheta,m_b,\kappa_b).
\end{align*}
Putting them all together {in Lemma~\ref{lem:Expected-TD-AC-PMD:Lan-Xiao-bound}} yields 
\begin{align*}
    \mathbb{E}\left[\frac{1}{K+1}\sum_{k=0}^K\left(V^*(\mu)-V^{\pi_k}(\mu)\right)\right] &\leq \frac{1}{K+1}\left[ \frac{D_{\pi_{0}}^{\pi^*}(d_\mu^*)}{{\eta}(1-\gamma)}+\frac{3}{\alpha\tilde{\sigma}_b(1-\gamma)^3}\right]+\frac{6|\mathcal{A}|^2{\eta}}{\alpha\tilde{\sigma}_b^2\lambda(1-\gamma)^5}\\
    &\;\;\;\;+\frac{2}{(1-\gamma)^2\tilde{\sigma}_b}\Psi({B},\vartheta,m_b,\kappa_b).
\end{align*}
Choose the best $\eta$ such that the minimal is achieved completes the proof.
\end{proof}

\begin{remark}
     Assume $0<\vartheta<\kappa_b$. By Theorem~\ref{thm:Expected-TD-AC-PMD:constant-step-size}, Expected TD-PMD can find an $\varepsilon$-optimal solution provided the iteration complexity and batch-size satisfying
    \begin{align*}
        K \asymp \max \ls\{ \frac{9}{\alpha \tilde\sigma_b (1-\gamma)^3 \varepsilon}, \; \frac{216 |\calA|^2 \| D^{\pi^*}_{\pi_0} \|_\infty}{\alpha\lambda\tilde\sigma_b^2 (1-\gamma)^6 \varepsilon^{2}} \rs\}, \quad B \asymp \ls( \log \kappa_b^{-1}\rs)^{-1} \cdot \log \ls(\frac{6}{\varepsilon} \cdot \frac{m_b}{\tilde\sigma_b(1-\gamma)^3(\kappa_b-\vartheta)}\rs).
    \end{align*}
    Thus, the total sample complexity is $\tilde{O}(\varepsilon^{-2})$, where some constants and logarithm terms are hidden in $\tilde{O}(\cdot)$. In particular, if we restrict the mirror map $h$ to be the negative entropy and let $\pi_b$ be the uniform policy, then $\lambda = 1$ and $\| D^{\pi^*}_{\pi_0} \|_\infty \leq \log |\calA|$. By further setting $\alpha = 1$, the overall sample complexity becomes
    \begin{align*}
        S = KB = O\ls( \varepsilon^{-2} \tilde\sigma_b^{-2} (1-\gamma)^{-6} |\calA|^2 \log |\calA| \times \ls( \log \kappa_b^{-1}\rs)^{-1} \cdot \log \ls(\varepsilon^{-1} m_b \tilde\sigma_b^{-1}(1-\gamma)^{-3}(\kappa_b-\vartheta)^{-1}\rs) \rs).
    \end{align*}
    Compared to the $O(\varepsilon^{-2}(1-\gamma)^{-5}|\calS||\calA| \log |\calA| (\log \gamma^{-1})^{-1} (\log \varepsilon^{-1})) \asymp O(\varepsilon^{-2}(1-\gamma)^{-6}|\calS||\calA|\log |\calA| (\log \varepsilon^{-1}))$ sample complexity of PMD with generative model in~\textup{\citet[Proposition~2]{Lan_2021}}, the sample complexity of Expected TD-PMD has same dependence on the accuracy $\varepsilon$ and similar dependence on the decision horizon $(1-\gamma)^{-1}$ and the state/action space size. Note that the dependence on $\tilde\sigma_b$, $m_b$, and $\kappa_b$ is for the Markov sampling to sufficiently explore the whole state-action space $\calS\times\calA$, while the exploration is not an issue for generative model.
    \label{rem:Expected-TD-AC-PMD:constant-step-size}
\end{remark}

\begin{remark}
    If we set $\vartheta = 1$, then the average weight in equation~\eqref{eq:Expected-TD-AC-PMD:average-weight} reduces to $ \nicefrac{1}{B}$, and $\bar\delta_k$ reduces to the arithmetic mean of $\delta_t^k$. In such setting, the sample complexity of Expected TD-PMD is $\tilde{O}(\varepsilon^{-3})$. In the next section, we will establish the $O(\varepsilon^{-2})$ \textup{(}without hidden logarithm term\textup{)} last-iterate sample complexity for the case $\vartheta=1$ by using an adaptive policy update step size. 
\end{remark}

\begin{remark}
\label{rem:entire-mixing}
    The careful reader may find out that in our analysis, we do not use the stationary property of the entire Markov chain induced by $\pi_b$, but only use the stationary property of every segment. 
    It may be possible to improve the sample complexity on the dependency of other parameters, but the $\varepsilon^{-2}$ sample complexity on $\varepsilon$ overall cannot be improved, as far as we know. In addition, we will consider a mixed policy setting in Section~\ref{sec:Approximated-TD-AC-PMD}, which involves Markov chains induced by varying policies in contrast to the static $\pi_b$. Thus, it is convenient to keep the analysis within a similar framework.
\end{remark}

\subsection{Expected TD-PMD with adaptive step sizes}
\label{sec:Expected-TD-AC-PMD-with-adaptive-step-size}
As stated above, this section focuses on the later iteration complexity of Expected TD-PMD with adaptive step sizes when $\vartheta=1$. In this case, it can be shown the $O(\varepsilon^{-2})$ sample complexity without a logarithm term can be established. Moreover, the mirror map $h$ is not required to be strongly convex in this section. 

 The main idea is to approximate the classic value iteration algorithm (VI)~\citep{suttonRL} by using a sufficiently large adaptive policy update step size. More precisely, recall the policy  update formula of Expected TD-PMD
\begin{align*}
    \forall\, k: \quad \pi_{k+1}(\cdot|s) = \underset{p\in\Delta(\calA)}{\arg\max} \; \ls\{ \eta_k \ls\langle p, \, Q^k(s,\cdot) \rs\rangle - D^{p}_{\pi_k}(s) \rs\}.
\end{align*}
Notice that when $\eta_k \to \infty$, the updated policy $\pi_{k+1}$ tends to be greedy, i.e., choosing the action with largest critic value
\begin{align*}
    \pi_{k+1} \stackrel{\eta_k\to\infty}{\approx} \pi^k_{\mbox{\tiny greedy}}: \quad \pi_{\mbox{\tiny greedy}}^k(a^\prime|s) = 0 , \quad \forall\, s\in\calS, \;\forall\, a^\prime \not\in\underset{a\in\calA}{\arg\max} \; Q^k(s,a).
\end{align*}
Therefore, one has $\calF^{\pi^k_{k+1}} Q^k \approx\calF Q^k$, and the critic update in equation~\eqref{eq:Expected-TD-AC-PMD:critic-update} becomes
\begin{align*}
     Q^{k+1} &\approx Q^k + \alpha_k \Sigma^{\pi_b} \ls[ \calF Q^k - Q^k \rs] + \alpha_k \bar\omega_k.
\end{align*}
Noting that $\Sigma^{\pi_b} \succeq \tilde\sigma_b \cdot I \succ \mathbf{0}$, the critic update above can be viewed as a weighted VI with step size $\alpha_k$.  Thus we can  bound $\| Q^* - Q^k \|_\infty$ directly by using the $\gamma$-contraction property of $\calF$. 

\begin{lemma}
    Suppose Assumption~\ref{ass:behavior-exploration} and Assumption~\ref{ass:ergodicity} hold. Consider Expected TD-PMD (Algorithm~\ref{alg:Expected-TD-AC-PMD}) with constant critic step size $\alpha_k = \alpha \in (0,1]$. There holds
    \begin{align*}
        \forall\, 0\leq k \leq K: \quad \ls\| Q^* - Q^{k+1} \rs\|_\infty \leq [1-(1-\gamma)\alpha\tilde\sigma_b] \cdot \ls\| Q^* - Q^k \rs\|_\infty + \frac{\alpha\gamma}{\eta_k} \| D^{\tilde\pi_k}_{\pi_k} \|_\infty + \alpha \| \bar\omega_k \|_\infty,
    \end{align*}
    where 
    \begin{align*}
        \ls\|D_{\pi_k}^{\tilde{\pi}_k} \rs\|_\infty :=\max_s D_{\pi_k}^{\tilde{\pi}_k}(s),
    \end{align*}
    and $\tilde{\pi}_k$ is any policy that satisfies $\langle\tilde{\pi}_k(\cdot|s),Q^k(s,\cdot)\rangle = \max_aQ^{k}(s,a),\;\forall s$.
    \label{lem:Expected-TD-AC-PMD:linear-critic-convergence}
\end{lemma}

Lemma~\ref{lem:Expected-TD-AC-PMD:linear-critic-convergence} establishes  the linear convergence rate of the critic error $\| Q^* - Q^k \|_\infty$ and the proof can be found in Section~\ref{sec:proof-critic-linear}. The following lemma shows that the error of the policy value can be bounded by the critic error.

\begin{lemma}
    Suppose Assumption~\ref{ass:behavior-exploration} and Assumption~\ref{ass:ergodicity} hold. Consider Expected TD-PMD \textup{(}Algorithm~\ref{alg:Expected-TD-AC-PMD}\textup{)} with constant critic step size $\alpha_k = \alpha \in (0,1]$. Then there holds
    \begin{align*}
        \forall\, 0 \leq k \leq K: \quad \ls\| Q^* - Q^{\pi_{k+1}} \rs\|_\infty \leq \frac{1}{\alpha(1-\gamma)\tilde\sigma_b} \ls( \ls\| Q^* - Q^{k+1} \rs\|_\infty + \ls\| Q^* - Q^k \rs\|_\infty + \alpha \| \bar\omega_k \|_\infty \rs).
    \end{align*}
    \label{lem:Expected-TD-AC-PMD:linear-actor-convergence}
\end{lemma}

\begin{proof}
    By the critic update formula (equation~\eqref{eq:Expected-TD-AC-PMD:critic-update}),
    \begin{align*}
        \| Q^{\pi_{k+1}} - Q^{k+1} \|_\infty &= \| {Q^{\pi_{k+1}}} - \calF^{\pi_{k+1}}_{\pi_b, \alpha}Q^k - \alpha \bar\omega_k \|_\infty \\
        &= \| \calF^{\pi_{k+1}}_{\pi_b, \alpha}{Q^{\pi_{k+1}}} - \calF^{\pi_{k+1}}_{\pi_b, \alpha}Q^k - \alpha \bar\omega_k \|_\infty \\
        &\leq \underbrace{\ls\| \calF^{\pi_{k+1}}_{\pi_b, \alpha} Q^{\pi_{k+1}} - \calF^{\pi_{k+1}}_{\pi_b, \alpha} Q^k \rs\|_\infty}_{\mbox{\small Applying Proposition~\ref{pro:Expected-TD-AC-PMD:sigma-bellman-property}}} + \alpha \ls\| \bar\omega_k \rs\|_\infty \\
        &\leq [1-(1-\gamma)\alpha\tilde\sigma_b] \| Q^{\pi_{k+1}} - Q^k \|_\infty + \alpha \| \bar\omega_k \|_\infty \\
        &\leq [1-(1-\gamma)\alpha\tilde\sigma_b] \ls( \| Q^* - Q^{\pi_{k+1}} \|_\infty + \| Q^* - Q^k \|_\infty \rs) + \alpha \| \bar\omega_k \|_\infty \\
        &\leq [1-(1-\gamma)\alpha\tilde\sigma_b] \| Q^* - Q^{\pi_{k+1}} \|_\infty + \| Q^* - Q^k \|_\infty  + \alpha \| \bar\omega_k \|_\infty.
    \end{align*}
    The proof is completed by noting that $\| Q^* - Q^{\pi_{k+1}} \|_\infty \leq \| Q^* - Q^{k+1} \|_\infty + \| Q^{\pi_{k+1}} - Q^{k+1} \|_\infty$.
\end{proof}

It remains to bound the stochastic error term $\| \bar\omega_k \|_\infty$, which is given in Lemma~\ref{lem:Expected-TD-AC-PMD:error-control-2} below. The proof of this lemma is deferred to Section~\ref{sec:proof-stochastic-error-2}. It is worth noting that  Lemma~\ref{lem:Expected-TD-AC-PMD:error-control} establishes the bound for $\ls\| \E \ls[ \bar\omega_k\, | \, \calG_k \rs]\rs\|_\infty$ but here we need to compute a bound for $\E\ls[ \ls\| \bar\omega_k \rs\|_\infty\rs]$.

\begin{lemma}
    Consider Expected TD-PMD \textup{(}Algorithm~\ref{alg:Expected-TD-AC-PMD}\textup{)}. Suppose Assumption~\ref{ass:ergodicity} hold. Let $c_t^k=1/B_k$. Then there holds
    \begin{align*}
        \forall\, 0 \leq k \leq K: \quad  \E \ls[ \ls\| \bar\omega_k \rs\|_\infty \rs] \leq \ls( \frac{4|\calS||\calA|}{(1-\gamma)^2} \ls( 1+\frac{m_b}{1-\kappa_b} \rs) \rs)^{\nicefrac{1}{2}} \cdot (B_k)^{-\nicefrac{1}{2}}.
    \end{align*}
    \label{lem:Expected-TD-AC-PMD:error-control-2}
\end{lemma}

\begin{theorem}
    Suppose Assumption~\ref{ass:behavior-exploration} and Assumption~\ref{ass:ergodicity} hold. Consider Expected TD-PMD \textup{(}Algorithm~\ref{alg:Expected-TD-AC-PMD}\textup{)} with constant critic step size $\alpha_k = \alpha \in (0,1]$ and the adaptive policy update step size
    \begin{align*}
        \eta_k \geq \eta \cdot \| D^{\tilde\pi_k}_{\pi_k} \|_\infty \quad \mbox{with} \quad \eta > 0.
    \end{align*}
    Let $c_t^k=1/B_k$. Then there holds
    \begin{align*}
        \E\ls[ \| Q^* - Q^{\pi_K} \|_\infty \rs] &\leq \frac{2}{\alpha\tilde\sigma_b(1-\gamma)^2} [1-(1-\gamma)\alpha\tilde\sigma_b]^{K-1} + \frac{2\gamma}{\alpha\eta(1-\gamma)^2 \tilde\sigma_b^2} \\
        &+ \frac{1}{\tilde\sigma_b(1-\gamma)} \ls( \frac{4|\calS||\calA|}{(1-\gamma)^2} \ls( 1+\frac{m_b}{1-\kappa_b} \rs) \rs)^{\nicefrac{1}{2}} \ls( \Xi(K) + \Xi(K-1) + (B_{K-1})^{-\nicefrac{1}{2}} \rs),
    \end{align*}
    where $\Xi(t) := \sum_{k=0}^{t-1} [1-(1-\gamma)\alpha\tilde\sigma_b]^{t-1-k} (B_k)^{-\nicefrac{1}{2}}$.
    \label{thm:Expected-TD-AC-PMD:adaptive-step-size}
\end{theorem}

\begin{proof}
    Taking an expectation on both sides of the inequality in Lemma~\ref{lem:Expected-TD-AC-PMD:linear-critic-convergence} yields 
    \begin{align*}
        \forall\, 0 \leq k \leq K: \quad &\phantom{=\,\,\,}\E \| Q^* -Q^{k+1} \|_\infty \\
        &\leq [1-(1-\gamma)\alpha\tilde\sigma_b] \cdot \E\| Q^* - Q^k \|_\infty + \underbrace{\E\ls[\frac{\alpha\gamma}{\eta_k} \| D^{\tilde\pi_k}_{\pi_k} \|_\infty\rs]}_{\mbox{\small Applying } \eta_k \geq \eta \| D^{\tilde\pi_k}_{\pi_k} \|_\infty} + \alpha \underbrace{\E\ls\| \bar\omega_k \rs\|_\infty}_{\mbox{\small Applying Lemma~\ref{lem:Expected-TD-AC-PMD:error-control-2}}} \\
        &\leq [1-(1-\gamma) \alpha\tilde\sigma_b] \cdot \E\ls\| Q^* - Q^k \rs\|_\infty + \alpha\gamma\eta^{-1} + \alpha C_1 (B_k)^{-\nicefrac{1}{2}},
    \end{align*}
    where $C_1 = \ls( \frac{4|\calS||\calA|}{(1-\gamma)^2} \ls( 1+\frac{m_b}{1-\kappa_b} \rs) \rs)^{\nicefrac{1}{2}}$. Iterating from $K-1$ to $0$ we obtain
    \begin{align*}
        \E\| Q^* - Q^{K} \|_\infty &\leq [1-(1-\gamma)\alpha\tilde\sigma_b]^K \cdot \| Q^* - Q^0 \|_\infty  \\
        & \;\;\;\; + \alpha \sum_{k=0}^{K-1} [1-(1-\gamma)\alpha\tilde\sigma_b]^{K-1-k} [\gamma\eta^{-1} + C_1 (B_k)^{-\nicefrac{1}{2}}] \\
        &\leq [1-(1-\gamma)\alpha\tilde\sigma_b]^K \cdot \frac{1}{1-\gamma} + \frac{\gamma}{\eta(1-\gamma)\tilde\sigma_b} + \alpha C_1 \Xi(K). \numberthis \label{eq:Expected-TD-AC-PMD:adaptive-stepsize-critic-convergence}
    \end{align*}
    Similarly, there holds
    \begin{align*}
        \E\| Q^* - Q^{K-1} \|_\infty \leq [1-(1-\gamma)\alpha\tilde\sigma_b]^{K-1} \cdot \frac{1}{1-\gamma} + \frac{\gamma}{\eta(1-\gamma)\tilde\sigma_b} + \alpha C_1 \Xi(K-1).
    \end{align*}
    On the other hand, taking an expectation on both sides of the inequality in Lemma~\ref{lem:Expected-TD-AC-PMD:linear-actor-convergence} yields
    \begin{align*}
        \E\ls\| Q^* - Q^{\pi_{K}} \rs\|_\infty &\leq \frac{1}{\alpha(1-\gamma)\tilde\sigma_b} \Big( \E\ls\| Q^* - Q^{K} \rs\|_\infty + \E \ls\| Q^* - Q^{K-1} \rs\|_\infty + \alpha \underbrace{\E\ls\| \bar\omega_{K-1} \rs\|_\infty}_{\mbox{\small Applying Lemma~\ref{lem:Expected-TD-AC-PMD:error-control-2}}} \Big) \\
        & \leq \frac{1}{\alpha(1-\gamma)\tilde\sigma_b} \Big( \E\ls\| Q^* - Q^{K} \rs\|_\infty + \E \ls\| Q^* - Q^{K-1} \rs\|_\infty + \alpha C_1(B_{K-1})^{-\nicefrac{1}{2}} \Big).
    \end{align*}
    Inserting the  above bounds for $\E\| Q^* - Q^K \|_\infty$ and $\E\| Q^* - Q^{K-1} \|_\infty$ completes the proof.
\end{proof}

\begin{remark}
    Theorem~\ref{thm:Expected-TD-AC-PMD:adaptive-step-size} implies that $\E\| Q^* - Q^{\pi_K} \|_\infty \leq \varepsilon$ provided
    \begin{gather*}
        K = \log \ls( \frac{6}{\alpha\tilde\sigma_b(1-\gamma)^2 \varepsilon} \rs) \ls[ \log \ls( \frac{1}{1- (1-\gamma)\alpha\tilde\sigma_b^2} \rs) \rs]^{-1} + 1, \quad
        \eta = \frac{6\gamma}{\alpha(1-\gamma)^2\tilde\sigma_b^2 \varepsilon}, \\
        B_k = \frac{81}{\tilde\sigma_b^2(1-\gamma)^2} \ls( \frac{4|\calS||\calA|}{(1-\gamma)^2}\ls( 1+\frac{m_b}{1-\kappa_b} \rs) \rs) \frac{1}{\ls( 1-[1-(1-\gamma)\alpha\tilde\sigma_b]^{\nicefrac{1}{2}} \rs)^2} \times [1-(1-\gamma)\alpha\tilde\sigma_b]^{K-k-2} \varepsilon^{-2}.
    \end{gather*}
    Thus the total sample complexity is 
    \begin{align*}
        S = \sum_{k=0}^K B_k &\leq \frac{81}{\tilde\sigma_b^2(1-\gamma)^2} \ls( \frac{4|\calS||\calA|}{(1-\gamma)^2}\ls( 1+\frac{m_b}{1-\kappa_b} \rs) \rs) \frac{1}{\ls( 1-[1-(1-\gamma)\alpha\tilde\sigma_b]^{\nicefrac{1}{2}} \rs)^2} \times \frac{[1-(1-\gamma)\alpha\tilde\sigma_b]^{-2}}{(1-\gamma)\alpha\tilde\sigma_b} \varepsilon^{-2} \\
        &= O\ls( (1-\gamma)^{-7} |\calS||\calA| \alpha^{-3} \tilde\sigma_b^{-5} \varepsilon^{-2} \cdot m_b (1-\kappa_b)^{-1} \rs)
    \end{align*}
    which does not contain an logarithmic term. In contrast, existing sample complexities for PMD type methods under both the generative model ~\textup{\citep{Lan_2021,Xiao_2022,Johnson_Pike-Burke_Rebeschini_2023,protopapas2024policy,liu2025tdpmd}} and the Markov  sampling ~\textup{\citep{li2025policy,li2024stochastic}} are all $O(\varepsilon^{-2}\log(1/\varepsilon))$.
    \label{rem:Expected-TD-AC-PMD:adaptive-step-size}
\end{remark}

\begin{remark}
Note that  the adaptive step sizes $\eta_k$ in Theorem~\ref{thm:Expected-TD-AC-PMD:adaptive-step-size} is computable. In particular, for some specific mapping function $h$, the upper bound of $\| D^{\tilde\pi_k}_{\pi_k} \|_\infty$ can be easily derived. For instance, for $h(p)=\frac{1}{2}\|p\|_2^2$ there holds $\| D^{\tilde\pi_k}_{\pi_k} \|_\infty \leq 1$ and for $h(p)=-\sum_a p(a) \log p(a)$ there holds $\| D^{\tilde\pi_k}_{\pi_k} \|_\infty \leq \log [(\min_{s,a} \, \pi_k(a|s))^{-1}]$.
\end{remark}

\subsubsection{Batch Q-learning}
As mentioned previously, Expected TD-PMD approaches  batch Q-learning  (see Algorithm~\ref{alg:batch-Q-learning}) with $\eta_k \to +\infty$. In particular, by picking $h(p)=\frac{1}{2}\| p \|_2^2$,  Proposition~\ref{pro:Expected-TD-AC-PMD:equivalence} below tells us that under an adaptive but finite policy update step size, Expected TD-PMD is exactly equivalent to batch Q-learning.

\begin{algorithm}[ht!]
    \small
    \caption{Batch Q-learning}
    \label{alg:batch-Q-learning}
\begin{algorithmic}
    \STATE {\bfseries Input:} Iterations $K$. Initial action value $Q^0=0$, critic step size $\alpha_k$, initial state $s_0$, batch size $B_k$, average weight $c_t^k$.
    \STATE Set $s_0^0 = s_0$.
    \FOR{$k=0,1, \dots, K-1$}
    \STATE \textbf{(Sampling)} Obtain $\tau_k$ by obeying behavior policy $\pi_b$ starting from $s_0^k$,
    \begin{align*}
        \tau_k = \{ (s_t^k, a_t^k, r^k_t, s_{t+1}^k) \}_{t=0}^{B_k-1}, \quad \mbox{where} \; r^k_t = r(s^k_t, a^k_t), \;\; a^k_t \sim \pi_b(\cdot|s_t^k), \;\; s^k_{t+1} \sim P(\cdot| s_t^k, a_t^k),
    \end{align*}
    and set $s_0^{k+1} = s_{B_k}^k$.
    \STATE \textbf{(Critic update)} Construct
    \begin{align*}
        &\delta_t^k(s,a) := \mathds{1}_{(s^k_t,a^k_t)=(s,a) } \cdot\ls[r_t^k + \gamma \max_{a\in\calA} \ls[Q^k(s_{t+1}^k, a)\rs] - Q^k(s^k_t, a^k_t)\rs], \\
        &\bar\delta_k(s,a) := \sum_{t=0}^{B_k-1} c_t^k\cdot \delta_t^k(s,a).
    \end{align*}
    \STATE Update the critic by
    \begin{align*}
        Q^{k+1}(s,a) = Q^k(s,a) + \alpha_k \cdot \bar\delta_k(s,a).
    \end{align*}
    \ENDFOR
    \STATE Output $Q^K$.
\end{algorithmic}
\end{algorithm}

\begin{lemma}[\protect{\citet[Lemma~7]{ppgliu}}]
    Let $\mathcal{B}$ and $\mathcal{C}$ be two disjoint non-empty sets such that $\mathcal{A} =\mathcal{B}\cup \mathcal{C}$. Given an arbitrary vector $
    p=\left( p_a \right) _{a\in \mathcal{A}}\in \mathbb{R} ^{\left| \mathcal{A} \right|}$, let $y=\mathrm{Proj}_{\Delta \left( \mathcal{A} \right)}\left( p \right)$. Then
    \[
    \forall a^\prime\in \mathcal{C},~ y_{a^\prime}=0 \Leftrightarrow \,\,\sum_{a\in \mathcal{B}}{\left( p_a-\underset{a^{\prime}\in \mathcal{C}}{\max}\,\,p_{a^{\prime}} \right) _+}\ge 1.
    \]
    \label{lem:proj-property}    
\end{lemma}

\begin{proposition}
    Consider Expected TD-PMD \textup{(}Algorithm~\ref{alg:Expected-TD-AC-PMD}\textup{)} with $h(p)=\frac{1}{2} \| p \|_2^2$. At each iteration $k$,  let $\calA^k_s := \arg\max_{a\in\calA} \, Q^k(s,a)$ for each  $s\in\calS$, and also define the following value gap
    \begin{align*}
        \varDelta_{k,s} := \begin{cases}
            \max_{a\in\calA} Q^k(s,a) - \max_{a^\prime\not\in\calA^k_s} Q^k(s,a^\prime) & \mbox{if } \calA^k_s \neq \calA, \\
            +\infty & \mbox{if } \calA^k_s = \calA.
        \end{cases}
    \end{align*}
    If the policy update step size of Expected TD-PMD satisfies
    \begin{align*}
        \forall\, k\geq 0 : \quad \eta_k \geq 2\cdot\max_{s\in\calS} \; (\varDelta_{k,s})^{-1},\numberthis\label{eq:Qlearning-stepsize}
    \end{align*}
    then one has
    $
        \pi_{k+1} = \pi^k_{\mbox{\tiny greedy}}.
    $
    That is, Expected TD-PMD is equivalent to batch Q-learning under the step size condition in \eqref{eq:Qlearning-stepsize}.
    \label{pro:Expected-TD-AC-PMD:equivalence}
\end{proposition}

\begin{proof}
    Letting $h(p)=\frac{1}{2}\| p \|_2^2$, the policy update of Expected TD-PMD becomes
    \begin{align*}
        \forall\, s\in\calS: \quad \pi_{k+1}(\cdot|s) = \Proj_{\Delta(\calA)} \ls[ \pi_k(\cdot|s) + \eta_k \cdot Q^k(s,\cdot) \rs].
    \end{align*}
    For the state such that $\calA^k_s = \calA$,  we have $\pi_{k+1}(\cdot|s)=\pi^k_{\mbox{\tiny greedy}}(\cdot|s)$ for any $\eta_ k > 0$. Consider the state that $\calA^k_s \neq \calA$. Let $\mathcal{B} = \calA^k_s$, $\mathcal{C} = \calA \setminus \calA^k_s$. We have
    \begin{align*}
        \forall\, a^\prime \in \mathcal{C}: \quad  &\sum_{a\in\mathcal{B}} \ls[ \pi_k(a|s) + \eta_k Q^k(s,a) - \pi_k(a^\prime | s) - \eta_k Q^k(s,a^\prime) \rs]_+ \\
        \geq &\sum_{a\in\mathcal{B}} \ls[ (\pi_k(a|s) - \pi_k(a^\prime | s)) + \eta_k \ls( Q^k(s,a) - Q^k(s,a^\prime) \rs) \rs] \\
        \geq &\sum_{a\in\mathcal{B}}\ls[ -1 + \eta_k \cdot \varDelta_{k,s} \rs] \geq \sum_{a\in\mathcal{B}} [-1+2] \geq 1.
    \end{align*}
    Therefore, by Lemma~\ref{lem:proj-property},
    \begin{align*}
        \forall\, a^\prime\in\mathcal{C}: \quad \pi_{k+1}(a^\prime | s) = \Proj_{\Delta(\calA)} \ls[ \pi_k(\cdot | s) + \eta_k\cdot Q^k(s,\cdot) \rs](a^\prime) = 0,
    \end{align*}
    which implies that $\pi_{k+1}(\cdot | s) = \pi^k_{\mbox{\tiny greedy}}(\cdot | s)$ and Expected TD-PMD reduces to batch Q-learning.
\end{proof}

By Proposition~\ref{pro:Expected-TD-AC-PMD:equivalence}, batch Q-learning can be viewed as a special instance of Expected TD-PMD with $h(p)=\frac{1}{2}\| p \|_2^2$ provided the policy update step size $\eta_k$ is large. Thus the analysis in Section~\ref{sec:Approximated-TD-AC-PMD-with-adaptive-step-size} also applies for batch Q-learning.

\begin{corollary}
    Suppose Assumption~\ref{ass:behavior-exploration} and Assumption~\ref{ass:ergodicity} hold. Consider batch Q-learning  with $\alpha_k=\alpha\in(0,1]$ and let $c_t^k = \nicefrac{1}{B_k}$ \textup{(}i.e., the form in equation~\eqref{eq:Expected-TD-AC-PMD:average-weight} with $\vartheta=1$\textup{)}. There holds
    \begin{align*}
        \E\ls[ \ls\| Q^* - Q^K \rs\|_\infty \rs] \leq \frac{1}{1-\gamma} \ls[ 1 - (1-\gamma) \alpha\tilde\sigma_b \rs]^{K} + \alpha \ls( \frac{4|\calS||\calA|}{(1-\gamma)^2} \ls( 1 + \frac{m_b}{1-\kappa_b} \rs) \rs)^{\nicefrac{1}{2}} \Xi(K),
    \end{align*}
    where $\Xi$ is defined in Theorem~\ref{thm:Expected-TD-AC-PMD:adaptive-step-size}.
    \label{cor:batch-Q-learning}
\end{corollary}
\begin{proof}
Notice that $\| D^{\tilde\pi_k}_{\pi_k} \|_\infty\leq 1$ when $h(p)=\frac{1}{2}\|p\|_2^2$. Thus, if we set $\eta_k = \max\{ 2\cdot\max_{s\in\calS} (\varDelta_{k,s})^{-1}, \; \eta \}\geq \eta$ for Expected TD-PMD, if follows from equation~\eqref{eq:Expected-TD-AC-PMD:adaptive-stepsize-critic-convergence} that
    \begin{align*}
        \E[\| Q^* - Q^{K} \|_\infty] \leq [1-(1-\gamma)\alpha\tilde\sigma_b]^K \cdot \frac{1}{1-\gamma} + \frac{\gamma}{\eta(1-\gamma)\tilde\sigma_b} + \alpha \ls( \frac{4|\calS||\calA|}{(1-\gamma)^2} \ls( 1 + \frac{m_b}{1-\kappa_b} \rs) \rs)^{\nicefrac{1}{2}} \Xi(K).
    \end{align*}
    Moreover, Proposition~\ref{pro:Expected-TD-AC-PMD:equivalence} implies that Expected TD-PMD is equivalent to batch Q-learning, thus the upper bound above is valid for batch Q-learning for any $\eta > 0$. Letting $\eta \to \infty$ concludes the proof.
\end{proof}

\begin{remark}
    It is worth noting that, since for (batch) Q-learning the policy is greedy, we do not need to obtain the error bound of the policy value through Theorem~\ref{thm:Expected-TD-AC-PMD:adaptive-step-size}. In fact we can transform the error bound from $\| Q^* - Q^{K} \|_\infty$ to $\| Q^* - Q^{\pi_K} \|_\infty$ through the error amplification lemma, see for example \textup{\citet[Lemma 1.11]{Agarwal2019ReinforcementLT}}.
\end{remark}

\begin{remark}
    Corollary~\ref{cor:batch-Q-learning} implies that $\E[\| Q^* - Q^K \|_\infty] \leq \varepsilon$ provided
    \begin{gather*}
        K = \log \frac{2}{\varepsilon(1-\gamma)} \cdot \ls[ \log \frac{1}{1-(1-\gamma)\alpha\tilde\sigma_b} \rs]^{-1}, \\ 
        B_k = \frac{16\alpha^2|\calS||\calA|}{\varepsilon^{2}(1-\gamma)^2} \ls( 1 + \frac{m_b}{1-\kappa_b} \rs) [1-(1-\gamma)\alpha\tilde\sigma_b]^{K-k-1} \frac{1}{\ls(1-[1-(1-\gamma)\alpha\tilde\sigma_b]^{\nicefrac{1}{2}}\rs)^2}.
    \end{gather*}
    Thus the total sample complexity is
    \begin{align*}
        S = \sum_{k=0}^{K-1} B_k &\leq \frac{16\alpha^2|\calS||\calA|}{\varepsilon^{2}(1-\gamma)^2} \ls( 1 + \frac{m_b}{1-\kappa_b} \rs) \frac{1}{(1-\gamma)\alpha\tilde\sigma_b} \frac{1}{\ls(1-[1-(1-\gamma)\alpha\tilde\sigma_b]^{\nicefrac{1}{2}}\rs)^2} \\
        &= O \ls( (1-\gamma)^{-5} |\calS||\calA| \alpha^{-1} \tilde\sigma_b^{-3} \varepsilon^{-2} \cdot m_b (1-\kappa_b)^{-1} \rs).
    \end{align*}
    Considering the dependency on $(1-\gamma)$ and $\varepsilon$, the sample complexity of batch Q-learning is $O((1-\gamma)^{-5} \varepsilon^{-2})$. In~\textup{\citet{Li2021IsQM}}, a tight sample complexity of $O((1-\gamma)^{-4}\varepsilon^{-2}|\mathcal{S}||\mathcal{A}| \log^2 ((1-\gamma)^{-1}\varepsilon^{-1}))$ is established for Q-learning \textup{(}i.e., batch Q-learning with $B_k=1$\textup{)} with a small critic step size $\alpha=\tilde{O}(\varepsilon^{-2})$. In contrast, we have established the sample complexity of Q-learning with absolute constant critic step size by utilizing batch.  Though their result has a better dependency on $(1-\gamma)$, our sample complexity improves it by a logarithm  factor. Moreover, it is interesting to see whether the sample complexity of batch Q-learning can be further improved to match that in  \textup{\citet{Li2021IsQM}} or not. As discussed in Remark~\ref{rem:entire-mixing}, maybe this can be achieved by considering the stationary property of the entire Markov chain.
\end{remark}


\section{Approximate TD-PMD with Mixed-Policy Markov Data}
\label{sec:Approximated-TD-AC-PMD}
Notice that in Expected TD-PMD, it requires to compute the expectation $\E_{a \sim \pi_{k+1}(\cdot|s_t^k)} [Q^k(s^k_t, a)]$ when reconstructing  the TD error.  Thus, it is natural to ask what about we replacing the expectation with a sample mean? In this section we will investigate this problem.

In this setting, the trajectory $\tau_k$ admits the following form 
\[
\tau_k = (s_t^k, a_t^k, r_t^k, s_t^{\prime \, k}, a_t^{\prime \, k})_{t=0}^{B_k-1},
\]
where 
\begin{align*}
    a_t^k \sim \pi_b(\cdot | s_t^k), \;\; r_t^k = r(s_t^k, a_t^k), \;\; s^{\prime \, k}_t \sim P(\cdot | s_t^k, a_t^k), \;\; a_t^{\prime \, k} \sim \pi_{k+1}(\cdot|s^{\prime \, k}_t), \;\; s_{t+1}^k \sim P(\cdot|s_t^{\prime \, k}, a_t^{\prime \, k}).
\end{align*}
That is, $a^k_t$ is the action sampled from a behavior policy $\pi_b$, while $a^{\prime \, k}_t$ is the action sampled from the target policy $\pi_{k+1}$.  
Moreover, the TD error is constructed by
\begin{align*}
    \delta_t^k(s,a) = \mathds{1}_{(s^k_t,a^k_t)=(s,a) } \cdot\ls[r_t^k + \gamma\ls[Q^k(s_t^{\prime \, k}, a_t^{\prime \, k})\rs] - Q^k(s^k_t, a^k_t)\rs].
\end{align*}
Based on the new TD error, we obtain the {Approximate TD-PMD} algorithm, see Algorithm~\ref{alg:Approximated-TD-AC-PMD} for a detailed description with the illustration presented in Figure~\ref{fig:Approximated-TD-AC-PMD}. It is evident that Approximate TD-PMD is an online mixed-policy learning algorithm (off-policy for sampling an action to update the value while on-policy for estimating the expectation). As will be shown next, similar sample complexities can be established for Approximate TD-PMD. 

\begin{algorithm}[ht!]
    \small
    \caption{Approximate TD-PMD with Mixed-Policy Markov Data}
    \label{alg:Approximated-TD-AC-PMD}
\begin{algorithmic}
    \STATE {\bfseries Input:} Iterations $K$. Initial action value $Q^0=0$, initial policy $\pi_0$, critic step size $\alpha_k$, policy update step size $\eta_k$, initial state $s_0$, batch size $B_k$, average weight $c_t^k$.
    \STATE Set $s_0^0 = s_0$.
    \FOR{$k=0,1, \dots, K$}
    \STATE \textbf{(Policy update)} Update the target policy by
    \begin{align*}
        \forall\, s\in\calS: \quad \pi_{k+1}(\cdot|s) = \underset{p\in\Delta(\calA)}{\arg\max} \; \ls\{ \eta_k \ls\langle p ,\, Q^k(s,\cdot) \rs\rangle - D^p_{\pi_k}(s) \rs\}.
    \end{align*}
    \STATE \textbf{(Sampling)} Starting from $s_0^k$. obtain $\tau_k$ by
    \begin{align*}
        &\tau_k = \{ (s_t^k, a_t^k, r^k_t, s_{t}^{k\,\prime}, a_t^{k\,\prime}) \}_{t=0}^{B_k-1}, \\
        &\mbox{where} \; a_t^k \sim \pi_b(\cdot | s_t^k), \;\; r_t^k = r(s_t^k, a_t^k), \;\; s^{\prime \, k}_t \sim P(\cdot | s_t^k, a_t^k), \;\; a_t^{\prime \, k} \sim \pi_{k+1}(\cdot|s^{\prime \, k}_t), \;\; s_{t+1}^k \sim P(\cdot|s_t^{\prime \, k}, a_t^{\prime \, k}),
    \end{align*}
    and set $s_0^{k+1} = s_{B_k}^k$.
    \STATE \textbf{(Critic update)} Construct the approximate TD error,
    \begin{align*}
        &\delta_t^k(s,a) = \mathds{1}_{(s^k_t,a^k_t)=(s,a) } \cdot\ls[r_t^k + \gamma\ls[Q^k(s_t^{\prime \, k}, a_t^{\prime \, k})\rs] - Q^k(s^k_t, a^k_t)\rs], \\
        &\bar{\delta}_k(s,a) := \sum_{t=0}^{B_k-1} c_t^k \cdot \delta_t^k(s,a).
    \end{align*}
    \STATE Update the critic by
    \begin{align*}
        Q^{k+1}(s,a) = Q^k(s,a) + \alpha_k \cdot \bar{\delta}_k(s,a).
    \end{align*}
    \ENDFOR
    \STATE Sample $\hat K$ from $\{ 0, 1, \dots, K \}$ with uniform probability.
    \STATE Output $\pi_K$ and $\pi_{\hat K}$.
\end{algorithmic}
\end{algorithm}

\begin{figure}
    \centering
    \includegraphics[width=1\linewidth]{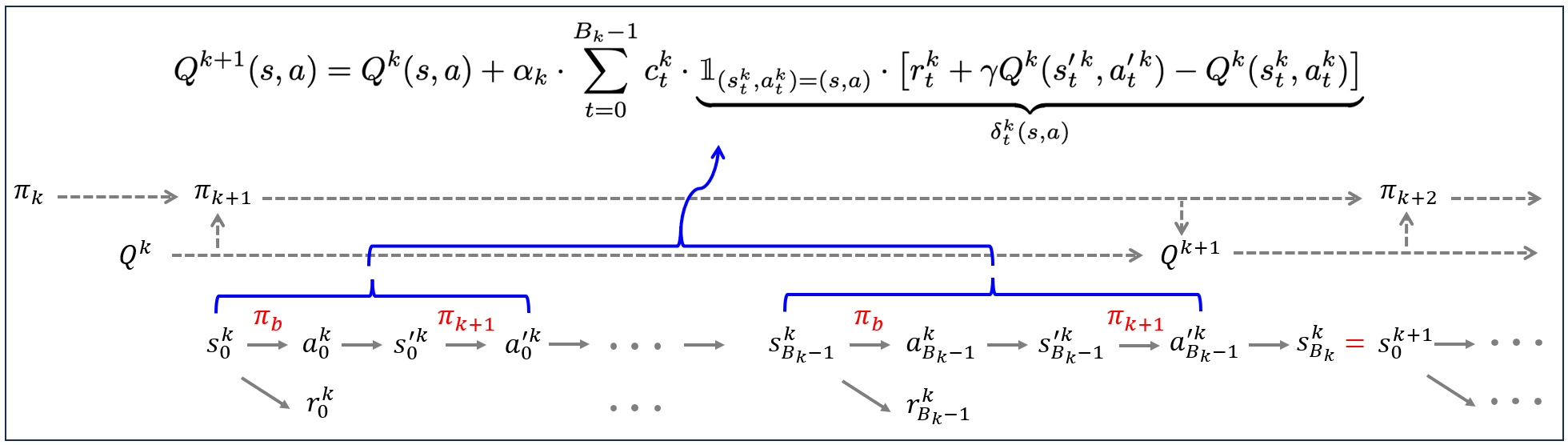}
    \caption{An illustration of Approximate TD-PMD with mixed-policy Markov data (Algorithm~\ref{alg:Approximated-TD-AC-PMD}).}
    \label{fig:Approximated-TD-AC-PMD}
\end{figure}

For Approximate TD-PMD,  the expectation of the critic update conditioned on $s_0^k$, $Q^k$, and $\pi_{k+1}$ is given by 
\begin{align}
    \E\ls[Q^{k+1}(s,a)\rs] = \E \ls[Q^k(s,a) + \alpha_k\cdot \sum_{t=0}^{B_k-1} c_t^k\cdot {\P(s_t^k=s) \pi_b(a|s)}\cdot {\ls[ \calF^{\pi_{k+1}} Q^k - Q^k \rs]}  \rs].
    \label{eq:Approximated-TD-AC-PMD:intuition}
\end{align}
In contrast to the expression in Expected TD-PMD, the term $\P(s_t^k=s)$ does not converge to the stationary distribution $\sigma^{\pi_b}(s)$ due to alternative  sampling. To obtain the stationary distribution of chain $\{ s_t^k \}_{t\geq 0}$ (note that the complete state chain is $\{ s_t^k, s_t^{\prime\,k} \}_{t\geq0}$), for any $\pi, \pi^\prime \in \Pi$, define the following composition transition matrix:
\begin{align*}
    P^{\pi \circ \pi^\prime}_{\scriptscriptstyle \calS} \in \mathbb{R}^{|\calS|\times|\calS|}: \quad P^{\pi\circ\pi^\prime}_{\scriptscriptstyle\calS}(s,\tilde s) = \sum_{a\in\calA,\, s^\prime \in \calS,\, a^\prime \in \calA} \pi(a|s) P(s^\prime | s,a) \pi^\prime(a^\prime | s^\prime) P(\tilde s| s^\prime, a^\prime) = [P^{\pi}_{\scriptscriptstyle\calS} P^{\pi^\prime}_{\scriptscriptstyle\calS}](s, \tilde s).
\end{align*}
One can verify that $P^{\pi_b \circ \pi_{k+1}}_{\scriptscriptstyle \calS}$ is the transition matrix of the state chain $\{ s_t^k \}_{t\geq0}$ when $\pi_{k+1}$ is given.  Thus, we make the following assumption for the ergodicity of $\{ s_t^k \}_{t\geq 0}$.
\begin{assumption}
    For any $\pi\in\Pi$, the state transition matrix $P^{\pi_b\circ\pi}_{\scriptscriptstyle\calS}$ is ergodic over $\calS$.
    \label{ass:ergodicity-2}
\end{assumption}
\noindent As the target policy $\pi_k$ varies over from iteration to iteration, here we assume  $P^{\pi_b\circ\pi}_{\scriptscriptstyle\calS}$ is ergodic for any policy $\pi \in \Pi$. This is an assumption that is common in the analysis of AC algorithms. In fact, in   existing works such as~\citet{khodadadian2022finite,li2024stochastic,wang2024non}, they assume the ergodicity of $P^{\pi}_{\scriptscriptstyle\calS}$ for any $\pi\in\Pi$. Under Assumption~\ref{ass:ergodicity-2}, a result  similar to Proposition~\ref{pro:MDP-ergodicity} can be obtained.

\begin{proposition}
    Suppose that Assumption~\ref{ass:ergodicity-2} holds. For any policy $\pi\in\Pi$, there exists a unique stationary distribution for transition kernel $P^{\pi_b\circ\pi}_{\scriptscriptstyle\calS}$,
    \begin{align*}
        \nu^{\pi_b\circ\pi} \in \Delta(\calS): \quad (\nu^{\pi_b\circ\pi})^\top P^{\pi_b\circ\pi}_{\scriptscriptstyle\calS} = (\nu^{\pi_b\circ\pi})^\top.
    \end{align*}
    Additionally, there exists a constant $\tilde\nu > 0$ such that 
    \begin{align*}
        \forall\, \pi \in \Pi, \; s\in\calS: \quad \nu^{\pi_b\circ\pi}(s) \geq \tilde\nu > 0.
    \end{align*}
    Furthermore, for any $\pi$, let $\{ s_t \}_{t\geq 0}$ be the Markov chain with transition kernel of $P^{\pi_b\circ\pi}_{\scriptscriptstyle\calS}$. Then the chain $\{ s_t \}_{t\geq0}$ is \textup{(}uniform\textup{)} geometrically mixing, i.e., there exists $m > 0$ and $\kappa \in (0,1)$ \textup{(}independent of $\pi$\textup{)} such that
    \begin{align*}
        \forall\, s\in\calS, \; n \geq 0, \; t \geq 0: \quad d_{TV}(\nu^{\pi_b\circ\pi}(\cdot), \, \P(s_{n+t} = \cdot | s_n = s)) \leq m \kappa^t.
    \end{align*}
    \label{pro:MDP-ergodicity-2}
\end{proposition}

Proposition~\ref{pro:MDP-ergodicity-2} can be proved similarly as ~\citet[Lemma~7]{khodadadian2022finite} and the details are omitted. At the $k$-th iteration of Approximate TD-PMD, the trajectory $\{ s_t^k \}_{t\geq0}$ is the Markov chain with transition kernel of $P^{\pi_b\circ\pi_{k+1}}_{\scriptscriptstyle\calS}$ when $\pi_{k+1}$ is fixed. Thus by Proposition~\ref{pro:MDP-ergodicity-2},
\begin{align}
    \forall\, 0\leq k \leq K, \; s\in\calS, \; n \geq 0, \; t \geq 0: \quad d_{TV}(\nu^{\pi_b\circ\pi_{k+1}}(\cdot), \, \P(s_{n+t}^k = \cdot | s_n^k = s) \; \big| \; \pi_{k+1}) \leq m \kappa^t.
    \label{eq:Approximated-TD-AC-PMD:MDP-ergodicity}
\end{align}
Let $\sigma^{\pi_b\circ\pi}(s,a) := \nu^{\pi_b\circ\pi}(s) \cdot \pi_b(a|s)$. One can directly verify that $\sigma^{\pi_b\circ\pi_{k+1}}$ is the stationary distribution of chain $\{ (s_t^k, a_t^k) \}_{t\geq0}$. By Proposition~\ref{pro:MDP-ergodicity-2}, we have $\P(s_t^k = s) \to \nu^{\pi_b\circ\pi_{k+1}}(s)$ and $\P(s^k_t=s, \, a^k_t=a) \to \sigma^{\pi_b\circ\pi_{k+1}}(s,a)$ as $B_k$ increases.

As in Expected TD-PMD (see Lemma~\ref{lem:Expected-TD-AC-PMD:bounded-value}),  the Q-values in Approximate TD-PMD are also bounded.
\begin{lemma}
    Consider Approximate TD-PMD \textup{(}Algorithm~\ref{alg:Approximated-TD-AC-PMD}\textup{)} with critic step size satisfying $\alpha_k \in (0,1]$. We have
    \begin{align*}
        \forall\, 0\leq k\leq K, \; (s,a) \in \calS\times\calA: \quad 0 \leq Q^k(s,a) \leq \frac{1}{1-\gamma}.
    \end{align*}
    \label{lem:Approximated-TD-AC-PMD:bounded-value}
\end{lemma}

In the sequel, we will extend the analysis for Expected TD-PMD to establish the sample complexity of Approximate TD-PMD under constant and adaptive policy update step sizes. The major difference is to handle the time-varying stationary distribution $\sigma^{\pi_b\circ\pi_{k+1}}$, which requires following lemma.

\begin{lemma}[\protect{\citet[Corollary~3.1]{Mitrophanov_2005}}]
    Let $\{ X_t \}_{t\geq0}$ and $\{ \tilde X_t \}_{t\geq0}$ be two Markov chains over $|\mathcal{X}|< \infty$ with transition kernels $P$ and $\tilde P$. Denote by $\nu$, $\tilde\nu$ the stationary distribution of $\{ X_t \}_{t\geq0}$, $\{ \tilde X_t \}_{t\geq0}$ respectively, and assume $\{ X_t \}_{t\geq0}$ is geometrically mixing, i.e.,
    \begin{align*}
        \exists\, m>0, \; \exists\, \kappa \in (0,1), \;\; \forall\, x\in \mathcal{X}, \;\; \forall\, n \geq 0 , \;\; \forall\, t\geq 0: \quad d_{TV} \ls( \nu(\cdot), \; \P(X_{n+t}=\cdot  | X_n = x) \rs) \leq m \kappa^{t}.
    \end{align*}
    Then there holds
    \begin{align*}
        d_{TV}(\nu(\cdot), \, \tilde\nu(\cdot)) \leq \ls( \ls\lceil \log_\kappa m^{-1} \rs\rceil + \frac{1}{1-\kappa} \rs) \| P - \tilde P \|_\infty.
    \end{align*}
    \label{lem:perturbated-chain-stationary-diff}
\end{lemma}

\subsection{Approximate TD-PMD with constant step size}
\label{sec:Approximated-TD-AC-PMD-with-constant-step-size}
Notice that the proof of Lemma~\ref{lem:Expected-TD-AC-PMD:Lan-Xiao-bound} does not rely on the sampling scheme and the critic update formula. Thus the upper bound in Lemma~\ref{lem:Expected-TD-AC-PMD:Lan-Xiao-bound} still applies for Approximate TD-PMD, which is presented below for readers' convenience. 

\begin{lemma}
    Consider Approximate TD-PMD \textup{(}Algorithm~\ref{alg:Approximated-TD-AC-PMD}\textup{)}. Suppose that Assumptions~\ref{ass:h} holds. Let $\eta_k = \eta > 0$ and $\alpha_k = \alpha \in (0,1]$. Fixing any optimal policy $\pi^*\in \Pi$ and $\mu \in \Delta(\calS)$, one has
    \begin{align*}
        \E\ls[ V^*(\mu) - V^{\pi_{\hat K}}(\mu) \rs] &\leq \underbrace{\frac{1}{K+1} \ls[ \frac{\| D^{\pi^*}_{\pi_0} \|_\infty}{\eta(1-\gamma)} + \frac{1}{(1-\gamma)^2}  \rs]}_{\mbox{ \small Exact PMD bound}} \\
        &\;\;\;\;+ \underbrace{\frac{1}{K+1} \cdot \frac{\eta |\calA|^2}{2(1-\gamma)^2\lambda} \cdot \E\ls[ \sum_{k=0}^K \| Q^{k} - Q^{\pi_k} \|_\infty^2 \rs]}_{\mbox{\small Stochastic Variance term}} \\
        &\;\;\;\;+ \underbrace{\frac{1}{K+1} \cdot \frac{1}{1-\gamma} \cdot \E \ls[ \sum_{k=0}^K \E_{s\sim d^*_\mu} \ls[ \ls\langle \pi^*(\cdot|s) - \pi_k(\cdot|s), \, Q^{\pi_k}(s,\cdot) - Q^k(s,\cdot) \rs\rangle \rs] \rs]}_{\mbox{\small Stochastic Bias term}}.
    \end{align*}
    \label{lem:Approximated-TD-AC-PMD:Lan-Xiao-bound}
\end{lemma}

Next, the key is still to bound the stochastic bias, which requires us to rewrite the update of the critic.  By Proposition~\ref{pro:MDP-ergodicity-2}, under Assumption~\ref{ass:ergodicity-2}, we have $\P(s_t^k = s) \to \nu^{\pi_b\circ\pi_{k+1}}(s)$. Thus it follows from equation~\eqref{eq:Approximated-TD-AC-PMD:intuition}that
\begin{align*}
    \E\ls[ Q^{k+1}(s,a) \rs] \approx \E\ls[ Q^k(s,a) + \alpha_k \cdot \sigma^{\pi_b\circ\pi_{k+1}}(s,a) \cdot \ls[ \calF^{\pi_{k+1}} Q^k - Q^k \rs](s,a) \rs].
\end{align*}
This inspires us to  define the $\sigma^{\pi_b\circ\pi}$-weighted Bellman operator with $\alpha \in (0,1]$,
\begin{align*}
    \forall\, \pi\in\Pi, \; Q\in\mathbb{R}^{|\calS||\calA|}: \quad  \calF^{\pi}_{\pi_b\circ\pi, \alpha} Q(s,a) := Q(s,a) + \alpha \cdot \sigma^{\pi_b\circ\pi}(s,a) \cdot \ls[ \calF^{\pi} Q(s,a) - Q(s,a) \rs],
\end{align*}
or equivalently,
\begin{align*}
    \calF^{\pi}_{\pi_b\circ\pi, \alpha} Q = Q + \alpha\Sigma^{\pi_b\circ\pi} \ls[ \calF^{\pi} Q - Q \rs],
\end{align*}
where $\Sigma^{\pi_b\circ\pi} := \mbox{diag}(\sigma^{\pi_b\circ\pi}) \in \mathbb{R}^{|\calS||\calA| \times |\calS||\calA|}$. When $\alpha = 1$ we simply denote $\calF^{\pi}_{\pi_b\circ\pi, \alpha}$ by $\calF^{\pi}_{\pi_b\circ\pi}$,
\begin{align*}
    \calF^\pi_{\pi_b\circ\pi} Q := Q + \Sigma^{\pi_b\circ\pi} [\calF^\pi Q - Q].
\end{align*}
By Proposition~\ref{pro:MDP-ergodicity-2}, we know that under Assumption~\ref{ass:ergodicity-2} and Assumption~\ref{ass:behavior-exploration} there holds
\begin{align*}
    \tilde\sigma := \min_{\pi\in\Pi} \, \min_{(s,a)\in\calS\times\calA} \; \sigma^{\pi_b\circ\pi}(s,a) > 0,
\end{align*}
implying that $\Sigma^{\pi_b\circ\pi} \succeq \tilde\sigma \cdot I$. Such a positive lower bound offers further implies the contraction property of $\calF^{\pi}_{\pi_b\circ\pi, \alpha}$.

\begin{proposition}
    Suppose Assumption~\ref{ass:behavior-exploration} and Assumption~\ref{ass:ergodicity-2} hold. Then for any $\pi \in \Pi$,
    \begin{enumerate}
        \item[\textup{(a)}] $\calF^{\pi}_{\pi_b\circ\pi, \alpha} Q^\pi = Q^\pi$\textup{;}
        \item[\textup{(b)}] $\forall\, Q, \, Q^\prime \in \mathbb{R}^{|\calS||\calA|}: \;\; \ls\| \calF^{\pi}_{\pi_b\circ\pi, \alpha} Q - \calF^{\pi}_{\pi_b\circ\pi, \alpha} Q^\prime \rs\|_\infty \leq (1-\alpha(1-\gamma)\tilde\sigma) \| Q - Q^\prime \|_\infty$.
    \end{enumerate}
    \label{pro:Approximated-TD-AC-PMD:sigma-bellman-property}
\end{proposition}
The proof is same with the one for Proposition~\ref{pro:Expected-TD-AC-PMD:sigma-bellman-property}. Similarly, letting $A^{\pi_b\circ\pi} = I - \alpha \Sigma^{\pi_b\circ\pi} (I - \gamma P^{\pi}_{\scriptscriptstyle\calS\times\calA})$, one has
\begin{align*}
    \calF^{\pi}_{\pi_b\circ\pi, \alpha} Q - \calF^{\pi}_{\pi_b\circ\pi,\alpha} Q^\prime = A^{\pi_b\circ\pi} (Q - Q^\prime),
\end{align*}
and thus
\begin{align*}
    \forall\, \pi \in \Pi: \quad \ls\| A^{\pi_b\circ\pi} \rs\|_\infty \leq (1-\alpha(1-\gamma)\tilde\sigma).
\end{align*}

Based on the operator $\calF^{\pi}_{\pi_b\circ\pi, \alpha}$, we can reformulate the critic update of Approximate TD-PMD as
\begin{equation}
\begin{aligned}
    Q^{k+1} = \calF^{\pi_{k+1}}_{\pi_b\circ\pi_{k+1}, \alpha} Q^k + \alpha \cdot \bar\omega_k \quad \mbox{or} \quad Q^{k+1} = Q^k + \alpha \cdot \Sigma^{\pi_b\circ\pi_{k+1}} \ls[ \calF^{\pi_{k+1}} Q^k - Q^k \rs] + \alpha \cdot \bar\omega_k,
\end{aligned}
    \label{eq:Approximated-TD-AC-PMD:critic-update}
\end{equation}
where $\bar{\omega}_k \in \mathbb{R}^{|\calS||\calA|}$ is the stochastic term defined as follows:
\begin{equation}
\begin{aligned}
    \omega_t^k(s,a) &= \delta_t^k(s,a) - \ls[ \calF^{\pi_{k+1}}_{\pi_b \circ \pi_{k+1}} Q^k - Q^k \rs](s,a), \\
    \bar\omega_k(s,a) &= \sum_{t=0}^{B_k-1} c_t^k \cdot \omega_t^k(s,a) = \bar\delta_k(s,a) - \ls[ \calF^{\pi_{k+1}}_{\pi_b \circ \pi_{k+1}} Q^k - Q^k \rs](s,a).
\end{aligned}
\label{eq:Approximated-TD-AC-PMD:stochastic-noise}
\end{equation}

In Approximate TD-PMD, we adopt the same  weights $c_t^k$ as in Expected TD-PMD, see equation~\eqref{eq:Expected-TD-AC-PMD:average-weight}. The following lemma is a counterpart of Lemma~\ref{lem:decomposition of bias}, and the proof details are omitted.

\begin{lemma}
    Consider Approximate TD-PMD \textup{(}Algorithm~\ref{alg:Approximated-TD-AC-PMD}\textup{)}. Let $\eta_k = \eta > 0$ and $\alpha_k = \alpha \in (0,1]$. Fixing any optimal policy $\pi^* \in \Pi^*$, for any $0 \leq k \leq K$ and $s\in\calS$, there holds
    \begin{align*}
        \ls| \E \ls[ \inner{\pi^*(\cdot|s) - \pi_k(\cdot|s)}{Q^{\pi_k}(s,\cdot) - Q^k(s,\cdot)} \rs] \rs| &\leq \ls| B_0^{(k)} \rs| + \sum_{j=1}^k \E \ls[ \ls|C_j^{(k)}\rs| + \ls|D_j^{(k)}\rs| + \ls|E_j^{(k)}\rs| \rs] + \sum_{j=1}^k \ls| \E\ls[F_j^{(k)}\rs] \rs|,
    \end{align*}
    where $B_0^{(k)}$, $C_j^{(k)}$, $D_j^{(k)}$, $E_j^{(k)}$, and $F_j^{(k)}$ admit the same expressions as in Lemma~\ref{lem:decomposition of bias}, and  $A_j = A^{\pi_b\circ\pi_j} = I - \alpha \Sigma^{\pi_b\circ\pi_j}(I - \gamma P^{\pi_j}_{\scriptscriptstyle\calS\times\calA})$.
    \label{lem:Approxiamted decomposition of bias}
\end{lemma}
 The upper bounds for the terms with $B_0^{(k)}$, $C_j^{(k)}$, and $D_j^{(k)}$ can be similarly established as in the analysis for Expected TD-PMD {by leveraging $\| A_j \|_\infty \leq (1-\alpha(1-\gamma)\tilde\sigma)$ and Lemma~\ref{lem:policy-shift}}. Thus we only present them in the following lemma without proofs.
\begin{lemma}[Bounds for terms with $B_0^{(k)}$, $C_j^{(k)}$, and $D_j^{(k)}$]
    Suppose that Assumptions~\ref{ass:behavior-exploration},~\ref{ass:ergodicity-2}, and~\ref{ass:h} hold. One has
    \begin{align*}
        \ls| B_0^{(k)} \rs| \leq \frac{2}{1-\gamma} \cdot \ls[ 1-(1-\gamma)\alpha\tilde\sigma \rs]^k, \quad \sum_{j=1}^k\ls| C_j^{(k)} \rs| \leq \frac{2\gamma |\calA|^2 \eta}{\alpha\lambda\tilde\sigma (1-\gamma)^4}, \;\; \mbox{and} \;\; \sum_{j=1}^k \ls| D_j^{(k)} \rs| \leq \frac{|\calA|^2\eta}{\alpha\lambda\tilde\sigma(1-\gamma)^3}.
    \end{align*}
    \label{lem:approximated-B0kCjkDjk-bounds}
\end{lemma}

The bound for  the term with $E_j^{(k)}$ is presented in the following lemma, whose proof is provided in Section~\ref{sec:pf:lem:approximated-Ejk-bound}. Differing from the analysis for Expected TD-PMD, it requires us to bound the drift between $\Sigma^{\pi_b\circ\pi_j}$ and $\Sigma^{\pi_b\circ\pi_{j-1}}$ by leveraging Lemma~\ref{lem:perturbated-chain-stationary-diff}. 

\begin{lemma}[Bound for terms with $E_j^{(k)}$]
    Suppose that Assumptions~\ref{ass:behavior-exploration},~\ref{ass:ergodicity-2}, and~\ref{ass:h} hold. One has
    \begin{align*}
        \sum_{j=1}^k \ls| E_j^{(k)} \rs| \leq \frac{2|\calA|^2 \cdot [2L+1]}{\alpha \lambda \tilde\sigma^2 (1-\gamma)^4}\eta,
    \end{align*}
    where $L := |\calA| \ls( \ls\lceil \log_\kappa m^{-1} \rs\rceil + (1-\kappa)^{-1} \rs)$.
    \label{lem:approximated-Ejk-bound}
\end{lemma}
The bound for the Markov noise part is presented in the following lemma. {The proof is overall similar to the one for Lemma~\ref{lem:expected-Fjk-bounds} and can be found in Section~\ref{sec:pf:lem:approximated-Fjk-bound}.}
\begin{lemma}[Bound for term with $F_j^{(k)}$]
    Suppose that Assumptions~\ref{ass:behavior-exploration} and~\ref{ass:ergodicity-2} hold. Recall the average weight $c_t^k$ in equation~\eqref{eq:Expected-TD-AC-PMD:average-weight}. Assume $B_k = B$, $k=0, \dots, K$. One has
    \begin{align*}
        \sum_{j=1}^k \ls| \E\ls[F_j^{(k)}\rs] \rs| \leq \frac{2}{(1-\gamma)\tilde\sigma} \Psi(B, \vartheta, m, \kappa),
    \end{align*}
    where $\Psi$ is defined in Lemma~\ref{lem:Expected-TD-AC-PMD:error-control}.
    \label{lem:approximated-Fjk-bound}
\end{lemma}
 
 Finally, one can obtain the upper bound for the stochastic bias term  and then establish the convergence for Approximate TD-PMD with the constant step size. Moreover, the $\tilde{O}(\varepsilon^{-2})$ sample complexity of Approximate TD-PMD follows immediately from this convergence result.

\begin{theorem}
    Consider Approximate TD-PMD \textup{(}Algorithm~\ref{alg:Approximated-TD-AC-PMD}\textup{)}. Suppose that Assumptions~\ref{ass:behavior-exploration},~\ref{ass:ergodicity-2}, and~\ref{ass:h} hold. Let
    \begin{align*}
        \eta_k = \eta \sqrt{\frac{\alpha\lambda \tilde\sigma^2 (1-\gamma)^4 \| D^{\pi^*}_{\pi_0} \|_\infty}{(6+4L)|\calA|^2(K+1)}}, \quad \alpha_k = \alpha \in (0,1], \quad B_k = B > 0,
    \end{align*}
    and $c_t^k$ takes the form in equation~\eqref{eq:Expected-TD-AC-PMD:average-weight}. Fixing any optimal policy $\pi^* \in \Pi^*$ and $\mu \in \Delta(\calS)$, there holds
    \begin{align*}
        \E\ls[ V^*(\mu) - V^{\pi_{\hat K}}(\mu) \rs] &\leq \frac{1}{K+1} \cdot \frac{3}{\alpha\tilde\sigma (1-\gamma)^3} + \frac{2}{\sqrt{K+1}} \cdot \ls( \frac{(6+4L)|\calA|^2 \| D^{\pi^*}_{\pi_0} \|_\infty}{\alpha\lambda \tilde\sigma^2(1-\gamma)^6} \rs)^{\nicefrac{1}{2}} + \frac{2}{\tilde\sigma(1-\gamma)^2} \Psi(B, \vartheta, m, \kappa),
    \end{align*}
    where $\Psi$ is defined in Lemma~\ref{lem:Expected-TD-AC-PMD:error-control} and $L$ is given in Lemma~\ref{lem:approximated-Ejk-bound}.
    \label{thm:Approximated-TD-AC-PMD:constant-step-size}
\end{theorem}

\begin{proof}
    The proof is the same with  that for  Theorem~\ref{thm:Expected-TD-AC-PMD:constant-step-size} by gathering Lemma~\ref{lem:Approximated-TD-AC-PMD:Lan-Xiao-bound} and Lemmas~\ref{lem:Approxiamted decomposition of bias}--\ref{lem:approximated-Fjk-bound} together.
\end{proof}


\subsection{Approximate TD-PMD with adaptive step sizes}
\label{sec:Approximated-TD-AC-PMD-with-adaptive-step-size}
The analysis for the adaptive step sizes in Section~\ref{sec:Expected-TD-AC-PMD-with-adaptive-step-size} can be similarly extended to Approximate TD-PMD. Recall that the key ingredient of the adaptive step sizes analysis in Section~\ref{sec:Expected-TD-AC-PMD-with-adaptive-step-size} is the utilization of the contraction property of $\calF^{\pi}_{\pi_b, \alpha}$ to establish the convergence rate of the critic $\| Q^* - Q^k \|_\infty$. Since the contraction property also holds for $\calF^{\pi}_{\pi_b\circ\pi, \alpha}$ (Proposition~\ref{pro:Approximated-TD-AC-PMD:sigma-bellman-property}), the analysis in Section~\ref{sec:Expected-TD-AC-PMD-with-adaptive-step-size} can be easily generalized to Approximate TD-PMD.

\begin{lemma}
    Suppose Assumption~\ref{ass:behavior-exploration} and Assumption~\ref{ass:ergodicity-2} hold. Consider Approximate TD-PMD \textup{(}Algorithm~\ref{alg:Approximated-TD-AC-PMD}\textup{)} with constant critic step size $\alpha_k = \alpha \in (0,1]$. There holds
    \begin{align*}
        \forall\, 0 \leq k \leq K : \quad \ls\| Q^* - Q^{k+1} \rs\|_\infty \leq [1-(1-\gamma)\alpha\tilde\sigma]\cdot \ls\| Q^* - Q^k \rs\|_\infty + \frac{\alpha\gamma}{\eta_k} \| D^{\tilde\pi_k}_{\pi_k} \|_\infty + \alpha \| \bar{ \omega}_k \|_\infty,
    \end{align*}
    where 
    \begin{align*}
        \ls\|D_{\pi_k}^{\tilde{\pi}_k} \rs\|_\infty :=\max_s D_{\pi_k}^{\tilde{\pi}_k}(s),
    \end{align*}
    and $\tilde{\pi}_k$ is any policy that satisfies $\langle\tilde{\pi}_k(\cdot|s),Q^k(s,\cdot)\rangle = \max_aQ^{k}(s,a),\;\forall s$.
    \label{lem:Approximated-TD-AC-PMD:linear-critic-convergence}
\end{lemma}
\begin{proof}
    Using the same argument as in the proof of Lemma~\ref{lem:Expected-TD-AC-PMD:linear-critic-convergence}, we have
    \begin{align*}
        \calF^{\pi_{k+1}} Q^k(s,a) \geq \calF Q^k(s,a) - \frac{\gamma}{\eta_k} \| D^{\tilde\pi_k}_{\pi_k} \|_\infty.
    \end{align*}
    By the critic update formula (equation~\eqref{eq:Approximated-TD-AC-PMD:critic-update}),
    \begin{align*}
        &\phantom{=\,\,\,}[Q^*-Q^{k+1}](s,a) \\
        &=[1-\alpha\sigma^{\pi_b\circ\pi_{k+1}}(s,a)][Q^*(s,a) - Q^k(s,a)] + \alpha \sigma^{\pi_b\circ\pi_{k+1}}(s,a) [ Q^*(s,a) - \calF^{\pi_{k+1}} Q^k(s,a)] - \alpha \bar{\omega}_k(s,a),
    \end{align*}
    we can similarly show that
    \begin{align*}
        &\phantom{=\,\,\,}|Q^*(s,a)-Q^{k+1}(s,a)| \\
        &\leq [1-\alpha\sigma^{\pi_b\circ\pi_{k+1}}(s,a)]\ls|Q^*(s,a) - Q^k(s,a)\rs| + \alpha \sigma^{\pi_b\circ\pi_{k+1}}(s,a) \ls|\calF Q^*(s,a) - \calF Q^k(s,a)\rs| \\
        &\;\;\;\; + \alpha\sigma^{\pi_b\circ\pi_{k+1}}(s,a) \frac{\gamma}{\eta_k} \| D^{\tilde\pi_k}_{\pi_k} \|_\infty + \alpha \ls|\bar{\omega}_k(s,a)\rs| \\
        &\leq [1-(1-\gamma)\alpha\sigma^{\pi_b\circ\pi_{k+1}}(s,a)] \ls\| Q^* - Q^k \rs\|_\infty + \alpha\gamma\eta_k^{-1}\| D^{\tilde\pi_k}_{\pi_k} \|_\infty + \alpha \ls|\bar{\omega}_k(s,a)\rs| \\
        &\leq [1-(1-\gamma)\alpha\tilde\sigma]\ls\| Q^* - Q^k \rs\|_\infty + \alpha \gamma\eta_k^{-1}\| D^{\tilde\pi_k}_{\pi_k} \|_\infty + \alpha\|\bar{\omega}_k\|_\infty,
    \end{align*}
    which completes the proof.
\end{proof}

\begin{lemma}
    Consider Approximate TD-PMD \textup{(}Algorithm~\ref{alg:Approximated-TD-AC-PMD}\textup{)}. Suppose Assumption~\ref{ass:ergodicity-2} hold. Let $c_t^k$ take the form in equation~\eqref{eq:Expected-TD-AC-PMD:average-weight} with $\vartheta=1$. Then there holds
    \begin{align*}
        \forall\, 0 \leq k \leq K : \quad \E\ls[ \ls\| \bar\omega_k \rs\|_\infty \rs] \leq \ls( \frac{4|\calS||\calA|}{(1-\gamma)^2} \ls( 1+\frac{m}{1-\kappa} \rs) \rs)^{\nicefrac{1}{2}} \cdot (B_k)^{-\nicefrac{1}{2}}.
    \end{align*}
    \label{lem:Approximated-TD-AC-PMD:error-control-2}
\end{lemma}

\begin{proof}
    Recall that $\calG_k = \{ s_0^k, \, \pi_{k+1}, \, Q^k \}$ and let $o_t^k := (s_t^k, a_t^k, s_t^{\prime\,k}, a_t^{\prime\,k}, s_{t+1}^k)$. By Lemma~\ref{lem:Approximated-TD-AC-PMD:bounded-value}, there holds $|\omega_t^k(s,a)| \leq 2(1-\gamma)^{-1}$ and $|\bar\omega_k(s,a)| \leq 2 (1-\gamma)^{-1}$. For any $0 \leq i < j \leq B_k-1$, following the same argument as in the proof of Lemma~\ref{lem:Expected-TD-AC-PMD:error-control-2} we have
    \begin{align*}
        \E\ls[ \omega_i^k(s,a) \omega_j^k(s,a) \, \big | \, \calG_k \rs] \leq \frac{2}{1-\gamma} \E \Bigg [ \ls| \E\ls[ \omega_j^k(s,a) \, \big | \, \calG_k, \, o_i^k \rs] \rs| \; \Bigg | \; \calG_k \Bigg].
    \end{align*}
    By the same computation as in equation~\eqref{eq:Approximated-TD-AC-PMD:delta-expectation}, one has 
    \begin{align*}
        \E\ls[ \delta_j^k(s,a) \, \big | \, \calG_k, \, o_i^k \rs] &= \P\ls[ (s_j^k, a_j^k) = (s,a) \, \big | \, \calG_k, \, o_i^k \rs] \cdot \ls[ \calF^{\pi_{k+1}} Q^k(s,a) - Q^k(s,a) \rs] \\
        &= \P \ls[ s_j^k = s \, \big | \, \pi_{k+1}, \, s_{i+1}^k \rs] \cdot \pi_b(a|s) \cdot \ls[ \calF^{\pi_{k+1}} Q^k(s,a) - Q^k(s,a) \rs].
    \end{align*}
    Therefore,
    \begin{align*}
        &\phantom{=\,\,\,}\ls| \E\ls[ \omega_j^k(s,a) \, | \, \calG_k, \, o_i^k \rs] \rs| \\
        &= \ls| \E\ls[ \delta_j^k(s,a) \, | \, \calG_k, \, o_i^k \rs] - \ls[ \calF^{\pi_{k+1}}_{\pi_b\circ\pi_{k+1}} Q^k - Q^k \rs](s,a) \rs| \\
        &= \ls| \P \ls[ s_j^k = s \, \big | \, \pi_{k+1}, \, s_{i+1}^k \rs] \cdot \pi_b(a|s) \cdot \ls[ \calF^{\pi_{k+1}} Q^k(s,a) - Q^k(s,a) \rs] - \sigma^{\pi_b\circ\pi_{k+1}}(s,a) \ls[ \calF^{\pi_{k+1}} Q^k - Q^k \rs](s,a) \rs| \\
        &\leq \ls| \P\ls[ s_j^k = s \, \big | \, \pi_{k+1}, \, s_{i+1}^k \rs] - \nu^{\pi_b\circ\pi_{k+1}}(s) \rs| \cdot \ls| \pi_b(a|s) \rs| \cdot \underbrace{\ls| \ls[ \calF^{\pi_{k+1}} Q^k(s,a) - Q^k(s,a) \rs] \rs|}_{\leq (1-\gamma)^{-1} \mbox{\small \;by Lemma~\ref{lem:Approximated-TD-AC-PMD:bounded-value}}} \\
        &\leq \frac{1}{1-\gamma} \underbrace{d_{TV} \ls( \P\ls[ s_j^k = \cdot \, | \, s_{i+1}^k \rs], \; \nu^{\pi_b\circ\pi_{k+1}}(\cdot) \, | \, \pi_{k+1} \rs)}_{\mbox{\small Applying equation~\eqref{eq:Approximated-TD-AC-PMD:MDP-ergodicity}}} \\
        &\leq \frac{1}{1-\gamma} m\kappa^{j-i-1}.
    \end{align*}
    Then the remaining proof is the same with that for Lemma~\ref{lem:Expected-TD-AC-PMD:error-control-2}.
\end{proof}

\begin{lemma}
    Suppose Assumption~\ref{ass:behavior-exploration} and Assumption~\ref{ass:ergodicity-2} hold. Consider Approximate TD-PMD \textup{(}Algorithm~\ref{alg:Approximated-TD-AC-PMD}\textup{)} with constant critic step size $\alpha_k = \alpha \in (0,1]$. There holds
    \begin{align*}
        \forall\, 0 \leq k \leq K: \quad \ls\| Q^* - Q^{\pi_{k+1}} \rs\|_\infty \leq  \frac{1}{\alpha(1-\gamma)\tilde\sigma} \ls( \ls\| Q^* - Q^{k+1} \rs\|_\infty + \ls\| Q^* - Q^k \rs\|_\infty + \alpha \| \bar{{\omega}}_k \|_\infty \rs).
    \end{align*}
    \label{lem:Approximated-TD-AC-PMD:linear-actor-convergence}
\end{lemma}
\begin{proof}
     By substituting equation~\eqref{eq:Approximated-TD-AC-PMD:critic-update} for equation~\eqref{eq:Expected-TD-AC-PMD:critic-update}, $\calF^{\pi_{k+1}}_{\pi_b\circ\pi_{k+1}, \alpha}$ for $\calF^{\pi_{k+1}}_{\pi_b, \alpha}$, and Proposition~\ref{pro:Approximated-TD-AC-PMD:sigma-bellman-property} for Proposition~\ref{pro:Expected-TD-AC-PMD:sigma-bellman-property}, the proof remains same with  that for Lemma~\ref{lem:Expected-TD-AC-PMD:linear-actor-convergence}.
\end{proof}

\begin{theorem}
    Suppose Assumption~\ref{ass:behavior-exploration} and Assumption~\ref{ass:ergodicity-2} hold. Consider Approximate TD-PMD \textup{(}Algorithm~\ref{alg:Approximated-TD-AC-PMD}\textup{)} with constant critic step size $\alpha_k = \alpha \in (0,1]$ and the adaptive policy update step size
    \begin{align*}
        \eta_k \geq \eta \cdot \| D^{\tilde\pi_k}_{\pi_k} \|_\infty \quad \mbox{with} \quad \eta > 0.
    \end{align*}
    There holds
    \begin{align*}
        \E[\| Q^* - Q^{\pi_k} \|_\infty] &\leq \frac{2}{\alpha \tilde\sigma (1-\gamma)^2}[1-(1-\gamma)\alpha\tilde\sigma]^{K-1} + \frac{2\gamma}{\alpha\eta(1-\gamma)^2\tilde\sigma^2 } \\
        &\;\;\;\;+ \frac{1}{\tilde\sigma (1-\gamma)} \ls(\frac{4|\calS||\calA|}{(1-\gamma)^2} \ls( 1 + \frac{m}{1-\kappa}\rs)\rs)^{\nicefrac{1}{2}} \cdot \ls( B_{K-1}^{-\nicefrac{1}{2}} + \tilde\Xi(K) + \tilde\Xi(K-1) \rs),
    \end{align*}
    where $\tilde\Xi(t) := \sum_{k=0}^{t-1} [1-(1-\gamma)\alpha\tilde\sigma]^{t-1-k} B_k^{-\nicefrac{1}{2}}$.
    \label{thm:Approximated-TD-AC-PMD:adaptive-step-size}
\end{theorem}
\begin{proof}
    By substituting Lemma~\ref{lem:Approximated-TD-AC-PMD:error-control-2} for Lemma~\ref{lem:Expected-TD-AC-PMD:error-control-2}, Lemma~\ref{lem:Approximated-TD-AC-PMD:linear-critic-convergence} for Lemma~\ref{lem:Expected-TD-AC-PMD:linear-critic-convergence}, and Lemma~\ref{lem:Approximated-TD-AC-PMD:linear-actor-convergence} for Lemma~\ref{lem:Expected-TD-AC-PMD:linear-actor-convergence}, the proof remains same with that for Theorem~\ref{thm:Expected-TD-AC-PMD:adaptive-step-size}.
\end{proof}
Base on the same parameter settings as in Remark~\ref{rem:Expected-TD-AC-PMD:adaptive-step-size}, one can similarly  establish the $O(\varepsilon^{-2})$ sample complexity of Approximate TD-PMD with adaptive step sizes.

\section{Proofs of critical lemmas}\label{sec:lemma-proof}
\subsection{Proof of Lemma~\ref{lem:Expected-TD-AC-PMD:Lan-Xiao-bound}}\label{sec:proof:Expected-TD-AC-PMD:Lan-Xiao-bound}
For any state $s\in\calS$, applying the three-point-descent lemma (Lemma~\ref{lem:three-point-descent-lemma}) with $p=\pi_{k }(\cdot|s)$ and $p=\pi^*(\cdot|s)$ gives that
    \begin{align*}
        \eta \ls\langle \pi_{k+1}(\cdot|s) - \pi_k(\cdot|s), \, Q^k(s,\cdot) \rs\rangle &\geq D^{\pi_{k+1}}_{\pi_k}(s) + D^{\pi_k}_{\pi_{k+1}}(s), \numberthis \label{eq:Expected-TD-AC-PMD:three-poinj-1} \\
        \eta \ls\langle \pi_{k+1}(\cdot|s) - \pi^*(\cdot|s) , \, Q^k(s,\cdot) \rs\rangle &\geq D^{\pi_{k+1}}_{\pi_k}(s) + D^{\pi^*}_{\pi_{k+1}}(s) - D^{\pi^*}_{\pi_k}(s). \numberthis \label{eq:Expected-TD-AC-PMD:three-point-2}
    \end{align*}
    Subtracting $\ls\langle \pi_k(\cdot|s), \, Q^k(s,\cdot) \rs\rangle$ on both of sides of equation~\eqref{eq:Expected-TD-AC-PMD:three-point-2} and after rearrangement we have 
    \begin{equation*}
    \begin{aligned}
        \underbrace{\eta \ls\langle \pi_{k+1}(\cdot|s) - \pi_k(\cdot|s), \, Q^k(s,\cdot) \rs\rangle - D^{\pi_{k+1}}_{\pi_k}(s)}_{:= I_k(s)} &\geq \eta \ls\langle \pi^*(\cdot|s) - \pi_k(\cdot|s), \, Q^{\pi_k}(s,\cdot) \rs\rangle + D^{\pi^*}_{\pi_{k+1}}(s) - D^{\pi^*}_{\pi_k}(s) \\
        &\;\;\;\;+ \underbrace{\eta \ls\langle \pi^*(\cdot|s) - \pi_k (\cdot|s), \, Q^k(s,\cdot) - Q^{\pi_k}(s,\cdot)\rs\rangle}_{:= G_k(s)}.
    \end{aligned}
    \end{equation*}
    For arbitrary $\mu\in\Delta(\calS)$, taking expectation with respect to $s\sim d^{*}_\mu$ on both sides gives
    \begin{align*}
        \E_{s\sim d^*_\mu} [I_k(s)] &\geq \eta \, \E_{s\sim d^*_\mu} \ls[ \ls\langle \pi^*(\cdot|s) - \pi_k(\cdot|s), \, Q^{\pi_k}(s,\cdot) \rs\rangle \rs] + \E_{s\sim d^*_\mu} \ls[ D^{\pi^*}_{\pi_{k+1}}(s) - D^{\pi^*}_{\pi_k}(s) \rs] + \E_{s\sim d^*_{\mu}} [G_k(s)] \\
        &= \eta(1-\gamma) (V^*(\mu) - V^{\pi_k}(\mu)) + \ls[D^{\pi^*}_{\pi_{k+1}}(d^*_\mu) - D^{\pi^*}_{\pi_k}(d^*_\mu) \rs] + \E_{s\sim d^*_{\mu}} [G_k(s)],
    \end{align*}
    where we have applied the performance difference lemma (Lemma~\ref{lem:pdl}) and set $D^{\pi^*}_{\pi_{k+1}}(d^*_\mu) = \E_{s\sim d^*_\mu} [D^{\pi^*}_{\pi_{k+1}}(s)]$, and  $D^{\pi^*}_{\pi_k}(d^*_\mu)=\E_{s\sim d^*_\mu} [D^{\pi^*}_{\pi_k}(s)]$. Summing it from $k=0$ to $k=K$ yields
    \begin{align*}
        \sum_{k=0}^K \ls( V^*(\mu) - V^{\pi_k}(\mu) \rs) &\leq \frac{1}{\eta(1-\gamma)} \ls[ D^{\pi^*}_{\pi_0}(d^*_\mu) - D^{\pi^*}_{\pi_{K+1}}(d^*_\mu) + \sum_{k=0}^K \E_{s\sim d^*_\mu} [I_k(s)] - \sum_{k=0}^K \E_{s\sim d^*_\mu}[ G_k (s) ] \rs] \\
        &\leq \frac{1}{\eta(1-\gamma)} \ls[ \| D^{\pi^*}_{\pi_0} \|_\infty + \sum_{k=0}^K \E_{s\sim d^*_\mu} [I_k(s)] - \sum_{k=0}^K \E_{s\sim d^*_\mu}[ G_k (s) ] \rs]. \numberthis \label{eq:Xiao-bound-1}
    \end{align*}
    Next we need to bound 
    $\sum_{k=0}^{K} \E_{s\sim d^*_\mu} [I_k(s)]$. Noticing that $I_k(s) \geq 0$ by equation~\eqref{eq:Expected-TD-AC-PMD:three-poinj-1} and the fact that $d^{\pi_{k+1}}_{d^*_\mu}(s) \geq (1-\gamma) d^*_\mu(s)$, we have
    \begin{align*}
        \E_{s\sim d^*_\mu} [I_k(s)] \leq \frac{1}{1-\gamma} \E_{s\sim d^{\pi_{k+1}}_{d^*_\mu}} \ls[ I_k(s) \rs].
    \end{align*}
    Then $I_k(s)$ can be upper bounded as follows,
    \begin{align*}
        I_k(s) &= \eta \ls\langle \pi_{k+1}(\cdot|s) - \pi_k(\cdot|s), \, Q^{\pi_k}(s,\cdot) \rs\rangle \\
        &\;\;\;\;+ \eta \ls\langle \pi_{k+1}(\cdot|s) - \pi_k(\cdot|s), \, Q^k(s,\cdot) - Q^{\pi_k}(s,\cdot) \rs\rangle - D^{\pi_{k+1}}_{\pi_k}(s) \\[.5em]
        &\stackrel{(a)}{\leq} \eta \ls\langle \pi_{k+1}(\cdot|s) - \pi_k(\cdot|s), \, Q^{\pi_k}(s,\cdot) \rs\rangle \\
        &\;\;\;\; + \ls( \sqrt{\frac{\lambda}{|\calA|^2}} \cdot \| \pi_{k+1}(\cdot|s) - \pi_k(\cdot|s) \|_1 \rs) \cdot \ls( \eta \sqrt{\frac{|\calA|^2}{\lambda}} \cdot \ls\| Q^k(s,\cdot) - Q^{\pi_k}(s,\cdot) \rs\|_\infty \rs) - D^{\pi_{k+1}}_{\pi_k}(s) \\[.5em]
        &\stackrel{(b)}{\leq} \eta \ls\langle \pi_{k+1}(\cdot|s) - \pi_k(\cdot|s), \, Q^{\pi_k}(s,\cdot) \rs\rangle \\
        &\;\;\;\; + \underbrace{\frac{\lambda}{2|\calA|^2} \cdot \ls\| \pi_{k+1}(\cdot|s) - \pi_k(\cdot|s) \rs\|_1^2 - D^{\pi_{k+1}}_{\pi_k}(s)}_{\leq 0 \mbox{\small \; by equation~\eqref{eq:Bregman-lower-bound}}} + \frac{\eta^2 |\calA|^2}{2\lambda} \cdot \| Q^k(s,\cdot) - Q^{\pi_k}(s,\cdot) \|_\infty^2 \\[.5em]
        & \leq \eta \ls\langle \pi_{k+1}(\cdot|s) - \pi_k(\cdot|s), \, Q^{\pi_k}(s,\cdot) \rs\rangle + \frac{\eta^2 |\calA|^2}{2\lambda} \cdot \| Q^k - Q^{\pi_k} \|_\infty^2,
    \end{align*}
    where we apply H{\"o}lder's inequality in $(a)$ and Young's inequality in $(b)$. Together with the performance difference lemma (Lemma~\ref{lem:pdl}), we have
    \begin{align*}
        \E_{s\sim d^*_\mu}[I_k(s)] &\leq \frac{\eta}{1-\gamma} \E_{s\sim d^{\pi_{k+1}}_{d^*_\mu}}\ls[ \ls\langle \pi_{k+1}(\cdot|s) - \pi_k(\cdot|s), \, Q^{\pi_k}(s,\cdot) \rs\rangle \rs] + \frac{\eta^2|\calA|^2}{2(1-\gamma)\lambda} \cdot \| Q^{k} - Q^{\pi_k} \|_\infty^2 \\
        &= \eta \ls[ V^{\pi_{k+1}}(d^*_\mu) - V^{\pi_k}(d^*_\mu) \rs] + \frac{\eta^2|\calA|^2}{2(1-\gamma)\lambda} \cdot \| Q^{k} - Q^{\pi_k} \|_\infty^2.
        \end{align*}
        It follows that
        \begin{align*}
        \sum_{k=0}^K \E_{s\sim d^*_\mu}[I_k(s)] & \leq \eta \ls[ V^{\pi_{K+1}} (d^*_\mu) - V^{\pi_0}(d^*_\mu) \rs] + \sum_{k=0}^{K}\frac{\eta^2|\calA|^2}{2(1-\gamma)\lambda} \cdot \| Q^{k} - Q^{\pi_k} \|_\infty^2 \\
        &\leq \frac{\eta}{1-\gamma} + \sum_{k=0}^{K}\frac{\eta^2|\calA|^2}{2(1-\gamma)\lambda} \cdot \| Q^{k} - Q^{\pi_k} \|_\infty^2.
    \end{align*}
    Plugging it back to equation~\eqref{eq:Xiao-bound-1}, taking the total expectation, and using the fact $\E\ls[ V^*(\mu) - V^{\pi_{\hat K}}(\mu) \rs] = (K+1)^{-1} \sum_{k=0}^K \E\ls[ V^* (\mu) - V^{\pi_k}(\mu) \rs]$ complete the proof.

\subsection{Proof of Lemma~\ref{lem:decomposition of bias}}\label{sec:proof-bias-deomposition}
Define
\begin{align*}
B_j^{(k)}=\inner{\ls[(A_j)^{k-j}\rs]^\top J_s(\pi^*-\pi_j)}{Q^{\pi_j}-Q^j}.
\end{align*}
Recalling the definition of $E_s$ and $J_s$ in equation~\eqref{eq:EsJs}, it is evident that 
\begin{align*}
\inner{\pi^*(\cdot|s) - \pi_k(\cdot|s)}{Q^{\pi_k}(s,\cdot) - Q^k(s,\cdot)}& = \langle E_s(\pi^*-\pi_k),E_s(Q^{\pi_k}-Q^{k})\rangle = B_k^{(k)}.
\end{align*}
Therefore, 
\begin{align}
\inner{\pi^*(\cdot|s) - \pi_k(\cdot|s)}{Q^{\pi_k}(s,\cdot) - Q^k(s,\cdot)}& =B_0^{(k)}+\sum_{j=1}^{k}\left[B^{(k)}_{j}-B_{j-1}^{(k)}\right].\label{eq:proof-bias-decomp-01}
\end{align}
To compute $B^{(k)}_{j}-B_{j-1}^{(k)}$, first one has 
\begin{align*}
        B_j^{(k)} &= \inner{\ls[(A_j)^{k-j}\rs]^\top J_s [\pi^* - \pi_j]}{Q^{\pi_j} - Q^j} \\
        &\stackrel{(a)}{=} \inner{\ls[(A_j)^{k-j}\rs]^\top J_s [\pi^* - \pi_j]}{Q^{\pi_j} - Q^{j-1} - \alpha \Sigma^{\pi_b} \ls[ \calF^{\pi_j}Q^{j-1} - Q^{j-1} \rs] - \alpha \bar\omega_{j-1}} \\
        &\stackrel{(b)}{=} \inner{\ls[(A_j)^{k-j}\rs]^\top J_s [\pi^* - \pi_j]}{[I - \alpha \Sigma^{\pi_b}] [Q^{\pi_j} - Q^{j-1}]} \\
        & \;\;\;\; + \inner{\ls[(A_j)^{k-j}\rs]^\top J_s [\pi^* - \pi_j]}{\alpha \Sigma^{\pi_b} [\calF^{\pi_j}Q^{\pi_j} - \calF^{\pi_j}Q^{j-1}]} \\
        & \;\;\;\; - \inner{\ls[(A_j)^{k-j}\rs]^\top J_s [\pi^* - \pi_j]}{\alpha \bar\omega_{j-1}} \\
        &\stackrel{(c)}{=} \inner{\ls[(A_j)^{k-j}\rs]^\top J_s [\pi^* - \pi_j]}{[I - \alpha \Sigma^{\pi_b}] [Q^{\pi_j} - Q^{j-1}]} \\
        & \;\;\;\; + \inner{\ls[(A_j)^{k-j}\rs]^\top J_s [\pi^* - \pi_j]}{(\alpha \Sigma^{\pi_b})(\gamma P^{\pi_j}_{\scriptscriptstyle\calS\times\calA}) [Q^{\pi_j} - Q^{j-1}]} \\
        & \;\;\;\; - \inner{\ls[(A_j)^{k-j}\rs]^\top J_s [\pi^* - \pi_j]}{\alpha \bar\omega_{j-1}} \\
        &= \inner{\ls[(A_j)^{k-j}\rs]^\top J_s [\pi^* - \pi_j]}{[I - \alpha \Sigma^{\pi_b} (I - \gamma P^{\pi_j}_{\scriptscriptstyle\calS\times\calA}) ] [Q^{\pi_j} - Q^{j-1}]} - \inner{\ls[(A_j)^{k-j}\rs]^\top J_s [\pi^* - \pi_j]}{\alpha \bar\omega_{j-1}} \\
        &\stackrel{(d)}{=} \inner{\ls[(A_j)^{k-j+1}\rs]^\top J_s [\pi^* - \pi_j]}{ [Q^{\pi_j} - Q^{j-1}]} - \inner{\ls[(A_j)^{k-j}\rs]^\top J_s [\pi^* - \pi_j]}{\alpha \bar\omega_{j-1}},
    \end{align*}
    where $(a)$ is from the critic update formula (equation~\eqref{eq:Expected-TD-AC-PMD:critic-update}), $(b)$ uses the fact that $\calF^{\pi_j} Q^{\pi_j} = Q^{\pi_j}$, $(c)$ is from the fact that $\calF^{\pi} Q = r + \gamma P^{\pi}_{\scriptscriptstyle\calS\times\calA} Q$ for any $\pi \in \Pi$ and $Q \in \mathbb{R}^{|\calS||\calA|}$, and $(d)$ is from $A_j = [I - \alpha \Sigma^{\pi_b}(I - \gamma P^{\pi_j}_{\scriptscriptstyle\calS\times\calA})]$. By Further decomposing the first term yields
    \begin{align*}
        B_j^{(k)} &= \inner{\ls[(A_j)^{k-j+1}\rs]^\top J_s [\pi^* - \pi_j]}{ [Q^{\pi_j} - Q^{j-1}]} - \inner{\ls[(A_j)^{k-j}\rs]^\top J_s [\pi^* - \pi_j]}{\alpha \bar\omega_{j-1}} \\
        &= \underbrace{\inner{\ls[(A_{j-1})^{k-j+1}\rs]^\top J_s [\pi^* - \pi_{j-1}]}{ [Q^{\pi_{j-1}} - Q^{j-1}]}}_{=B_{j-1}^{(k)}} \\
        &\;\;\;\; + \underbrace{\inner{\ls[ (A_j)^{k-j+1} \rs]^\top J_s[\pi^* - \pi_j]}{[Q^{\pi_j} - Q^{\pi_{j-1}}]}}_{:= C^{(k)}_j} + \underbrace{\inner{\ls[ (A_j)^{k-j+1} \rs]^\top J_s[\pi_{j-1} - \pi_j]}{[Q^{\pi_{j-1}} - Q^{{j-1}}]}}_{:= D^{(k)}_j} \\
        & \;\;\;\; + \underbrace{\inner{\ls[ (A_j)^{k-j+1} - (A_{j-1})^{k-j+1} \rs]^\top J_s[\pi^* - \pi_{j-1}]}{[Q^{\pi_{j-1}} - Q^{{j-1}}]}}_{:=E_j^{(k)}} - \underbrace{\inner{\ls[(A_j)^{k-j}\rs]^\top J_s [\pi^* - \pi_j]}{\alpha \bar\omega_{j-1}}}_{:=F_j^{(k)}}.
    \end{align*}
    Plugging it back to \eqref{eq:proof-bias-decomp-01}, taking the expectation, and taking the absolute value completes the proof.

    \subsection{Proof of Lemma~\ref{lem:expected-CjkDjk-bounds}}\label{sec:expected-CjkDjk-bounds}
     For the term with $C_j^{(k)}$,
 \begin{align*}
        \sum_{j=1}^k \ls| C_j^{(k)} \rs| &= \sum_{j=1}^k \ls| \inner{\ls[ (A_j)^{k-j+1} \rs]^\top J_s [\pi^* - \pi_j]}{[Q^{\pi_j} - Q^{\pi_{j-1}}]} \rs| \\
        &\leq \sum_{j=1}^k\ls\| A_j \rs\|_\infty^{k-j+1} \cdot \ls\| J_s[\pi^* - \pi_j] \rs\|_1 \cdot \underbrace{\ls\| Q^{\pi_j} - Q^{\pi_{j-1}} \rs\|_\infty}_{\mbox{\small applying Lemma~\ref{lem:Lip:Q-value}}} \\
        &\leq \sum_{j=1}^k\frac{\gamma|\calA|}{(1-\gamma)^2} \cdot \ls\| A_j \rs\|_\infty^{k-j+1} \cdot \underbrace{\ls\| J_s[\pi^* - \pi_j] \rs\|_1}_{\leq 2} \cdot \underbrace{\ls\| \pi_j - \pi_{j-1} \rs\|_\infty}_{\mbox{\small applying Lemma~\ref{lem:policy-shift} and Lemma~\ref{lem:Expected-TD-AC-PMD:bounded-value}}} \\
        &\leq \sum_{j=1}^k\frac{2\gamma |\calA|^2\eta}{\lambda(1-\gamma)^3} \cdot \ls[ 1-(1-\gamma)\alpha\tilde \sigma_b \rs]^{k-j+1} \\
        &\leq \frac{2\gamma |\calA|^2 \eta}{\alpha\lambda\tilde \sigma_b(1-\gamma)^4}.
    \end{align*}
    For the term with  $D_j^{(k)}$, 
    \begin{align*}
        \sum_{j=1}^k \ls| D_j^{(k)} \rs| &= \sum_{j=1}^k\ls| \inner{\ls[\ls( A_j \rs)^{k-j+1}\rs]^\top J_s [\pi_{j-1} - \pi_j]}{[Q^{\pi_{j-1}} - Q^{j-1}]} \rs| \\
        &\leq \sum_{j=1}^k\ls\| A_j \rs\|_\infty^{k-j+1} \cdot \underbrace{\ls\| J_s [\pi_{j-1} - \pi_j] \rs\|_1}_{\leq |\calA| \cdot \ls\| \pi_{j-1} - \pi_j \rs\|_\infty} \cdot \underbrace{\ls\| Q^{\pi_{j-1}} - Q^{j-1} \rs\|_\infty}_{\leq (1-\gamma)^{-1}} \\
        &\leq \sum_{j=1}^k\frac{|\calA|^2\eta}{\lambda(1-\gamma)^2} \cdot \ls[ 1 - (1-\gamma) \alpha \tilde \sigma_b \rs]^{k-j+1} \\
        &\leq \frac{|\calA|^2\eta}{\alpha\lambda\tilde \sigma_b(1-\gamma)^3}.
    \end{align*}

    \subsection{Proof of Lemma~\ref{lem:expected-Ejk-bounds}}\label{sec:expected-Ejk-bound}
     First, one has
    \begin{align*}
        \ls| E_j^{(k)} \rs| &= \ls| \inner{\ls[ (A_{j})^{k-j+1} - (A_{j-1})^{k-j+1} \rs]^\top J_s [\pi^* - \pi_{j-1}]}{[Q^{\pi_{j-1}} - Q^{j-1}]} \rs| \\
        &\leq \ls\| (A_j)^{k-j+1} - (A_{j-1})^{k-j+1} \rs\|_\infty \cdot \ls\| J_s[\pi^* - \pi_{j-1}] \rs\|_1 \cdot \ls\| Q^{\pi_{j-1}} - Q^{j-1} \rs\|_\infty \\
        &\leq \frac{2}{1-\gamma} \cdot \ls\| (A_j)^{k-j+1} - (A_{j-1})^{k-j+1} \rs\|_\infty.
    \end{align*}
    For the term $\ls\| (A_{j})^{k-j+1} - (A_{j-1})^{k-j+1} \rs\|_\infty$,  notice that
    \begin{align*}
        \ls\| (A_{j})^{k-j+1} - (A_{j-1})^{k-j+1} \rs\|_\infty &= \ls\| (A_{j})^{k-j+1} - (A_{j})\cdot (A_{j-1})^{k-j} + (A_{j})\cdot (A_{j-1})^{k-j} - (A_{j-1})^{k-j+1} \rs\|_\infty \\
        &\leq \ls\| (A_{j}) \cdot \ls[ (A_{j})^{k-j} - (A_{j-1})^{k-j} \rs] \rs\|_\infty + \ls\| [A_{j} - A_{j-1}] \cdot (A_{j-1})^{k-j} \rs\|_\infty \\
        &\leq [1-(1-\gamma) \alpha\tilde\sigma_b] \cdot \ls\| (A_{j})^{k-j} - (A_{j-1})^{k-j} \rs\|_\infty \\
        & \;\;\;\; + {\ls\| A_{j} - A_{j-1} \rs\|_\infty} \cdot [1-(1-\gamma)\alpha\tilde\sigma_b]^{k-j} \\
        &= [1-(1-\gamma) \alpha\tilde\sigma_b] \cdot \ls\| (A_{j})^{k-j} - (A_{j-1})^{k-j} \rs\|_\infty \\
        & \;\;\;\; + \underbrace{\ls\| (\alpha\gamma \Sigma^{\pi_b}) (P^{\pi_{j}}_{\scriptscriptstyle\calS\times\calA} - P^{\pi_{j-1}}_{\scriptscriptstyle\calS\times\calA}) \rs\|_\infty}_{\leq \ls\| \alpha\gamma\Sigma^{\pi_b} \rs\|_\infty \cdot \ls\| P^{\pi_{j}}_{\scriptscriptstyle\calS\times\calA} - P^{\pi_{j-1}}_{\scriptscriptstyle\calS\times\calA} \rs\|_\infty \leq \alpha\gamma |\calA| \cdot \| \pi_{j} - \pi_{j-1} \|_\infty} \times [1-(1-\gamma)\alpha\tilde\sigma_b]^{k-j} \\
        &\leq [1-(1-\gamma)\alpha\tilde\sigma_b] \cdot \ls\| (A_{j})^{k-j} - (A_{j-1})^{k-j} \rs\|_\infty + [1-(1-\gamma)\alpha\tilde\sigma_b]^{k-j} \cdot \frac{\alpha\gamma |\calA|^2\eta }{\lambda(1-\gamma)},
    \end{align*}
    where we leverage the fact that $\| P^{\pi}_{\scriptscriptstyle\calS\times\calA} - P^{\pi^\prime}_{\scriptscriptstyle\calS\times\calA} \|_\infty \leq |\calA| \cdot \ls\| \pi - \pi^\prime \rs\|_\infty$ which can be directly verified. By induction we have
    \begin{align*}
        \ls\| (A_{j})^{k-j+1} - (A_{j-1})^{k-j+1} \rs\|_\infty \leq \frac{\alpha\gamma |\calA|^2 \eta}{\lambda (1-\gamma)} \cdot (k-j+1) \cdot \ls[ 1-(1-\gamma)\alpha\tilde\sigma_b \rs]^{k-j}.
    \end{align*}
    Plugging it back yields
    \begin{align*}
        \sum_{j=1}^k \ls| E_j^{(k)} \rs| &\leq \sum_{j=1}^k \frac{2\alpha\gamma|\calA|^2 \eta}{\lambda (1-\gamma)^2} \cdot (k-j+1) \cdot \ls[ 1-(1-\gamma) \alpha \tilde\sigma_b \rs]^{k-j} \\
        &\leq \frac{2\alpha\gamma|\calA|^2 \eta}{\lambda (1-\gamma)^2} \cdot \frac{1}{\alpha^2\tilde\sigma_b^2 (1-\gamma)^2} \\
        &\leq \frac{2\gamma |\calA|^2 \eta}{\alpha\lambda\tilde\sigma_b^2 (1-\gamma)^4},
    \end{align*}
    where the inequality follows from the following inequality:
    \begin{align*}
        0 < q < 1: \quad \sum_{j=1}^k (k-j+1) \cdot q^{k-j} = \sum_{j=1}^k j \cdot q^{j-1} &= \frac{1-(k+1) q^k + kq^{k+1}}{(1-q)^2} \leq \frac{1}{(1-q)^2}.
    \end{align*}
\subsection{Proof of Lemma~\ref{lem:expected-Fjk-bounds}}\label{sec:expected-Fjk-bound}
Define 
$
\mathcal{G}_{k}=\{Q^k,\pi_{k+1},s_0^k\}.
$ A direct calculation yields 
\begin{align*}
\ls|\mathbb{E}\ls[F_j^{(k)}\rs]\rs|&=\ls|\mathbb{E}\ls[\ls\langle \ls[(A_j)^{k-j}\rs]^\top J_s(\pi^*-\pi_j), \;\alpha\bar{\omega}_{j-1}\rs\rangle\rs]\rs|=
\ls|\mathbb{E}\ls[\mathbb{E}\ls[\ls\langle \ls[(A_j)^{k-j}\rs]^\top J_s(\pi^*-\pi_j), \;\alpha\bar{\omega}_{j-1}\rs\rangle \Big |\mathcal{G}_{j-1}\rs]\rs]\rs|\\
&\leq \mathbb{E}\ls[\left|\mathbb{E}\ls[\ls\langle \ls[(A_j)^{k-j}\rs]^\top J_s(\pi^*-\pi_j), \;\alpha\bar{\omega}_{j-1}\rs\rangle\Big |\mathcal{G}_{j-1}\rs]\right|\rs]\\
&=\mathbb{E}\ls[\left|\ls\langle \ls[(A_j)^{k-j}\rs]^\top J_s(\pi^*-\pi_j), \;\alpha\mathbb{E}[\bar{\omega}_{j-1}|\mathcal{G}_{j-1}]\rs\rangle\right|\rs]\\
&\leq \mathbb{E}[\|A_j\|_\infty^{k-j}\cdot\|J_s(\pi^*-\pi_j)\|_1\cdot\alpha\|\mathbb{E}[\bar{\omega}_{j-1}|\mathcal{G}_{j-1}]\|_\infty]\\
&\leq 2\alpha(1-\alpha(1-\gamma)\tilde{\sigma}_b)^{k-j}\mathbb{E}[\|\mathbb{E}[\bar{\omega}_{j-1}|\mathcal{G}_{j-1}]\|_\infty].
\end{align*}
Thus, it suffices to consider $|\mathbb{E}[\bar{\omega}_{k}(s,a)|\mathcal{G}_{k}]|$ next.  Recall that
\begin{align*}
\omega_t^k(s,a)=\delta_{t}^k(s,a)-[\mathcal{F}^{\pi_{k+1}}_{\pi_b}Q^k-Q^k](s,a)
\quad\mbox{and}\quad
\bar{\omega}_k(s,a)=\sum_{t=0}^{B_k-1}c_t^k\cdot\omega_t^k(s,a).
\end{align*}
One has 
\begin{align*}
    \mathbb{E}[\delta_t^k(s,a)|\mathcal{G}_k]&=\mathbb{E}[1_{(s_t^{k},a^{k}_t)=(s,a)}\left[r^{k}_t+\gamma\mathbb{E}_{a\sim\pi_{k+1}(\cdot|s_{t+1}^{k})}[Q^{k}(s^{k}_{t+1},a)]-Q^{k}(s^{k}_t,a^{k}_t)\right]|\mathcal{G}_k]\\
    &=\mathbb{E}\left[1_{(s_t^{k},a^{k}_t)=(s,a)}\mathbb{E}\left[r^{k}_t+\gamma\mathbb{E}_{a\sim\pi_{k+1}(\cdot|s_{t+1}^{k})}[Q^{k}(s^{k}_{t+1},a)]-Q^{k}(s^{k}_t,a^{k}_t)|\mathcal{G}_k,s_t^k,a_t^k\right]|\mathcal{G}_k\right]\\
    &=\mathbb{E}\left[1_{(s_t^{k},a^{k}_t)=(s,a)}\cdot[\mathcal{F}^{\pi_{k+1}}Q^k-Q^k](s_t^k,a_t^k)|\mathcal{G}_k\right]\\
    &=\mathbb{E}\left[1_{(s_t^{k},a^{k}_t)=(s,a)}\cdot[\mathcal{F}^{\pi_{k+1}}Q^k-Q^k](s,a)|\mathcal{G}_k\right]\\
    &=\mathbb{P}[s_t^k=s|\mathcal{G}_k]\pi_b(a|s)[\mathcal{F}^{\pi_{k+1}}Q^k-Q^k](s,a)\\
    &=\mathbb{P}[s_t^k=s|s_0^k]\cdot\pi_b(a|s)\cdot [\mathcal{F}^{\pi_{k+1}}Q^k-Q^k](s,a). \numberthis \label{eq:Expected-TD-AC-PMD:delta-expectation}
\end{align*}
It follows that 
\begin{align*}
|\mathbb{E}[\omega_t^k(s,a)|\mathcal{G}_k]|&=|\mathbb{E}[\delta_t^k(s,a)|\mathcal{G}_k]-\sigma^{\pi_b}(s,a)[\mathcal{F}^{\pi_{k+1}}Q^k-Q^k](s,a)|\\
&=|\mathbb{E}[\delta_t^k(s,a)|\mathcal{G}_k]-\nu^{\pi_b}(s)\pi_b(a|s)[\mathcal{F}^{\pi_{k+1}}Q^k-Q^k](s,a)|\\
&=|\mathbb{P}[s_t^k=s|s_0^k]\cdot\pi_b(a|s)\cdot [\mathcal{F}^{\pi_{k+1}}Q^k-Q^k](s,a)-\nu^{\pi_b}(s)\pi_b(a|s)[\mathcal{F}^{\pi_{k+1}}Q^k-Q^k](s,a)|\\
&=|\mathbb{P}[s_t^k=s|s_0^k]-\nu^{\pi_b}(s)|\cdot\pi_b(a|s)\cdot\underbrace{|[\mathcal{F}^{\pi_{k+1}}Q^k-Q^k](s,a)|}_{\leq 1/(1-\gamma)\mbox{\small by Lemma~\ref{lem:Expected-TD-AC-PMD:bounded-value}}}\\
&\leq \frac{1}{1-\gamma}|\mathbb{P}[s_t^k=s|s_0^k]-\nu^{\pi_b}(s)|\\
&\leq \frac{1}{1-\gamma}\underbrace{d_{TV}(\mathbb{P}[s_t^k=s|s_0^k],\nu^{\pi_b}(\cdot))}_{\mbox{\small applying equation~\eqref{eq:Expected-TD-AC-PMD:MDP-ergodicity}}}\\
&\leq \frac{1}{1-\gamma}m_b\kappa_b^t.
\end{align*}
Under the weight setting in \eqref{eq:Expected-TD-AC-PMD:average-weight}, 
\begin{itemize}
    \item When $\vartheta=0$, then $c_t^k=0,t=1,\cdots,B_k-2$ and $c_{B_k-1}^k=1$, and one has
    \begin{align*}
|\mathbb{E}[\bar{\omega}_{k}(s,a)|\mathcal{G}_{k}]|&=|\mathbb{E}[{\omega}^{k}_{B-1}(s,a)|\mathcal{G}_{k}]|\leq \frac{1}{1-\gamma}m_b\kappa_b^{B-1};
\end{align*}
\item When $\vartheta=1$, 
\begin{align*}
|\mathbb{E}[\bar{\omega}_{k}(s,a)|\mathcal{G}_{k}]|=\left|\frac{1}{B}\sum_{t=0}^{B-1}\mathbb{E}[\omega_t^k(s,a)|\mathcal{G}_k]\right|\leq \frac{m_b}{1-\gamma}\left|\frac{1}{B}\sum_{t=0}^{B-1}\kappa_b^t\right|\leq \frac{m_b}{B(1-\gamma)(1-\kappa_b)};
\end{align*}
\item When $\vartheta\in(0,1)$ and $\vartheta\neq \kappa_b$,
\begin{align*}
|\mathbb{E}[\bar{\omega}_{k}(s,a)|\mathcal{G}_{k}]|&=\left|\frac{1}{\sum_{\ell=1}^{B-1}\vartheta^\ell}\sum_{t=0}^{B-1}\vartheta^{B-t-1}\mathbb{E}[\omega_t^k(s,a)|\mathcal{G}_k]\right|\\
&\leq\frac{m_b}{1-\gamma} \frac{1}{\sum_{\ell=0}^{B-1}\vartheta^\ell}\sum_{t=0}^{B-1}\vartheta^{B-t-1}\kappa_b^t\\
&=\frac{m_b}{1-\gamma}\frac{1-\vartheta}{1-\vartheta^{B}}\frac{\vartheta^{B}-\kappa_b^{B}}{\vartheta-\kappa_b}\\
&\leq \begin{cases}\displaystyle
\frac{m_b}{1-\gamma}\frac{\kappa_b^{B}}{\kappa_b-\vartheta},&\mbox{if }\vartheta<\kappa_b\\
\displaystyle\frac{m_b}{1-\gamma}\frac{\vartheta^{B}}{\vartheta-\kappa_b},&\mbox{if }\vartheta>\kappa_b;
\end{cases}
\end{align*}
\item When $\vartheta=\kappa_b$,
\begin{align*}
|\mathbb{E}[\bar{\omega}_{k}(s,a)|\mathcal{G}_{k}]|&=\left|\frac{1}{\sum_{\ell=1}^{B-1}\vartheta^\ell}\sum_{t=0}^{B-1}\vartheta^{B-t-1}\mathbb{E}[\omega_t^k(s,a)|\mathcal{G}_k]\right|\\
&\leq\frac{m_b}{1-\gamma} \frac{1}{\sum_{\ell=0}^{B-1}\vartheta^\ell}\sum_{t=0}^{B_k-1}\vartheta^{B-t-1}\kappa_b^t\\
&=\frac{m_b}{1-\gamma}\frac{1-\vartheta}{1-\vartheta^{B}}\cdot B\cdot\vartheta^{B-1}\\
&\leq \frac{m_b\cdot B\cdot \vartheta^{B-1}}{1-\gamma}.
\end{align*}
\end{itemize}
Finally, one has
\begin{align*}
\sum_{j=1}^k|\mathbb{E}[F_j^k]|&\le 2\alpha\sum_{j=1}^k (1-\alpha(1-\gamma)\tilde{\sigma}_b)^{k-j}\mathbb{E}[\|\mathbb{E}[\bar{\omega}_{j-1}|\mathcal{G}_{j-1}]\|_\infty]\\
&\le 2\alpha\sum_{j=1}^k(1-\alpha(1-\gamma)\tilde{\sigma}_b)^{k-j}\Psi(B_{j-1},\vartheta,m_b,\kappa_b)\\
&\leq \frac{2}{(1-\gamma)\tilde{\sigma}_b}\Psi({B},\vartheta,m_b,\kappa_b).
\end{align*}
\subsection{Proof of Lemma~\ref{lem:Expected-TD-AC-PMD:linear-critic-convergence}}\label{sec:proof-critic-linear}

    We start with applying the three-point-descent lemma (Lemma~\ref{lem:three-point-descent-lemma}) with $p=\tilde\pi_k(\cdot|s^\prime)$ for any $s^\prime\in\calS$,
    \begin{align*}
        &\eta_k \ls\langle \pi_{k+1}(\cdot|s^\prime) - \tilde\pi_k(\cdot|s^\prime), \; Q^k(s^\prime,\cdot) \rs\rangle \geq D^{\pi_{k+1}}_{\pi_k}(s^\prime) + D^{\tilde\pi_k}_{\pi_{k+1}}(s^\prime) - D^{\tilde\pi_k}_{\pi_k}(s^\prime) \geq -D^{\tilde\pi_k}_{\pi_k}(s^\prime), 
        \end{align*}
        yielding that
        \begin{align*}
        \ls\langle \pi_{k+1}(\cdot|s^\prime), \, Q^k(s^\prime,\cdot) \rs\rangle \geq \max_{a\in\calA} \, Q^k(s^\prime,a) - \frac{1}{\eta_k} \| D^{\tilde\pi_k}_{\pi_k} \|_\infty.
    \end{align*}
   Therefore,
    \begin{align*}
        \calF^{\pi_{k+1}} Q^k(s,a) &= r(s,a) + \gamma \, \E_{s^\prime\sim P(\cdot|s,a)}\ls[\ls\langle \pi_{k+1}(\cdot|s^\prime), \, Q^k(s^\prime, \cdot) \rs\rangle\rs] \\
        &\geq r(s,a) + \gamma \, \E_{s^\prime\sim P(\cdot|s,a)} \ls[ \max_{a^\prime\in\calA} \, Q^k(s^\prime, a^\prime) \rs] - \frac{\gamma}{\eta_k} \| D^{\tilde\pi_k}_{\pi_k} \|_\infty \\
        &= \calF Q^k(s,a) - \frac{\gamma}{\eta_k} \| D^{\tilde\pi_k}_{\pi_k} \|_\infty.
    \end{align*}
    According to the critic update formula (equation~\eqref{eq:Expected-TD-AC-PMD:critic-update}), for any $(s,a)\in\Delta(\calS\times\calA)$,
    \begin{align*}
        &\phantom{=\,\,\,}[Q^* - Q^{k+1}](s,a) \\
        &= [1-\alpha\sigma^{\pi_b}(s,a)] [Q^*(s,a) - Q^k(s,a)] + \alpha\sigma^{\pi_b}(s,a) [Q^*(s,a) - \calF^{\pi_{k+1}} Q^k(s,a)] - \alpha \bar\omega_k(s,a) \\
        &\leq [1-\alpha\sigma^{\pi_b}(s,a)] [Q^*(s,a) - Q^k(s,a)] + \alpha\sigma^{\pi_b}(s,a) [\calF Q^*(s,a) - \calF Q^k(s,a) + \gamma \eta_k^{-1} \| D^{\tilde\pi_k}_{\pi_k} \|_\infty] - \alpha \bar\omega_k(s,a),
    \end{align*}
    where we use $\calF Q^* = Q^*$ and apply the lower bound of $\calF^{\pi_{k+1}} Q^k(s,a)$ above. On the other hand,
    \begin{align*}
        &\phantom{=\,\,\,}[Q^* - Q^{k+1}](s,a) \\
        &= [1-\alpha\sigma^{\pi_b}(s,a)] [Q^*(s,a) - Q^k(s,a)] + \alpha\sigma^{\pi_b}(s,a) [Q^*(s,a) - \calF^{\pi_{k+1}} Q^k(s,a)] - \alpha \bar\omega_k(s,a) \\
        &\geq [1-\alpha\sigma^{\pi_b}(s,a)] [Q^*(s,a) - Q^k(s,a)] + \alpha\sigma^{\pi_b}(s,a) [\calF Q^*(s,a) - \calF Q^k(s,a)] - \alpha \bar\omega_k(s,a),
    \end{align*}
    where we use the fact that $\calF Q^k(s,a) \geq \calF^{\pi_{k+1}} Q^k(s,a)$. Combining the upper and lower bounds for $[Q^* - Q^{k+1}](s,a)$ together, one has
    \begin{align*}
        &\phantom{=\,\,\,}\ls| Q^*(s,a) - Q^{k+1}(s,a) \rs| \\
        & \leq [1-\alpha\sigma^{\pi_b}(s,a)] \underbrace{\ls| Q^*(s,a) - Q^k(s,a) \rs|}_{\leq \| Q^* - Q^k \|_\infty} + \alpha\sigma^{\pi_b}(s,a) \underbrace{\ls| \calF Q^*(s,a) - \calF Q^k(s,a) \rs|}_{\leq \| \calF Q^* - \calF Q^k \|_\infty \leq \gamma\| Q^* -Q^k \|_\infty} \\
        &\;\;\;+ \alpha\gamma\eta_k^{-1} \underbrace{\sigma^{\pi_b}(s,a)}_{\leq 1} \| D^{\tilde\pi_k}_{\pi_k} \|_\infty + \alpha \underbrace{|\bar\omega_k(s,a)|}_{\leq \| \bar\omega_k \|_\infty} \\
        &\leq [1-(1-\gamma)\alpha\sigma^{\pi_b}(s,a)] \cdot \ls\| Q^* - Q^k \rs\|_\infty + \frac{\alpha\gamma}{\eta_k} \| D^{\tilde\pi_k}_{\pi_k} \|_\infty + \alpha \| \bar\omega_k \|_\infty \\
        &\leq [1-(1-\gamma)\alpha\tilde\sigma_b] \cdot \ls\| Q^* - Q^k \rs\|_\infty + \frac{\alpha\gamma}{\eta_k} \| D^{\tilde\pi_k}_{\pi_k} \|_\infty + \alpha \| \bar\omega_k \|_\infty,
    \end{align*}
   which completes the proof.
\subsection{Proof of Lemma~\ref{lem:Expected-TD-AC-PMD:error-control-2}}\label{sec:proof-stochastic-error-2}

    Recall that $\calG_k = \{ s_0^k, \, \pi_{k+1}, \, Q^k \}$ and let $o_t^k := (s_t^k, a_t^k, s_{t+1}^k)$. By Lemma~\ref{lem:Expected-TD-AC-PMD:bounded-value}, there holds $| \omega_t^k(s,a) | \leq 2(1-\gamma)^{-1}$ and $| \bar\omega_k(s,a) |\leq 2(1-\gamma)^{-1}$. For any $0 \leq i < j \leq B_k-1$, one has
    \begin{align*}
        \E\ls[ \omega_i^k(s,a) \omega_j^k(s,a) \, \big | \, \calG_k \rs] &= \E\ls[ \omega_i^k(s,a) \cdot \E \ls[ \omega_j^k(s,a) \, \big | \, \calG_k, \,o_i^k \rs] \; \big | \; \calG_k \rs] \\
        &\leq \E \Bigg[ \ls| \omega_t^k(s,a) \rs| \cdot \ls| \E\ls[ \omega_j^k(s,a) \, \big | \, \calG_k, \, o_i^k \rs] \rs| \; \Bigg | \; \calG_k  \Bigg] \\
        &\leq \frac{2}{1-\gamma} \E \Bigg[ \ls| \E\ls[ \omega_j^k(s,a) \, \big | \, \calG_k, \, o_i^k \rs] \rs| \; \Bigg | \; \calG_k  \Bigg].
    \end{align*}
    By  the same computation as in equation~\eqref{eq:Expected-TD-AC-PMD:delta-expectation}, we have
    \begin{align*}
        \E\ls[ \delta_j^k(s,a) \, \big | \, \calG_k,\, o_i^k  \rs] &= \P\ls[ (s^k_j, a^k_j)=(s,a)  \; \big | \; \calG_k, \, o_i^k \rs] \cdot [\calF^{\pi_{k+1}} Q^k(s,a) - Q^k(s,a)] \\
        &= \P\ls[ (s^k_j, a^k_j)=(s,a)  \; \big | \;  s^k_{i+1} \rs] \cdot [\calF^{\pi_{k+1}} Q^k(s,a) - Q^k(s,a)] \\
        &= \P \ls[ s_j^k = s \, | \, s_{i+1}^k \rs] \cdot \pi_b(a|s) \cdot [\calF^{\pi_{k+1}} Q^k(s,a) - Q^k(s,a)].
    \end{align*}
    Therefore,
    \begin{align*}
        &\phantom{=\,\,\,}\ls| \E\ls[ \omega_j^k(s,a) \, \big | \, \calG_k, \, o_i^k \rs] \rs| \\
        &= \ls| \E\ls[ \delta_j^k(s,a) \, \big | \, \calG_k, \, o_i^k \rs] - \ls[ \calF^{\pi_{k+1}}_{\pi_b} Q^k - Q^k \rs](s,a)  \rs| \\
        &= \ls| \P \ls[ s_j^k = s \, | \, s_{i+1}^k \rs] \cdot \pi_b(a|s) \cdot [\calF^{\pi_{k+1}} Q^k(s,a) - Q^k(s,a)] - \sigma^{\pi_b}(s,a) [\calF^{\pi_{k+1}} Q^k - Q^k](s,a) \rs| \\
        &= \ls| \P\ls[ s^k_j=s  \; \big | \; s^k_{i+1} \rs] - \nu^{\pi_b}(s) \rs | \cdot |\pi_b(a|s)| \cdot \ls | [\calF^{\pi_{k+1}} Q^k(s,a) - Q^k(s,a)] \rs| \\
        &\leq \underbrace{d_{TV}\ls( \nu^{\pi_b}(\cdot) , \; \P[s_j^k=\cdot \, | \, s^k_{i+1}]\rs)}_{\mbox{\small Applying equation~\eqref{eq:Expected-TD-AC-PMD:MDP-ergodicity}}} \cdot \underbrace{|\pi_b(a|s)|}_{\leq 1} \cdot \underbrace{\ls | [\calF^{\pi_{k+1}} Q^k(s,a) - Q^k(s,a)] \rs|}_{\leq (1-\gamma)^{-1} \mbox{\small by Lemma~\ref{lem:Expected-TD-AC-PMD:bounded-value}}}\\
        &\leq \frac{1}{1-\gamma} m_b \cdot \kappa_b^{j-i-1}.
    \end{align*}
    As a result,
    \begin{align*}
        \forall\, j > i: \quad \E\ls[ \omega_i^k(s,a) \omega_j^k(s,a) \; \big | \; \calG_k \rs] \leq \frac{2}{(1-\gamma)^2} m_b \kappa_b^{j-i-1}.
    \end{align*}
  It follows that
    \begin{align*}
        \E\ls[ \bar\omega_k(s,a)^2 \rs] &= \E \ls[ \E\ls[ \bar\omega_k(s,a)^2 \, | \, \calG_k \rs] \rs] \\
        &= \frac{1}{B_k^2} \E\ls[ \sum_{t=0}^{B_k-1} \E\ls[ \underbrace{\omega_t^k(s,a)^2}_{\leq 4(1-\gamma)^{-2}} \, \Big | \, \calG_k \rs] + 2\sum_{i<j} \E\ls[ \omega_i^k(s,a) \omega_j^k(s,a) \, | \, \calG_k \rs] \rs] \\
        &\leq \frac{1}{B_k^2} \ls[ \frac{4B_k}{(1-\gamma)^2} + 2B_k \cdot \frac{2}{(1-\gamma)^2} \frac{m_b}{1-\kappa_b} \rs] \\
        &\leq \frac{1}{B_k} \cdot \frac{4}{(1-\gamma)^2} \ls[ 1+ \frac{m_b}{1-\kappa_b} \rs].
    \end{align*}
   Thus there holds
    \begin{align*}
        \E\ls[ \| \bar\omega_k \|_\infty \rs] \leq \E\ls[ \| \omega_k \|_2 \rs] \leq \sqrt{\E\ls[ \| \omega_k \|_2^2 \rs]} = \sqrt{\E\ls[ \sum_{(s,a)\in\calS\times\calA} \bar\omega_k(s,a)^2 \rs]}.
    \end{align*}
\subsection{Proof of Lemma~\ref{lem:approximated-Ejk-bound}}
\label{sec:pf:lem:approximated-Ejk-bound}
We first bound $\| A_j - A_{j-1} \|_\infty$. To this end, for any $\pi, \tilde \pi \in \Pi$, let $\{ s_t \}_{t\geq0}$ and $\{ \tilde s_t \}_{t\geq 0}$ be the Markov chains with the transition kernels of $P^{\pi_b\circ\pi}_{\scriptscriptstyle\calS}$, $P^{\pi_b\circ\tilde\pi}_{\scriptscriptstyle\calS}$, respectively. By Assumption~\ref{ass:ergodicity-2} and Proposition~\ref{pro:MDP-ergodicity-2} we know that both of these two chains are ergodic and geometrically mixing with parameters $m > 0$ and $0 < \kappa < 1$. 
Notice that
\begin{align*}
    \ls\| P^{\pi_b\circ\pi}_{\scriptscriptstyle\calS} - P^{\pi_b\circ\tilde\pi}_{\scriptscriptstyle\calS} \rs\|_\infty &= \ls\| P^{\pi_b}_{\scriptscriptstyle\calS}P^{\pi}_{\scriptscriptstyle\calS} - P^{\pi_b}_{\scriptscriptstyle\calS}P^{\tilde\pi}_{\scriptscriptstyle\calS} \rs\|_\infty \\
    &\leq \ls\| P^{\pi}_{\scriptscriptstyle\calS} - P^{\tilde\pi}_{\scriptscriptstyle\calS} \rs\|_\infty \\
    &= \max_{s\in\calS} \, \sum_{s^\prime \in \calS} \ls| P^\pi_{\scriptscriptstyle\calS}(s,s^\prime) - P^{\tilde\pi}_{\scriptscriptstyle\calS}(s,s^\prime) \rs| \\
    &= \max_{s\in\calS} \, \sum_{s^\prime \in \calS} \ls| \sum_{a\in\calA} \ls(\pi(a|s) - \tilde\pi(a|s)\rs) P(s^\prime | s,a) \rs| \leq |\calA| \cdot \| \pi - \tilde\pi \|_\infty.
\end{align*}
The application of Lemma~\ref{lem:perturbated-chain-stationary-diff} yields
\begin{align*}
    \ls\| \Sigma^{\pi_b\circ\pi} - \Sigma^{\pi_b\circ\tilde\pi} \rs\|_\infty &\leq  \max_{s\in\calS} \ls| \nu^{\pi_b\circ\pi}(s) - \nu^{\pi_b\circ\tilde\pi}(s) \rs| \\
    &\leq d_{TV}(\nu^{\pi_b\circ\pi}(\cdot), \; \nu^{\pi_b\circ\tilde\pi}(\cdot))\\
    &\leq \ls( \ls\lceil \log_\kappa m^{-1} \rs\rceil + \frac{1}{1-\kappa} \rs) \ls\| P^{\pi_b\circ\pi}_{\scriptscriptstyle\calS} - P^{\pi_b\circ\tilde\pi}_{\scriptscriptstyle\calS} \rs\|_\infty \\
    &\leq |\calA|\ls( \ls\lceil \log_\kappa m^{-1} \rs\rceil + \frac{1}{1-\kappa} \rs) \ls\| \pi - \tilde\pi \rs\|_\infty = L\cdot \ls\| \pi - \tilde\pi \rs\|_\infty.
\end{align*}
By setting $\pi = \pi_j$ and $\tilde\pi = \pi_{j-1}$ we get
\begin{align*}
    \| A_j - A_{j-1} \|_\infty &= \alpha \ls\| \Sigma^{\pi_b\circ\pi_j}(I - \gamma P^{\pi_j}_{\scriptscriptstyle\calS\times\calA}) - \Sigma^{\pi_b\circ\pi_{j-1}}(I - \gamma P^{\pi_{j-1}}_{\scriptscriptstyle\calS\times\calA}) \rs\|_\infty \\
    &= \alpha \ls\| (\Sigma^{\pi_b\circ\pi_j} - \Sigma^{\pi_b\circ\pi_{j-1}})(I - \gamma P^{\pi_j}_{\scriptscriptstyle\calS\times\calA}) + \Sigma^{\pi_b\circ\pi_{j-1}}[(I - \gamma P^{\pi_{j}}_{\scriptscriptstyle\calS\times\calA})-(I - \gamma P^{\pi_{j-1}}_{\scriptscriptstyle\calS\times\calA})] \rs\|_\infty \\
    &\leq \alpha \underbrace{\ls\|\Sigma^{\pi_b\circ\pi_j} - \Sigma^{\pi_b\circ\pi_{j-1}}\rs\|_\infty}_{\leq L\cdot \| \pi_j - \pi_{j-1} \|_\infty} \cdot \underbrace{\ls\| I - \gamma P^{\pi_{j}} \rs\|_\infty}_{\leq 1 + \gamma} + \alpha \gamma \underbrace{\ls\| \Sigma^{\pi_b\circ\pi_{j-1}} \rs\|_\infty}_{\leq 1} \cdot \underbrace{\ls\| P^{\pi_{j}}_{\scriptscriptstyle\calS\times\calA} - P^{\pi_{j-1}}_{\scriptscriptstyle\calS\times\calA} \rs\|_\infty}_{\leq |\calA| \cdot \ls\| \pi_j - \pi_{j-1} \rs\|_\infty} \\
    &\leq \alpha ((1+\gamma) L + \gamma |\calA|) \cdot \underbrace{\ls\| \pi_j - \pi_{j-1} \rs\|_\infty}_{\mbox{\small Applying Lemma~\ref{lem:policy-shift} and Lemma~\ref{lem:Approximated-TD-AC-PMD:bounded-value}}}  \\
    &\leq \frac{\alpha |\calA| \cdot \ls[ (1+\gamma) L + \gamma |\calA| \rs]}{\lambda (1-\gamma)} \eta.
\end{align*}
Thus
\begin{align*}
    \ls\| (A_j)^{k-j+1} - (A_{j-1})^{k-j+1} \rs\|_\infty &\leq [1-(1-\gamma)\alpha\tilde\sigma] \cdot \ls\| (A_j)^{k-j} - (A_{j-1})^{k-j} \rs\|_\infty + [1-(1-\gamma)\alpha\tilde\sigma]^{k-j} \cdot \ls\| A_j - A_{j-1} \rs\|_\infty \\
    &\leq \frac{\alpha |\calA| \cdot \ls[ (1+\gamma) L + \gamma |\calA| \rs]}{\lambda (1-\gamma)} \eta \cdot (k-j+1) \cdot [1-(1-\gamma)\alpha\tilde\sigma]^{k-j}.
\end{align*}
It follows that
\begin{align*}
    \sum_{j=1}^k \ls| E_j^{(k)} \rs| &= \sum_{j=1}^k \ls| \inner{\ls[ (A_{j})^{k-j+1} - (A_{j-1})^{k-j+1} \rs]^\top J_s [\pi^* - \pi_{j-1}]}{[Q^{\pi_{j-1}} - Q^{j-1}]} \rs| \\
    &\leq \frac{2}{1-\gamma} \sum_{j=1}^k \ls\| (A_j)^{k-j+1} - (A_{j-1})^{k-j+1} \rs\|_\infty \\
    &\leq \frac{2|\calA| \cdot \ls[ (1+\gamma) L + \gamma |\calA| \rs]}{\alpha\lambda\tilde\sigma^2 (1-\gamma)^4} \eta \leq \frac{2|\calA|^2 \cdot \ls[ 2 L +  1 \rs]}{\alpha\lambda\tilde\sigma^2 (1-\gamma)^4} \eta.
\end{align*}

\subsection{Proof of Lemma~\ref{lem:approximated-Fjk-bound}}
\label{sec:pf:lem:approximated-Fjk-bound}
Recall that $\calG_k = \{ Q^k, \, \pi_{k+1}, \, s_0^k \}$. One has
\begin{align*}
    \E\ls[ \delta_t^k(s,a) \, | \, \calG_k \rs] &= \E\ls[ \mathds{1}_{(s_t^k, a_t^k) = (s,a)} \cdot \ls[ r_t^k + \gamma \, Q^k(s_t^{\prime\,k}, a_t^{\prime\,k}) - Q^k(s_t^k, a_t^k) \rs] \; \bigg | \; \calG_k \rs] \\
    &= \E\ls[ \mathds{1}_{(s_t^k, a_t^k) = (s,a)} \cdot \E \ls[ r_t^k + \gamma \, Q^k(s_t^{\prime\,k}, a_t^{\prime\,k}) - Q^k(s_t^k, a_t^k) \; \big | \; \calG_k, \, s_t^k, a_t^k \rs] \; \bigg | \; \calG_k \rs] \\
    &= \E \ls[ \mathds{1}_{(s_t^k, a_t^k)=(s,a)} \cdot \ls[ \calF^{\pi_{k+1}} Q^k(s_t^k, a_t^k) - Q^k(s_t^k, a_t^k) \rs] \; \bigg | \; \calG_k \rs] \\
    &= \P\ls[ (s_t^k, a_t^k)=(s,a) \; \big | \; \calG_k \rs] \cdot \ls[ \calF^{\pi_{k+1}} Q^k(s,a) - Q^k(s,a) \rs] \\
    &= \P\ls[ s_t^k = s\, | \, \pi_{k+1}, \, s_0^k \rs] \cdot \pi_b(a|s) \cdot \ls[ \calF^{\pi_{k+1}} Q^k(s,a) - Q^k(s,a) \rs]. \numberthis \label{eq:Approximated-TD-AC-PMD:delta-expectation}
\end{align*}
It follows that
\begin{align*}
    \ls| \E\ls[ \omega_t^k(s,a) \, | \, \calG_k \rs] \rs| &= \ls| \P\ls[ s_t^k=s \, | \, \pi_{k+1}, \, s_0^k \rs] - \nu^{\pi_b\circ\pi_{k+1}}(s) \rs| \cdot \pi_b(a|s) \cdot \ls| \calF^{\pi_{k+1}} Q^k(s,a) - Q^k(s,a) \rs| \\
    &\leq \frac{1}{1-\gamma} \underbrace{d_{TV}(\P\ls[ s_t^k=\cdot \, | \, s_0^k \rs], \; \nu^{\pi_b\circ\pi_{k+1}}(\cdot ) \; \big | \; \pi_{k+1})}_{\mbox{\small Applying equation~\eqref{eq:Approximated-TD-AC-PMD:MDP-ergodicity}}} \\
    &\leq \frac{1}{1-\gamma} m \kappa^t.
\end{align*}
{The rest of the proof is the same with that for Lemma~\ref{lem:expected-Fjk-bounds}.}

\section{Conclusion}
\label{sec:conclusion}
In this paper, we investigate the sample complexity of policy mirror descent with temporal difference learning under online Markov sampling. Two algorithms including Expected TD-PMD (off-policy) and Approximate TD-PMD (mixed-policy) have been considered, and  a comprehensive sample complexity analysis have been conducted. Under the constant policy update step size, both algorithms can find the average-time $\varepsilon$-optimal solution with the $\tilde O(\varepsilon^{-2})$ sample complexity, which is further improved to   the last-iterate $\varepsilon$-optimal solution  with $O(\varepsilon^{-2})$ sample complexity when utilizing adaptive step sizes. 

For future work, it is interesting to extend our analysis to the RL problem with entropy regularization. In addition, it is also possible to consider the more general setting with policy or critic function approximation. Any progress towards these directions will be reported separately. 
\section*{Acknowledgment}
The authors would like to thank Gen Li and Luo Luo for fruitful discussion regarding Q-learning and stochastic gradient methods, respectively.
\bibliographystyle{plainnat}
\bibliography{refs}

\end{document}